\documentclass[11pt,reqno]{amsart}
\usepackage{amsmath, amssymb, amsthm, latexsym, enumerate, graphicx,stmaryrd}
\usepackage{ifpdf}
\ifpdf
\usepackage[pdftex]{hyperref}
\fi

\usepackage{bbm}
\usepackage[all]{xy}
\usepackage{pdfsync}
\usepackage{todonotes}
\usepackage{tikz}
\usetikzlibrary{arrows,decorations,decorations.pathmorphing,positioning,patterns}

\usepackage[vcentermath]{youngtab}

\newcommand{\bigslant}[2]{{\raisebox{.2em}{$#1$}\left/\raisebox{-.2em}{$#2$}\right.}}
\allowdisplaybreaks

 \newlength{\baseunit}               
 \newcount{\numlines}                
\newtheorem{theorem}{Theorem}[section]

\newtheorem{lemma}[theorem]{Lemma}
\newtheorem{prop}[theorem]{Proposition}
\newtheorem{corollary}[theorem]{Corollary}
\newtheorem{conjecture}[theorem]{Conjecture}

\theoremstyle{definition}
\newtheorem{definition}[theorem]{Definition}

\newtheorem{notation}[theorem]{Notation}
\newtheorem{ex}[theorem]{Example}

\newtheorem{remark}[theorem]{Remark}

\newcommand{\cB}{{\mathcal B}}
\newcommand{\cC}{{\mathcal C}}

\newcommand{\cE}{{\mathcal E}}
\newcommand{\cF}{{\mathcal F}}

\newcommand{\cI}{{\mathcal I}}

\newcommand{\cO}{{\mathcal O}}
\newcommand{\cP}{{\mathcal P}}
\newcommand{\cR}{{\mathcal R}}

\newcommand{\cS}{{\mathcal S}}

\newcommand{\cV}{{\mathcal V}}

\newcommand{\cY}{{\mathcal Y}}

\newcommand{\ii}{{\bf i}}

\def\down{\vee}
\def\up{\wedge}

\newcommand{\mg}{\mathfrak{g}}
\newcommand{\ms}{\mathfrak{s}}
\newcommand{\gl}{\mathfrak{gl}}
\newcommand{\mh}{\mathfrak{h}}

\newcommand{\mN}{\mathbb{N}}
\newcommand{\mB}{\mathbb{B}}
\newcommand{\mK}{\mathbb{K}}

\newcommand{\mR}{\mathbb{R}}

\newcommand{\mC}{\mathbb{C}}

\newcommand{\mP}{\mathbb{P}}
\newcommand{\mZ}{\mathbb{Z}}

\newcommand{\la}{\lambda}

\newcommand{\oneone}{\stackrel{1:1}\leftrightarrow}

\newcommand{\HOM}{\operatorname{Hom}}

\newcommand{\END}{\operatorname{End}}

\newcommand{\p}{\mathfrak{p}}

\newcommand{\op}{\operatorname}

  \DeclareMathOperator{\Ind}{Ind} 
\DeclareMathOperator{\End}{End}

\newcommand{\CupConnect}{\text{\----}}
\newcommand{\DCupConnect}{\text{\----}\hspace{-10pt}\bullet\hspace{2pt}}
\newcommand{\RayConnect}{\text{\----}\hspace{-5pt}\shortmid}
\newcommand{\DRayConnect}{\text{\----}\hspace{-7pt}\bullet\hspace{-3pt}\shortmid}
\DeclareMathOperator{\pt}{p}
\newcommand{\renc}{\renewcommand}

\renewcommand{\mp}{\mathfrak{p}}

\renc{\thefigure}{\thesection.\arabic{figure}}

\begin{document}
\title[2-row Springer fibres in Type D]{2-row Springer fibres and Khovanov diagram algebras for type $\rm D$}
\author{Michael Ehrig}
\author{Catharina Stroppel}

\address{Department of Mathematics, University of Bonn, 53115 Bonn, Germany}
\email{mehrig@math.uni-bonn.de}
\address{Department of Mathematics, University of Bonn, 53115 Bonn, Germany}
\email{stroppel@math.uni-bonn.de}

\begin{abstract}
We study in detail two row Springer fibres of even orthogonal type from an algebraic as well as topological point of view. We show that the irreducible components and their pairwise intersections are iterated $\mP^1$-bundles. Using results of Kumar and Procesi we compute the cohomology ring with its action of the Weyl group. The main tool is a type $\rm D$ diagram calculus labelling the irreducible components in a convenient way which relates to a diagrammatical algebra describing the category of perverse sheaves on isotropic Grassmannians based on work of Braden. The diagram calculus generalizes Khovanov's arc algebra to the type $\rm D$ setting and should be seen as setting the framework for generalizing well-known connections of these algebras in type $\rm A$ to other types.
\end{abstract}
\maketitle

\tableofcontents

\section{Introduction}
\renc{\thetheorem}{\Alph{theorem}}

This paper is part of a series of four quite different papers, \cite{ES}, \cite{ESBrauer}, \cite{ESnew} dealing with type $\rm D$ generalizations of Khovanov's arc algebra. We develop in detail the geometric background of this algebra using the geometry of topological and algebraic Springer fibres and explain connections to category $\cO$ for the orthogonal Lie algebra and to categories of finite dimensional representations of the associated $\mathcal{W}$-algebras.

The type $\rm D$ Khovanov algebra is defined in \cite{ES} using Braden's description of the category of perverse sheaves on isotropic Grassmannians, hence it gives an elementary description of parabolic category $\cO$ of type $({\rm D}_k,{\rm A}_{k-1})$ with all its nice properties like Koszulness, quasi-hereditary, cellularity, etc. Based on \cite{CdV} we introduce in \cite{ESBrauer} suitable idempotent truncations of the type ${\rm D}$ Khovanov algebra and show in \cite{ESnew} that these algebras are indeed graded analogues of Brauer algebras $\op{Br}_d(\delta)$ for arbitrary integral parameter $\delta$. These idempotent truncations can also be viewed as a kind of limit, in which case this is very much in the spirit of a similar result obtained in \cite{BS5} for the walled Brauer algebra using a limit version of generalized Khovanov algebras of type $\rm A$. Although the flavour of results is very similar to the known type $\rm A$ results, the techniques and difficulties are very different, including quite a few surprising new phenomena. In particular, the underlying geometry of Springer fibres is much more subtle and has not been studied so far. Let us now describe the content of this paper in more detail.

Let $V=\mC^{2k}$ be an even dimensional vector space with a non-degenerate symmetric bilinear form $\beta$. Let $G=O(2k,\mC)$ be the group of automorphisms preserving $\beta$ and $\mg=\mathfrak{o}(2k,\mC)$ the corresponding orthogonal Lie algebra. By a result of John Williamson, \cite{WilliamsonJ}, the $G$-conjugacy classes of nilpotent elements in $\mg$ are classified by their Jordan normal form, i.e. by partitions of $2k$ which are admissible in the sense that even parts appear with even multiplicity.

Let $\cF$ be the flag variety for $G$ of full isotropic flags with respect to $\beta$,
\begin{eqnarray}
\{F_0\subset F_1\subset \cdots \subset F_{2k}=\mC^{2k}\mid \op{dim}(F_i)=i, F_i^\perp=F_{2k-i}\}.
\end{eqnarray}

Given a nilpotent element $N\in\mg$ consider the Springer fibre $\cF_N$ of all flags $F_\bullet\in\cF$ fixed by $N$, i.e. satisfying $N(F_i)\subseteq F_{i-1}$ for all $i>0$. This is the fibre of the famous Springer resolution of the nilpotent cone $\mathcal{N}\subset\mg$ over $N$. The Springer fibres and their corresponding irreducible components are in general poorly understood. Even in type $\rm A$ the components are not smooth in general; the first case where singular components appear is the two columns case, \cite{Fressesing}. Recently, Fresse and Melnikov found a parametrisation of all smooth irreducible components in type $\rm A$ in terms of standard tableaux, \cite{FresseMelnikov}. This includes the quite easy cases of two row partitions, were Fung showed already much earlier, \cite{Fung}, that irreducible components are iterated $\mP(\mC)^1$-bundles and therefore smooth.

In this paper we are interested in generalizing the latter result to the orthogonal case and moreover to construct an easy topological model for these Springer fibres, analogous to the type $\rm A$ model of Khovanov, \cite{Khovanov_Springer}. Our result in Section~\ref{Section6} is the following:

\begin{theorem}
\label{maintheoremintro}
Let $N\in\mg$ be a nilpotent element of Jordan type $(\la_1,\la_2)$, where $\la_1\geq \la_2$. Every irreducible component $Y$ of $\mathcal{F}_N$ is an iterated fibre bundle of base type $(\mP^1(\mC), \ldots,\mP^1(\mC))$ where each $\mP^1(\mC)$ corresponds canonically to a cup in the cup diagram associated with the signed admissible standard tableau labelling $Y$. In particular $\mathcal{F}_N$ is equidimensional of dimension $\la_2$.
\end{theorem}

Originally, irreducible components in the orthogonal case were classified by Spaltenstein, \cite{SpaltensteinLNM}, and van Leeuwen, \cite{vanLeeuwen}, using (signed) domino tableaux. Our results rely on a new parametrisation in terms of certain decorated cup diagrams on $k$ vertices.

Fixing the points $P=\{(1,0),(2,0),\ldots,(k,0)\}$ and the rectangle $R=\{(x,y)\mid 0\leq x\leq k+1,-2\leq y\leq 0\}$ in the plane, an \emph{undecorated cup diagram}  is the isotopy class of a diagram consisting of $k$ non-intersecting lines in $R$, each of them connecting two distinct vertices in $P$ or one vertex in $P$ with a point on the bottom face of $R$. A (decorated) cup diagram, is a cup diagram with possibly "$\bullet$"'s (called {\it dots}) on the lines such that every dot is accessible from the left side of the rectangle,
meaning for each dot there exists a line in the rectangle connecting the
dot with the left face of $R$ not intersecting the diagram. We
allow at most one dot per line; here are some examples:
\begin{eqnarray}
\label{introcup}
\begin{tikzpicture}[thick]
\node at (-.5,0) {i)};

\draw (0,0) node[above]{$1$} .. controls +(0,-.5) and +(0,-.5) .. +(.5,0) node[above]{$2$};
\fill (0.25,-.36) circle(2.5pt);
\draw (1,0) node[above]{$3$} .. controls +(0,-.5) and +(0,-.5) .. +(.5,0) node[above]{$4$};
\fill (1.25,-.36) circle(2.5pt);

\begin{scope}[xshift=4cm]
\node at (-.5,0) {ii)};

\draw (0,0) node[above]{$1$} .. controls +(0,-1) and +(0,-1) .. +(1.5,0) node[above]{$4$};
\fill (0.75,-.74) circle(2.5pt);
\draw (0.5,0) node[above]{$2$} .. controls +(0,-.5) and +(0,-.5) .. +(.5,0) node[above]{$3$};
\end{scope}

\begin{scope}[xshift=8cm]
\node at (-.5,0) {iii)};

\draw (0,0) node[above]{$1$} .. controls +(0,-.5) and +(0,-.5) .. +(.5,0) node[above]{$2$};
\draw (1,0) node[above]{$3$} .. controls +(0,-.5) and +(0,-.5) .. +(.5,0) node[above]{$4$};
\end{scope}
\end{tikzpicture}
\end{eqnarray}
The set $\cC_k$ of cup diagrams on $k$ vertices describes combinatorially a cell partition of $\cF_N$ where the set of the diagrams $\mB_k$ with the maximal possible number of cups correspond to the irreducible components when taking the closure. We show in Section~\ref{Section7} that the $\cF_N$ for $N$ of Jordan type $(k,k)$ have a filtration by subvarieties isomorphic to the Springer fibres of type $(2k-j,j)$ corresponding to cup diagrams in $\cC_k$ with $j$ cups. There is also a natural action of the Weyl group of type ${\rm D}_k$ on $\cC_k$ in terms of type $\rm D$-Temperley-Lieb diagrams, \cite{LS}. This action preserves the parity of dots on the diagrams. Moreover $\cC_k$ labels a basis of the top degree part of the cohomology of the Springer fibre, see Theorem \ref{H} and Lemma \ref{decomp_disjoint}. In fact this illustrates nicely the Springer representation on the top degree part and on the whole cohomology, see Remark \ref{rem:weyl_action}. It is the nontrivial component group which relates the subsets $\cC_k^{\rm even}$ and $\cC_k^{\rm odd}$ of diagrams with fixed parity; $\mZ_2$ acts by adding or removing a dot on the line attached to vertex $1$. Note that some sort of decorated diagrams also appear in \cite{RussellTymoczko} but the dots have a very different meaning and the diagrams are used to describe the whole cohomology ring. We expect that an analogue of such a description also exists for our setup, but we do not pursue this direction here.

The cup diagrams should be seen as "foldings" or fixed point sets of an involution on the set of type $\rm A$ cup diagrams of doubled size, see \cite{LS}. This folding on the level of Springer fibres corresponds to a folding on the level of Nakajima quiver varieties recently studied in \cite{HL}.

The cup diagrams are the only ingredients to construct the topological Springer fibre $\widetilde{S}$ following the main idea of \cite{Khovanov_Springer}. For fixed $k$ we consider the $k$-fold product of $2$-spheres, $(\mathbb{S}^2)^k$ and assign to each $a\in\mB_k$ the subspace $S_a$ of all tuples $(x_1,x_2,\ldots, x_k)\in(\mathbb{S}^2)^k$ with the condition $x_i=x_j$ (resp. $x_i=-x_j$) if there is an undotted (resp. dotted) cup connecting vertices $i$ and $j$, and $x_i$ equals the north or south pole if there is a vertical ray at $i$, depending if the ray is dotted or not, see \eqref{Sa}. The topological Springer fibre is then simply the union
\begin{eqnarray*}
\widetilde{S} &=& \bigcup_{a\in\mB_k} S_a\subset (\mathbb{S}^2)^k.
\end{eqnarray*}
We prove that the cohomology rings of the topological and algebraic Springer fibres agree and are described as follows:

 \begin{theorem}
 \label{Hintro}
 Let $N\in\mg$ be nilpotent of Jordan type $(k,k)$. Then we have isomorphisms of graded rings
 \begin{eqnarray*}
 H(\widetilde{S})\;\cong\; H(\mathcal{F}_N)
  \;\cong\;\bigoplus_{r=1}^2
 \bigslant{\mC[x_1,\ldots,x_k]}{\left\langle x_i^2, y_I\;
  \begin{array}{|c}
  1 \leq i \leq k,\\
  I\subset  \{1,\ldots,k\}, |I|=\frac{k+\epsilon}{2}
  \end{array}
  \right\rangle},
 \end{eqnarray*}
where $\epsilon=0$ if $k$ is even and $\epsilon=1$ if $k$ is $odd$ and
\begin{eqnarray*}
y_I&=&
\begin{cases}
x_I - x_{\{1,\ldots,k\} \setminus I}&\text{if $k$ is even,}\\
x_I &\text{if $k$ is odd,}
\end{cases}
\end{eqnarray*}
using the abbreviation $x_I = \prod_{i \in I} x_i$ for $I \subset \{1,\ldots,k\}$.
\end{theorem}

The two summands appear due to the natural decomposition of both the topological as well as the algebraic Springer fibre into an "even" and "odd" component. The proof of the second isomorphism is based on the Kumar-Procesi description of equivariant cohomology of Springer fibres, \cite{KumarProcesi}. This result requires $N$ to be of standard Levi type and also the surjectivity of the canonical map $H(\cF)\rightarrow H(\cF_N)$ which fails in general. Using the filtration mentioned above we verify this assumption in our special case (Proposition \ref{filtration}). The proof of the first isomorphism follows very closely the arguments in \cite{Khovanov_Springer}, although most of the small ingredients have to be reproved in a slightly different way. As in type $\rm A$, and pointed out there by Khovanov, the irreducible components of the topological Springer fibre are always trivial fibre bundles, whereas algebraically they are not. The non-triviality is nicely encoded in the degree of nestedness in the cup diagram and should be compared with the theory of embedded cobordisms and TQFT developed in \cite{SW}.
\begin{conjecture}
\label{conj}
The topological and algebraic Springer fibres are homeomorphic in case $N$ is of Jordan type $(k,k)$.
\end{conjecture}
In the general two row case such a homeomorphism cannot exist since for instance  the dimensions of pairwise intersections of components differ, see Section \ref{explicitexample}. For these cases a different construction of a topological Springer fibre is required.

The case of Jordan type $(k,k)$ is however particularly nice. We show that the pairwise intersections $Y_1\cap Y_2$ of irreducible components in $\cF_N$ are again iterated $\mP^1(\mC)$-bundles of dimension equal the number of possible orientations $a(Y_1)^*\la a(Y_2)$ of the circle diagram obtained by putting the cup diagram $a(Y_1)$ assigned to $Y_1$ upside down on top of $a(Y_2)$. Generalizing Khovanov's original construction \cite{KhovanovJones}, see also \cite{BS1}, we consider the vector space $\mK_k$ of all such oriented circle diagrams and establish the following

\begin{theorem} \label{boringtheorem}
Let $N\in\mg$ be nilpotent of Jordan type $(k,k)$.  There is an isomorphism of vector spaces
\begin{eqnarray}
\label{vectiso}
\mathbb{K}_k&\cong&\bigoplus_{(Y_1,Y_2)\in\op{Irr}(\cF_N)\times\op{Irr}(\cF_N)}H(Y_1\cap Y_2).
\end{eqnarray}
\end{theorem}
This vector space can be equipped with an algebra structure using a geometric convolution product construction, following the arguments in \cite{SW}.

We finally explain partly conjectural connections with $W$-algebras and category $\cO$. Let $e\in\mg$ be nilpotent and consider the associated finite $W$-algebra $\mathcal{W}(e):=\mathcal{U}(\mg,e)$ as introduced originally by Premet, \cite{Premet}, see \cite{LosevICM} for an overview. In general it is a hard problem to describe and understand the categories $\mathcal{F}in(e)$ of finite dimensional $\mathcal{W}(e)$-modules or its subcategory $\mathcal{F}in_0(e)$ of polynomial representations with generalized trivial central character, \cite{LosevGoldie}, \cite{LO}. However, the original definition of $\mathcal{W}(e)$ and its realisation as deformation quantization of Slodowy slices, indicates a strong relationship between the representation of $\mathcal{W}(e)$ and the geometry of the Springer fibre. Our setup is related to the case where the nilpotent has Jordan type $2^k$. This is an extremal case of a rectangular Jordan shape, were a complete purely combinatorial classification of finite dimensional $\mathcal{W}(e)$-modules is available, \cite{Brownrectangular}, \cite{BrownGoodwin}. In the last section we give an easy bijection between Brown's parameter set of simple objects and our set $\mathbb{B}_k$ of decorated cup diagrams labelling the irreducible components of $\cF_N$, where $N$ is of Jordan type $(k,k)$. Built on this we conjecture the following:

\begin{conjecture}
\label{bigconj}
Let $k\geq 4$ even and $e, N\in\mg$ nilpotent of Jordan type $2^k$ and $(k,k)$ respectively. Let $P$ be a minimal projective generator of $\mathcal{F}in_0(e)$.
Under the identification \eqref{vectiso} there exists an isomorphism of vector spaces
\begin{eqnarray}
\label{conjiso}
\mathbb{K}_k=\bigoplus_{(Y_1,Y_2)\in\op{Irr}(\cF_N)\times\op{Irr}(\cF_N)}H(Y_1\cap Y_2) &\cong&\END_{\mathcal{W}(e)}(P),
\end{eqnarray}
even of algebras with the diagrammatical multiplication in $\mathbb{K}_k$.
  \end{conjecture}
The conjecture implies in particular that with the appropriate grading shifts, $\bigoplus_{(Y_1,Y_2)\in\op{Irr}(\cF_N)\times\op{Irr}(\cF_N)}H(Y_1\cap Y_2)\langle d(Y_1,Y_2)\rangle$ induces a positive grading on $\END_{\mathcal{W}(e)}(P)$ which is the shadow of a Koszul grading. More precisely, by \cite{ES}, the conjecture directly implies
\begin{corollary}
Let $\mp,\mp'\subset\mg$ be the two maximal parabolic subalgebras of type ${\rm A}_{k-1}$. Let $\cO_0^\mp(D_k)\oplus \cO_0^{\mp'}(D_k)$ be the direct sum of the principal blocks of the corresponding parabolic category $\cO$ for $\mg$. Let $P_{\rm inj}=P^\mp_{\rm inj}\oplus P^{\mp'}_{\rm inj}$ be a minimal projective-injective generator. Then there is an isomorphism of algebras
\begin{eqnarray}
\End_\mg(P_{\rm inj}) & \cong & \END_{\mathcal{W}(e)}(P).
\end{eqnarray}
in particular, $\mathcal{F}in_0(e)$ is a quotient category of $\cO_0^\mp(D_k)\oplus \cO_0^{\mp'}(D_k)$. The latter is in fact a quasi-hereditary cover in the sense of Rouquier, \cite{Rouquier}.
\end{corollary}

Note that this Corollary is analogous to the results in \cite{BK} for arbitrary nilpotent elements $e\in\mathfrak{gl}_n$, which realizes the corresponding $\mathcal{F}in_0(e)$ as a quotient category of $\cO_0^\mp(\mathfrak{gl}_n)$ where $\mp$ has type transposed to the Jordan type of $e$. Such a general result is clearly false in general outside of type $\rm A$. It already fails for Jordan types $(\la_1,\la_2)$ with $\la_1\not=\la_2$ or $\la_1=\la_2$ odd. It is expected that the above holds more general for all nilpotent $e$ with an even good grading in the sense of \cite{BrundanGoodwin} and trivial component group, but of course one cannot hope for such an explicit and elementary description of the endomorphism ring as in the case described above.

\subsection*{\bf The paper is organized as follows:}
In Section \ref{Section2} we introduce the algebraic Springer fibre and fix notation. Section \ref{Section3} sets up the diagram combinatorics which is crucial for the whole paper. Section \ref{Section4} defines the topological Springer fibre with its cell partition and computes the cohomology rings of the components and their intersections and finally connects it with the diagram algebra $\mK_k$ and its centre. Section~ \ref{Section5} establishes several combinatorial bijections between signed admissible standard tableaux, standard tableaux and cup diagrams and relates it to the combinatorics of representations of the Weyl group $W(D_k)$.  This is necessary to connect our cup diagram combinatorics to the original work of Spaltenstein and van Leeuwen and Hotta-Springer. For this paper the bijections are merely a tool, but we feel they are interesting on their own. Section \ref{Section6} gives the description of the irreducible components for the algebraic Springer fibre with an explicit algorithm for constructing the component attached to a cup diagram. Section~\ref{Section7} contains the proof of the main result, Theorem \ref{Hintro}. In Section~\ref{Section8} we finally describe the connection and applications and prove Theorem \ref{boringtheorem}.

\subsection*{Acknowledgment}
M.E. was financed by the DFG Priority program 1388. C.S. thanks Eric Sommers for many extremely helpful detailed explanations, and Pramod Achar and George Lusztig for examples of Springer fibres where the surjectivity assumption of \cite{KumarProcesi} fails. We thank Simon Goodwin, Jim Humphreys, Shrawan Kumar and Ivan Losev for helpful discussions. Both authors thank the university of Chicago where most of the research was done.

\section{Algebraic Springer fibre} \label{Section2}
\renc{\thetheorem}{\arabic{section}.\arabic{theorem}}

\subsubsection*{General assumptions}
We fix the ground field $\mC$ and denote by $\ii=\sqrt{-1}\in\mC$.

Fix an even natural number $n=2k$ and consider a complex $n$-dimensional vector space $W$ with fixed basis $w_i$, where $i\in \{\pm 1, \ldots, \pm k\}$. Equip $W$ with a symmetric non-degenerate bilinear form $\beta$ by setting $\beta(w_i,w_j)=\delta_{i,-j}$. Let $G=O(W,\beta)=O(2k,\mathbb{C})$ be the corresponding orthogonal group with Lie algebra  $\mg = \mathfrak{o}(2k,\mC)$.

\begin{definition}
A \emph{full isotropic flag} $F_\bullet$ in $W$ is a sequence of subspaces in $W$, $\{0\} \subset F_0 \subset F_1 \subset \cdots \subset F_k$, such that ${\rm dim}\, F_i = i$ and $\beta$ vanishes on $F_i \times F_i$ for all $0\leq i\leq k$ (i.e. $F_i$ is isotropic). Let $\mathcal{F}$  be the set of all such flags.
\end{definition}

Note that $F_k$ is maximal isotropic and $\op{dim}(F_i)=i$ for all $i$. Setting $F_{n-i}=(F_i)^\perp$ defines a full flag in $\mC^n$ in the usual sense.
We denote by $B$ the standard Borel of $G$ fixing the flag $F_\bullet$  where $F_i=\langle w_j\mid 1\leq j\leq i\rangle$. The set $\mathcal{F}$ inherits the structure of a projective algebraic variety, by identifying it with $G/B$.

\begin{remark} \label{Springer_two_components}
{\rm
In contrast to types $\rm A$, $\rm B$, $\rm C$, this flag variety of type $\rm D$ is not connected, but decomposes into two components isomorphic to $\mathcal{F}'=G'/B'$, where $G'=SO(2n,\mathbb{C})$ and $B'=G'\cap B$. The component containing $F_\bullet$ is determined by $F_k$. Given $F_\bullet$ there is a unique flag $F_\bullet'$ (called the \emph{companion flag}) which satisfies $F_i'=F_i$ for $0\leq i<k$ and $F_k'\not=F_k$. The two flags lie in different components.
}\end{remark}

A \emph{partition} $\lambda$ of $n=2k$ is a weakly decreasing sequence $\la_1\geq \la_2\geq\cdots\geq \la_r$ of positive integers summing up to $n$. Such partitions label the nilpotent orbits in $\mathfrak{gl}(n,\mC)$ under the conjugation action of $\op{GL}(n,\mC)$, where the parts $\la_i$ of $\la$ correspond to the sizes of the Jordan blocks.
\begin{definition}
\label{defpartition}
A partition $\la$ of $n=2k$ is called \emph{admissible} if even parts occur with even multiplicity.
\end{definition}

The classification of nilpotent orbits for $\mathfrak{gl}(n,\mC)$ restricts to $\mg$: the $G$-conjugation classes of nilpotent elements in $\mg$ are in bijection to admissible partitions of $n$, \cite{WilliamsonJ}, \cite{Gerstenhaber}. We now fix a nilpotent element $N\in\mg$.

\begin{definition}
\label{Springerfibre}
The \emph{Springer fibre} associated to $(G,N)$ is the algebraic variety $\mathcal{F}_N:=\mathcal{F}^{1+N}$ of $(1+N)$-fixed points in $\mathcal{F}$, i.e. of all flags $F_\bullet\in\mathcal{F}$ such that $NF_i\subset F_{i-1}$.
\end{definition}
The Springer fibre again decomposes into two connected components isomorphic to $\mathcal{F}_N':=\mathcal{F'}^{1+N}$ and is in contrast to $\cF$ in general not smooth. By a result of Spaltenstein it is equidimensional, \cite{Spaltensteinequi} (see also \cite[II.1.12]{SpaltensteinLNM}), i.e. all irreducible components have the same dimension. Let $\op{Irr}(\cF_N)$ denote the set of irreducible components.

Let $u\in G$ be the unipotent element corresponding to $N$ via the exponential map. The centraliser $C:=C_G(u)$ acts on $\mathcal{F}$, $\mathcal{F'}$ and the set  $\op{Irr}(\cF_N)$. In general $C$ is not connected and  we have the component group $A(u)=C_G(u)/C_G(u)^\circ$.

In the following we are interested in nilpotent orbits corresponding to 2-row partitions. Hence $N$ has Jordan type $(k,k)$ or $(n-l, l)$ with $l$ odd.

This situation is special in two perspectives: we will show in Theorem~\ref{maintheorem}, that all components are smooth (in analogy to 2-block nilpotents of type $\rm A$, \cite{Fung}, \cite{SW}, and in huge contrast to the general case \cite{Fressesing}) and secondly the component group is easy to calculate, \cite[2.26]{SpringerSteinberg}, as $A(u)=(\mathbb{Z}_2)^c$, where $c=0$ for the partition $(k,k)$ with $k$ even, $c=1$ if $k$ is odd, and $c=2$ otherwise.

\section{Diagram Combinatorics}
\label{Section3}
We start by introducing a diagram combinatorics which will later be used to describe the geometry of the topological Springer fibres. It generalizes the well-known approaches from
\cite{Khovanov_Springer}, \cite{Russell}, \cite{SW} to the more involved type $\rm D$ setting and is motivated by its connections to the type $\rm D$ highest weight Lie theory \cite{ESBrauer} and \cite{LS}.

\begin{definition}
We fix a rectangle $R$ in the lower half plane $\mR\times\mR_{\leq 0}$
with the set of points $P=\{1,\ldots, k\}$ on the top face.

An \emph{undecorated cup diagram}, is a diagram consisting of non-intersecting lines, each of them either connecting two distinct vertices in $P$ or one vertex in
$P$ with a point on the bottom face of $R$. The lines of the first type are called \emph{cups}, the others \emph{rays}. We consider two such diagrams as the same if they differ only by a planar isotopy of $R$ fixing $P$ pointwise and the bottom boundary line of the rectangle setwise.
A \emph{decorated cup diagram}, or short  \emph{cup diagram},  is a diagram with possibly "$\bullet$"'s (called {\it dots}) on the lines such that every dot is accessible from the left side of the rectangle,
meaning for each dot there exists a line in the rectangle connecting the
dot with the left face of $R$ not intersecting the diagram. We
allow at most one dot per line.
\end{definition}

\begin{ex} \label{example_cup diagrams}
The following first two diagrams are examples of cup diagrams, whereas the third is not, as the rightmost dot cannot be connected with the left face of $R$.
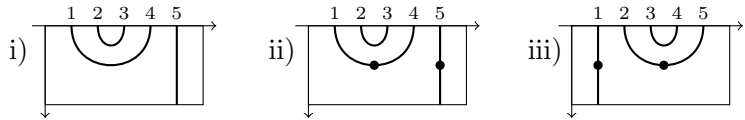
\begin{eqnarray*}
\begin{tikzpicture}[thick,scale=0.70]
\node at (-.5,-.5) {i)};
\node at (.5,.25) {\tiny{$1$}};
\node at (1,.25) {\tiny{$2$}};
\node at (1.5,.25) {\tiny{$3$}};
\node at (2,.25) {\tiny{$4$}};
\node at (2.5,.25) {\tiny{$5$}};
\draw (0.5,0) .. controls +(0,-1) and +(0,-1) .. +(1.5,0);
\draw (1,0) .. controls +(0,-.5) and +(0,-.5) .. +(.5,0);
\draw (2.5,0) -- (2.5,-1.5);
\draw [thin] (0,0) rectangle (3,-1.5);
\draw [thin,->] (-.25,0) -- (3.25,0);
\draw [thin,->] (0,0) -- (0,-1.75);

\begin{scope}[xshift=5cm]
\node at (-.5,-.5) {ii)};
\node at (.5,.25) {\tiny{$1$}};
\node at (1,.25) {\tiny{$2$}};
\node at (1.5,.25) {\tiny{$3$}};
\node at (2,.25) {\tiny{$4$}};
\node at (2.5,.25) {\tiny{$5$}};
\draw (0.5,0) .. controls +(0,-1) and +(0,-1) .. +(1.5,0);
\draw (1,0) .. controls +(0,-.5) and +(0,-.5) .. +(.5,0);
\draw (2.5,0) -- (2.5,-1.5);
\fill (1.25,-0.75) circle(2.5pt);
\fill (2.5,-0.75) circle(2.5pt);
\draw [thin] (0,0) rectangle (3,-1.5);
\draw [thin,->] (-.25,0) -- (3.25,0);
\draw [thin,->] (0,0) -- (0,-1.75);
\end{scope}

\begin{scope}[xshift=10cm]
\node at (-.5,-.5) {iii)};
\node at (.5,.25) {\tiny{$1$}};
\node at (1,.25) {\tiny{$2$}};
\node at (1.5,.25) {\tiny{$3$}};
\node at (2,.25) {\tiny{$4$}};
\node at (2.5,.25) {\tiny{$5$}};
\draw (1,0) .. controls +(0,-1) and +(0,-1) .. +(1.5,0);
\draw (1.5,0) .. controls +(0,-.5) and +(0,-.5) .. +(.5,0);
\draw (.5,0) -- (.5,-1.5);
\fill (1.75,-0.75) circle(2.5pt);
\fill (.5,-0.75) circle(2.5pt);
\draw [thin] (0,0) rectangle (3,-1.5);
\draw [thin,->] (-.25,0) -- (3.25,0);
\draw [thin,->] (0,0) -- (0,-1.75);
\end{scope}

\end{tikzpicture}
\end{eqnarray*}
\end{ex}

\begin{remark}
\label{counting}
Let $k$ be even. Then it is a well-known fact, see e.g. \cite{Stanley},  that the number of undecorated cup diagrams on $k$ vertices equals $\frac{1}{k+1} \binom{k}{k/2}$ (the $k$-th Catalan number) which can be identified with the dimension of the irreducible $S_k$-module indexed by the partition $(k/2,k/2)$, see e.g. \cite[Lemma 4.3.3]{Stperv}. On the other hand, the number of (decorated) cup diagrams on $k$ vertices equals, see Lemma \ref{bijchi}, twice the dimension of the irreducible representation of the Weyl group $S_k\ltimes(\mZ/2\mZ)^{k-1}$ of type ${\rm D}_k$ indexed by the pair of partitions $((k/2),(k/2))$. So, there are
$\binom{k}{k/2}$ such diagrams.
\end{remark}

\begin{notation}
\label{notation}
For a cup diagram $c$ and $i,j \in \{1,\ldots, k\}$ we write
\begin{enumerate}[i)]
\item $i \CupConnect j$ if $c$ contains a cup connecting $i$ and $j$,
\item $i \DCupConnect j$ if $c$ contains a dotted cup connecting $i$ and $j$,
\item $i \RayConnect$ if $c$ contains a ray that ends at $i$,
\item $i \DRayConnect$ if $c$ contains a dotted ray that ends at $i$.
\end{enumerate}

\begin{ex}
In the cup diagram i) from Example \ref{example_cup diagrams} we have $1 \CupConnect 4$, $2 \CupConnect 3$, and $5 \RayConnect$, while in the cup diagram ii) we have: $1 \DCupConnect 4$, $2 \CupConnect 3$, and $5 \DRayConnect$.
\end{ex}

We define the followings sets of (undotted/dotted/all) cups respectively rays appearing in a cup diagram $c$:
\begin{eqnarray*}
\begin{array}{lllllll}
{\rm cups_\circ}(c) &=& \{ (i<j) \mid i \CupConnect j \},&&
{\rm rays_\circ}(c) &=& \{ i \mid i \RayConnect \},\\
{\rm cups_\bullet}(c) &=& \{ (i<j) \mid i \DCupConnect j \},&&
{\rm rays_\bullet}(c) &=& \{ i \mid i \DRayConnect \},\\
{\rm cups}(c) &=& {\rm cups_\circ}(c) \cup {\rm cups_\bullet}(c),
&&
{\rm rays}(c) &= &{\rm rays_\circ}(c) \cup {\rm rays_\bullet}(c).
\end{array}
\end{eqnarray*}
\end{notation}

Reflecting cup diagrams along the horizontal middle line of $R$ gives us \emph{cap diagrams}, e.g. the cap diagrams arising from Example \ref{example_cup diagrams} are
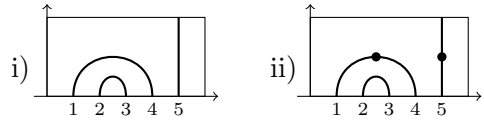
\begin{eqnarray*}
\begin{tikzpicture}[thick,scale=0.70]
\node at (-.5,.5) {i)};
\node at (.5,-.25) {\tiny{$1$}};
\node at (1,-.25) {\tiny{$2$}};
\node at (1.5,-.25) {\tiny{$3$}};
\node at (2,-.25) {\tiny{$4$}};
\node at (2.5,-.25) {\tiny{$5$}};
\draw (0.5,0) .. controls +(0,1) and +(0,1) .. +(1.5,0);
\draw (1,0) .. controls +(0,.5) and +(0,.5) .. +(.5,0);
\draw (2.5,0) -- (2.5,1.5);
\draw [thin] (0,0) rectangle (3,1.5);
\draw [thin,->] (-.25,0) -- (3.25,0);
\draw [thin,->] (0,0) -- (0,1.75);

\begin{scope}[xshift=5cm]
\node at (-.5,.5) {ii)};
\node at (.5,-.25) {\tiny{$1$}};
\node at (1,-.25) {\tiny{$2$}};
\node at (1.5,-.25) {\tiny{$3$}};
\node at (2,-.25) {\tiny{$4$}};
\node at (2.5,-.25) {\tiny{$5$}};
\draw (0.5,0) .. controls +(0,1) and +(0,1) .. +(1.5,0);
\draw (1,0) .. controls +(0,.5) and +(0,.5) .. +(.5,0);
\draw (2.5,0) -- (2.5,1.5);
\fill (1.25,0.75) circle(2.5pt);
\fill (2.5,0.75) circle(2.5pt);
\draw [thin] (0,0) rectangle (3,1.5);
\draw [thin,->] (-.25,0) -- (3.25,0);
\draw [thin,->] (0,0) -- (0,1.75);
\end{scope}
\end{tikzpicture}
\end{eqnarray*}

\begin{definition}
As above we assume to have fixed the rectangle $R$ and the set of points $P = \{1, \ldots, k \}$.
A \emph{cap diagram} is a diagram that is obtained from a cup diagram by reflecting it in the horizontal middle axis of R. If $c$ is a cup diagram we denote by $c^*$ the corresponding cap diagram. We use again the Notation \ref{notation} for cap diagrams instead of cup diagrams.
\end{definition}

To talk about oriented cup/cap diagrams we need the notion of a weight.

\begin{definition}
A \emph{(combinatorial) weight} of length $k$ is a sequence $\lambda = (\lambda_1,\ldots,\lambda_k)$ with $\lambda_i \in \{\up,\down\}$.

A cup (resp. cap) diagram $c$ together with a weight $\lambda$ is called an \emph{oriented cup (resp. cap) diagram} if the following holds:
\begin{enumerate}
\item if $i \CupConnect j$ then $\lambda_i \neq \lambda_j$,
\item if $i \DCupConnect j$ then $\lambda_i = \lambda_j$,
\item if $i \RayConnect$ then $\lambda_i = \down$,
\item if $i \DRayConnect$ then $\lambda_i = \up$.
\end{enumerate}
We denote the resulting diagram $\lambda c$ (resp. $c\lambda$).
\end{definition}
For instance, out of the following three only the first two diagrams are oriented. Note also that $\lambda c$ is oriented if and only if $c^*\lambda$ is oriented.
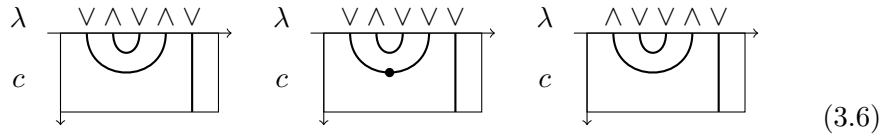
\begin{eqnarray}
\label{exoriented}
%
%
%
%
\begin{tikzpicture}[thick,scale=0.70]

\node at (-.8,.3) {$\lambda$};
\node at (-.8,-.9) {$c$};

\node at (.5,.3) {$\down$};
\node at (1,.3) {$\up$};
\node at (1.5,.3) {$\down$};
\node at (2,.3) {$\up$};
\node at (2.5,.3) {$\down$};
\draw (0.5,0) .. controls +(0,-1) and +(0,-1) .. +(1.5,0);
\draw (1,0) .. controls +(0,-.5) and +(0,-.5) .. +(.5,0);
\draw (2.5,0) -- (2.5,-1.5);
\draw [thin] (0,0) rectangle (3,-1.5);
\draw [thin,->] (-.25,0) -- (3.25,0);
\draw [thin,->] (0,0) -- (0,-1.75);

\begin{scope}[xshift=5cm]
\node at (-.8,.3) {$\lambda$};
\node at (-.8,-.9) {$c$};

\node at (.5,.3) {$\down$};
\node at (1,.3) {$\up$};
\node at (1.5,.3) {$\down$};
\node at (2,.3) {$\down$};
\node at (2.5,.3) {$\down$};
\draw (0.5,0) .. controls +(0,-1) and +(0,-1) .. +(1.5,0);
\draw (1,0) .. controls +(0,-.5) and +(0,-.5) .. +(.5,0);
\draw (2.5,0) -- (2.5,-1.5);
\fill (1.25,-.75) circle(2.5pt);
\draw [thin] (0,0) rectangle (3,-1.5);
\draw [thin,->] (-.25,0) -- (3.25,0);
\draw [thin,->] (0,0) -- (0,-1.75);
\end{scope}

\begin{scope}[xshift=10cm]
\node at (-.8,.3) {$\lambda$};
\node at (-.8,-.9) {$c$};

\node at (.5,.3) {$\up$};
\node at (1,.3) {$\down$};
\node at (1.5,.3) {$\down$};
\node at (2,.3) {$\up$};
\node at (2.5,.3) {$\down$};

\draw (0.5,0) .. controls +(0,-1) and +(0,-1) .. +(1.5,0);
\draw (1,0) .. controls +(0,-.5) and +(0,-.5) .. +(.5,0);
\draw (2.5,0) -- (2.5,-1.5);
\draw [thin] (0,0) rectangle (3,-1.5);
\draw [thin,->] (-.25,0) -- (3.25,0);
\draw [thin,->] (0,0) -- (0,-1.75);
\end{scope}
\end{tikzpicture}
\end{eqnarray}

Gluing a cap diagram on top of a cup diagram yields a {\it circle diagram}.

\begin{definition}
\label{defcircle}
A \emph{circle diagram} $(bc)$ consists of a cap diagram $b$ and a cup diagram $c$ where we identify the lower border of the rectangle of $b$ with the upper border of the rectangle of $c$ (identifying the marked points from left to right). An \emph{oriented circle diagram} $b\lambda c$ consists of a cap diagram $b$, a cup diagram $c$ and a weight $\lambda$ such that $b\lambda$ and $\lambda c$ are oriented diagrams:
\begin{eqnarray}
\label{circlediag}
\begin{tikzpicture}[thick,scale=0.90,decoration={zigzag,pre length=1mm,post length=1mm}]
\draw (.5,1) .. controls +(0,.5) and +(0,.5) .. +(.5,0);
\draw (1.5,1) .. controls +(0,.5) and +(0,.5) .. +(.5,0);
\draw (2.5,1) -- (2.5,2.5);
\fill (.75,1.36) circle(2.5pt);
\fill (1.75,1.36) circle(2.5pt);
\draw [thin] (0,1) rectangle (3,2.5);
\draw [thin,->] (-.25,1) -- (3.25,1);
\draw [thin,->] (0,1) -- (0,2.75);

\node at (-.8,.4) {$\lambda$};
\node at (-.8,-1) {$c$};
\node at (-.8,2) {$b$};

\node at (.5,.4) {$\down$};
\node at (1,.4) {$\down$};
\node at (1.5,.4) {$\up$};
\node at (2,.4) {$\up$};
\node at (2.5,.4) {$\down$};

\draw (0.5,0) .. controls +(0,-1) and +(0,-1) .. +(1.5,0);
\draw (1,0) .. controls +(0,-.5) and +(0,-.5) .. +(.5,0);
\draw (2.5,0) -- (2.5,-1.5);
\draw [thin] (0,0) rectangle (3,-1.5);
\draw [thin,->] (-.25,0) -- (3.25,0);
\draw [thin,->] (0,0) -- (0,-1.75);

\draw [->,decorate](3.7,.5) -- (5.7,.5);

\begin{scope}[xshift=7.2cm,yshift=.5cm]
\draw (.5,0) .. controls +(0,.5) and +(0,.5) .. +(.5,0);
\draw (1.5,0) .. controls +(0,.5) and +(0,.5) .. +(.5,0);
\draw (2.5,0) -- (2.5,1.5);
\fill (.75,.36) circle(2.5pt);
\fill (1.75,.36) circle(2.5pt);
\draw [thin] (0,0) rectangle (3,1.5);
\draw [thin,->] (0,0) -- (0,1.75);

\node at (-.8,-.1) {$b\lambda c$};

\node at (.505,.1) {$\down$};
\node at (1.005,.1) {$\down$};
\node at (1.505,-.15) {$\up$};
\node at (2.005,-.15) {$\up$};
\node at (2.505,.1) {$\down$};

\draw (0.5,0) .. controls +(0,-1) and +(0,-1) .. +(1.5,0);
\draw (1,0) .. controls +(0,-.5) and +(0,-.5) .. +(.5,0);
\draw (2.5,0) -- (2.5,-1.5);
\draw [thin] (0,0) rectangle (3,-1.5);
\draw [thin,->] (-.25,0) -- (3.25,0);
\draw [thin,->] (0,0) -- (0,-1.75);

\end{scope}

\end{tikzpicture}
\end{eqnarray}
\end{definition}

A circle diagram thus consists out of circles (=closed connected components) and lines. Note that each line has at most one allowed orientation, whereas each circle has precisely two, \cite[Lemma 4]{ES}.

\begin{definition}
A circle in an oriented circle diagram is called \emph{anticlockwise} if its rightmost label is $\up$ and \emph{clockwise} if its rightmost label is $\down$.
\end{definition}

\begin{center}
In the following we will usually omit drawing the rectangle $R$.
\end{center}

\noindent Fix $k \in \mN$. Then denote by $\mB_k$ the set of all cup diagrams on $k$ points with the maximal possible number of cups. We want to further divide this set into those cup diagrams with an even (respectively odd) number of dots, $\mB_k = \mB_k^{\rm even} \coprod \mB_k^{\rm odd}$.

\begin{ex}
For $k=3$ the set $\mB_3$ contains the following six diagrams
\begin{equation*}
\left\lbrace
\begin{array}{ccccccccccc}
\begin{tikzpicture}[thick, scale=0.7]
\begin{scope}[xshift=0cm]
\draw (0,0) .. controls +(0,-.5) and +(0,-.5) .. +(.5,0);

\draw (1,0) -- +(0,-.75);
\end{scope}
\end{tikzpicture}
&,&
\begin{tikzpicture}[thick, scale=0.7]
\begin{scope}[xshift=3cm]
\draw (0,0) .. controls +(0,-.5) and +(0,-.5) .. +(.5,0);
\fill (.25,-.365) circle(2pt);

\draw (1,0) -- +(0,-.75);
\fill (1,-.5) circle(2pt);
\end{scope}
\end{tikzpicture}
&,&
\begin{tikzpicture}[thick, scale=0.7]

\begin{scope}[xshift=6cm]
\draw (.5,0) .. controls +(0,-.5) and +(0,-.5) .. +(.5,0);

\draw (0,0) -- +(0,-.75);
\end{scope}
\end{tikzpicture}
&,&
\begin{tikzpicture}[thick, scale=0.7]
\begin{scope}[xshift=9cm]
\draw (0,0) .. controls +(0,-.5) and +(0,-.5) .. +(.5,0);
\fill (.25,-.365) circle(2pt);

\draw (1,0) -- +(0,-.75);
\end{scope}

\end{tikzpicture}
&,&
\begin{tikzpicture}[thick, scale=0.7]
\begin{scope}[xshift=12cm]
\draw (0,0) .. controls +(0,-.5) and +(0,-.5) .. +(.5,0);

\draw (1,0) -- +(0,-.75);
\fill (1,-.5) circle(2pt);
\end{scope}
\end{tikzpicture}
&,&
\begin{tikzpicture}[thick, scale=0.7]

\begin{scope}[xshift=15cm]
\draw (.5,0) .. controls +(0,-.5) and +(0,-.5) .. +(.5,0);

\draw (0,0) -- +(0,-.75);
\fill (0,-.5) circle(2pt);
\end{scope}
\end{tikzpicture}
\end{array}
\right\rbrace
\end{equation*}
with $\mB_3^{\rm even}$ and $\mB_3^{\rm odd}$ the first three respectively last three diagrams.
\end{ex}

\section{Topological Springer fibre}
\label{Section4}
In this section we construct the topological Springer fibre and establish basic properties and a cell decomposition.

\subsection{The definition of the topological Springer fibre}
In the following denote by $S$ the 2-sphere $\mathbb{S}^2 \subset \mR^3$ and by $X= S^{k}$ its $k$th power. By $H(Y)$ we denote the cohomology of a topological space $Y$ with complex coefficients, in particular, $H(S) \cong \mC[x]/\langle x^2\rangle$ with $\op{deg}(x)=2$. The isomorphism sends $x$ to the top class and $1$ to the class of a point which we fix and call $\pt$. We denote by $x \mapsto -x$ the antipodal map (i.e., the involution on $S$ that maps every point to its antipodal point).
Corresponding to $a \in \mB_k$ we define a subspace
$S_a \subset X$ as follows
\begin{equation}
\label{Sa}
(z_1,\ldots,z_k) \in S_a :\Leftrightarrow \left\lbrace \begin{array}{ll}
z_i = z_j & \text{ if } i \CupConnect j \\
z_i = -z_j & \text{ if } i \DCupConnect j \\
z_i = \pt & \text{ if } i \RayConnect \\
z_i = -\pt & \text{ if } i \DRayConnect.
\end{array} \right.
\end{equation}
Obviously, $S_a$ is homeomorphic to  $S^{\lfloor \frac{k}{2}\rfloor}=S^{{\rm cups}(a)}$. Hence we have
$$H(S_a)\cong\left(\mC[x]/\langle x^2\rangle\right)^{\otimes \lfloor \frac{k}{2}\rfloor}.$$
We will later use an explicit homeomorphism: For $a\in\mB_k$, $1 \leq i \leq k$ set
$$
\sigma(i) \, = \, (-1)^{ \begin{array}{c}
\#\left\lbrace (k<l) \in {\rm cups_\bullet}(a) \mid i \leq k \right\rbrace + \#\left\lbrace k \in {\rm rays_\bullet}(a) \mid i \leq k\right\rbrace
\end{array}},
$$
i.e., $\sigma$ encodes the parity of the number of dotted cups and rays to the right of the position $i$, including cups with left endpoint at $i$.

\begin{lemma}
Let $\{(i_r<j_r)\}={\rm cups}(a)$ with $i_1<i_2<\ldots <i_{\lfloor \frac{k}{2}\rfloor}$. The map
\begin{eqnarray}\label{local_coordinates}
\begin{array}{cccl}
\Psi_a :&S_a&\longrightarrow&S^{{\rm cups}(a)}=S^{\lfloor \frac{k}{2}\rfloor}\\
&(z_1,\ldots,z_k)&\longmapsto&(y_1,\ldots,y_{\lfloor \frac{k}{2}\rfloor})
\end{array}
\end{eqnarray}
where $y_{r}=\sigma(r)z_{i_r}$, $1\leq r\leq \lfloor \frac{k}{2}\rfloor$ defines a homeomorphism.
\end{lemma}

\begin{proof}
The inverse map $(\Psi_a)^{-1}$ is given by $z_{t}=\sigma(i_r)y_r$ if $t=i_r$, $z_{t}=\sigma(i_r)y_r$ if $(i_r<j_r=t)\in {\rm cups_\circ}(a)$, $z_{t}=-\sigma(i_r)y_r$ if $(i_r<j_r=t)\in {\rm cups_\bullet}(a)$ and $z_t=-p$ or $z_t=p$ otherwise, depending on if the ray is dotted or not.
\end{proof}

\begin{ex}
Consider the following cup diagram
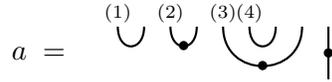
\begin{eqnarray*}
\begin{tikzpicture}[thick,scale=0.7]
\node at (-1,-.5) {$a\;=$};
\node at (.5,.25) {\tiny $(1)$};
\node at (1.5,.25) {\tiny $(2)$};
\node at (2.5,.25) {\tiny $(3)$};
\node at (3,.25) {\tiny $(4)$};

\draw (.5,0) .. controls +(0,-.5) and +(0,-.5) .. +(.5,0);
\draw (1.5,0) .. controls +(0,-.5) and +(0,-.5) .. +(.5,0);
\fill (1.75,-.36) circle(2.5pt);
\draw (3,0) .. controls +(0,-.5) and +(0,-.5) .. +(.5,0);
\draw (2.5,0) .. controls +(0,-1) and +(0,-1) .. +(1.5,0);
\fill (3.25,-.74) circle(2.5pt);
\draw (4.5,0) -- (4.5,-1);
\fill (4.5,-.5) circle(2.5pt);
\end{tikzpicture}
\end{eqnarray*}
and number the cups according to their left endpoints from left to right (as indicated by the bracketed numbers). Then
$$ (\Psi_a)^{-1}(x_1,x_2,x_3,x_4) = (-x_1,-x_1,-x_2,x_2,x_3,-x_4,-x_4,-x_3,-\pt).$$
\end{ex}

The following lemma follows directly from the definitions:

\begin{lemma} \label{Intersection_empty}
The subset $S_a \cap S_b \neq \emptyset$ if and only if the corresponding circle diagram $(a^*b)$ can be oriented. Furthermore if $(a^*b)$ can be oriented then $S_a \cap S_b \cong S^r$, where $r$ is the number of closed circles in $(a^*b)$. For the cohomology this implies
\begin{eqnarray*}
H(S_a \cap S_b) &\cong& \left\lbrace \begin{array}{ll}
\left(\mC[x]/\langle x^2\rangle\right)^{\otimes r} & \text{if } S_a \cap S_b \neq \emptyset \text{ and $r$ as above,} \\
0 & \text{otherwise.}
\end{array} \right.
\end{eqnarray*}
Each orientation of $(a^*b)$ corresponds then exactly to one standard basis vector $b_1\otimes \cdots \otimes b_l$, $b_l\in \{1,x\}$ in cohomology, by setting $b_i=1$ if the corresponding circle is oriented anticlockwise and $b_i=x$ if it is oriented clockwise.
\end{lemma}

\begin{remark}{\rm
Note the difference to \cite[Section 3]{Khovanov_Springer}, where the subsets $S_a \cap S_b$ are never empty.
}\end{remark}

Let $\widetilde{S} = \bigcup_{a \in \mB_k} S_a \subset X$. If $a \in \mB_k^{\rm even}$ and $b \in \mB_k^{\rm odd}$ then $S_a \cap S_b = \emptyset$ by Lemma~\ref{Intersection_empty}; thus we can decompose $\widetilde{S}$ into the disjoint subsets
\begin{eqnarray}
\label{soversuso}
\widetilde{S}^{\rm even} = \bigcup_{a \in \mB_k^{\rm even}} S_a& \text{ and }&\widetilde{S}^{\rm odd} = \bigcup_{a \in \mB_k^{\rm odd}} S_a.
\end{eqnarray}

We call $\widetilde{S}$ the \textit{topological Springer fibre}, as it is the topological substitute for the Springer fibre $\mathcal{F}_N$, see Remark \ref{conj}, and \eqref{soversuso} corresponds to the decomposition of $\mathcal{F}_N$ into the two copies of $\mathcal{F}_N'$.
Our goal is to show that the cohomology $H(\widetilde{S})$ of  $\widetilde{S}$ is isomorphic to the cohomology of the Springer fibre. We do this by identifying $H(\widetilde{S})$ with the centre of a diagrammatically defined algebra introduced in \cite{LS}, \cite{ES} of which we know that it has the desired centre.

\subsection{The cohomology rings $H(S_a)$}
We first realize all the involved cohomology rings as quotients of
\begin{eqnarray}
\label{HX}
H(X)=\mC[x]/\langle x^2\rangle^{\otimes {k}}=\mC[x_1,\ldots, x_k]/\left\langle x_i^2\mid 1\leq i\leq k\right\rangle.
\end{eqnarray}
where $x_i$ denotes the top class of the $i$th copy of $S$.

\begin{notation}
For $a,b \in \mB_k$ and $1 \leq i,j \leq k$ we write $i \thicksim_{a;b} j$ if there exists a sequence $i=i_1,i_2,\ldots,i_r=j$ such that for all $l$ the points $i_l$ and $i_{l+1}$ are connected by a cup in ${\rm cups}(a) \cup {\rm cups}(b)$ (i.e. $(i_l,i_{l+1}) \in {\rm cups}(a) \cup {\rm cups}(b)$ or $(i_{l+1},i_l) \in {\rm cups}(a) \cup {\rm cups}(b)$). We call such a sequence a \emph{path} connecting $i$ and $j$. This is obviously an equivalence relation on the set $\{1,\ldots,k\}$. For a chosen path from $i$ to $j$ we denote by $\alpha(i,j)$ the number of undotted cups used in the path. The parity of $\alpha(i,j)$ is independent of the chosen path and so $(-1)^{\alpha(i,j)}$ is well-defined. For $i,j \in \{1,\ldots,k\}$ we define
$$
\epsilon(i,j) = \left\lbrace \begin{array}{ll}
0 & \text{ if } i \nsim_{a;b} j, \\
(-1)^{\alpha(i,j)} & \text{ if } i \thicksim_{a;b} j.
\end{array} \right.
$$

We denote the set of equivalence classes by ${\rm Conn}(a,b)$, (they are in canonical bijection with the connected components of the corresponding circle diagram) and by $\overline{i}$ the equivalence class corresponding to the point $i$. We denote by ${\rm Conn}_\circ(a,b)$ those classes that correspond to closed circles in the diagram. We will later use the map
\begin{eqnarray*}
{\rm mx}_{a;b}:\{1,\ldots,k\} & \longrightarrow & \{1,\ldots,k\}\\
i &\mapsto& {\rm max}\{j \mid j \thicksim_{a;b} i\}.
\end{eqnarray*}
which maps each element to the maximal element in its equivalence class.
\end{notation}

\begin{lemma} \label{cohomology_realization}
\begin{enumerate}
\item Let $a \in \mB_k$. The canonical map induces a surjection
$$ \pi_a:\quad H(X) \twoheadrightarrow H(S_a) $$
with kernel
$$J_a=\left\langle \begin{array}{rrcl}
x_i + x_j & \text{if }& (i,j) &\in {\rm cups_\circ}(a), \\
x_i - x_j & \text{if }& (i,j) &\in {\rm cups_\bullet}(a), \\
x_i & \text{if }& i &\in {\rm rays}(a)
\end{array} \right\rangle.$$
\item Let $a,b \in \mB_k$, then the canonical map induces a surjection
$$ \pi_{a;b}:\quad H(X) \twoheadrightarrow H(S_a \cap S_b) $$
which in case $S_a \cap S_b\not=\emptyset$ (i.e. $(a^*b)$ can be oriented) has kernel
$$J_{a;b}=\left\langle \begin{array}{rrl}
x_i - \epsilon(i,{\rm mx_{a;b}}(i))x_{{\rm mx_{a;b}}(i)} & \text{if }& \overline{i} \in {\rm Conn}_\circ(a,b), \\
x_i & \text{if }& \overline{i} \notin {\rm Conn}_\circ(a,b)
\end{array} \right\rangle.$$
\end{enumerate}
\end{lemma}
\begin{proof}
Part a): The map is induced from the embedding of $S_a$ into $X$. In case of an undotted cup this is the induced map of embedding a sphere diagonally into a product of two spheres, hence gives the first type of relations in the ideal.  If there is a dotted cup, then the embedding is twisted by the antipodal map on one side, which produces a sign change and we obtain the second line of generators for the ideal. The last line is the induced map of embedding a point into a sphere. The generators of $H(S_a)$ are obviously in the image of the map. By comparing the dimensions the statement follows.

For part b) it is again clear that the kernel ${\rm ker}(\pi_{a;b})$ contains both $J_a$ and $J_b$ and thus we have a well-defined map
$$ H(X)/(J_a + J_b) \longrightarrow H(S_a \cap S_b)$$
 which contains the generators of  $H(S_a \cap S_b)$ in the image.
By part a), $J_a+J_b$ has the asserted generators. Comparing the dimensions using Lemma \ref{Intersection_empty} implies the claim.
\end{proof}

For $a,b\in\mB_k$, the canonical maps
$$ \psi_{a;a,b}:H(S_a) \longrightarrow H(S_a \cap S_b) \text{ and } \psi_{b;a,b}:H(S_b) \longrightarrow H(S_a \cap S_b),$$
induced by the inclusions of the corresponding spaces, satisfy $\pi_{a;b}= \psi_{a;a,b}\pi_{a}$ and  $\pi_{b;a}= \psi_{b;a,b}\pi_{b}$, and they are compatible with the $(H(S_a))$-module structure (resp. $(H(S_b))$-module structure) given by the cup product.

\subsection{Connection with the type $\rm D$ Khovanov arc algebra}
We shortly recall the main features of the type $\rm D$ Khovanov arc algebra.
\begin{definition}
For $k$ still fixed let $\mK_k$ be the vector space with basis
\begin{equation}
\mathcal{B}_k:=\left\{ b^*\la c \mid b,c\in\mB_k\text{ such that $b\la$ and $\la c$ are oriented}\right\},
\end{equation}
the set of oriented circle diagrams for $k$ points, where the involved cup and cap diagrams have the maximal possible number of cups. The \emph{degree}, $\op{deg}(b\la c)$, of a basis vector $b\la c$ is defined to be the number of \emph{clockwise} oriented cups and caps in the sense of \cite{LS} (i.e. oriented cups and caps labelled $\up\down$ or $\down\down$ in this order from left to right; e.g. the first two diagrams in \eqref{exoriented} have $1$ respectively $2$ clockwise cups).
\end{definition}

We have $\mK_k=\bigoplus_{(a,b)\in\mB_k\times \mB_k} {}_a(\mK_k)_b$, where ${}_a(\mK_k)_b$ is spanned by all basis vectors $a\la b\in\mathcal{B}_k$ with $a,b$ fixed.  By \cite{ES}, $\mK_k$ can be equipped with a multiplication turning it into a non-negatively graded associative algebra, $\mK_k$, where the grading is given by the degree function $\op{deg}$. The degree zero part is semisimple and spanned by the idempotents ${}_b \mathbbm{1}_b:=b\la b$ for $b \in \mB_k$, where $\la$ is the unique weight such that $b\la b$ is oriented of degree zero.

\begin{definition}
\label{min}
More generally, if $a,b \in \mB_k$ such that $a\la b\in\cB_k$ for some $\la$
we denote by ${}_a \mathbbm{1}_b$ the unique basis vector of ${}_a(\mK_k)_b$ of minimal degree, $d(a,b)$ (for an explicit formula see Lemma \ref{mindegree}).
\end{definition}
By \cite{ES} and Lemma \ref{cohomology_realization} we have for $a,b \in \mB_k$ canonical isomorphisms of graded vector spaces
\begin{equation}
\label{isoHK}
\begin{array}{llllllll}
{}_a(\mK_k)_a&\cong&H(S_a), &&{}_a(\mK_k)_b&\cong& H(S_a \cap S_b)\langle d(a,b)\rangle
\end{array}
\end{equation}
sending a basis vector $a\la b$ to $b_1\otimes \cdots \otimes b_l$, $b_l\in \{1,x\}$ where $b_i=1$ if the circle at vertex $i$ is anticlockwise and $b_i=x$ if it is clockwise. Here $\langle d\rangle$ denotes the grading shift upwards by $d$.
By the multiplication rules of the algebra $\mK_k$, a product $(a\la b)(c\mu d)$ of two basis vectors in $\mK_k$ can only be nonzero if $b^*=c$, moreover there are
induced linear maps
\begin{eqnarray}
\label{multidempotentwise}
 \gamma_{a;a,b}:{}_a(\mK_k)_a \longrightarrow {}_a(\mK_k)_b \text{ and } \gamma_{b;a,b}:{}_b(\mK_k)_b \longrightarrow {}_a(\mK_k)_b,
 \end{eqnarray}
defined as $d \longmapsto d\; {}_a \mathbbm{1}_b$ and $d \longmapsto {}_a \mathbbm{1}_b \;d$,
and turns them into left ${}_a(\mK_k)_a$-module respectively right ${}_b(\mK_k)_b$-module homomorphisms. With the vertical isomorphisms \eqref{isoHK} we obtain

\begin{prop} \label{cohomology_diagram_commutes}
The following diagram commutes
\begin{eqnarray}
\label{equ}
\begin{xy}
  \xymatrix{
      H(S_a) \ar[rr]^{\!\!\psi_{a;a,b}} \ar[d]^\cong && H(S_a \cap S_b) \ar[d]^\cong && H(S_b) \ar[ll]_{\;\;\psi_{b;a,b}\:} \ar[d]^\cong \\
      {}_a(\mK_k)_a \ar[rr]^{\gamma_{a;a,b}} && {}_a(\mK_k)_b && {}_b(\mK_k)_b \ar[ll]_{\gamma_{b;a,b}}
  }
\end{xy}
\end{eqnarray}
\end{prop}

\begin{proof}
Using the isomorphisms from Lemma \ref{cohomology_realization} it is enough to show that
\begin{eqnarray*}
\begin{xy}
  \xymatrix{
      H(X)/J_a \ar[r] \ar[d]^\cong & H(X)/J_{a;b} \ar[d]^\cong \\
      {}_a(\mK_k)_a \ar[r] & {}_a(\mK_k)_b
  }
\end{xy}
\end{eqnarray*}
commutes for any $a\not=b$.
By definition of the maps $1\in  H(X)/J_a$ is sent to ${}_a \mathbbm{1}_b\in {}_a(\mK_k)_b$ in both ways. Using first the vertical map, $x_i$ is sent to the basis vector which differs from ${}_a\mathbbm{1}_b$ by making the circle (if it exists) passing through vertex $i$ clockwise and multiply with $(-1)^{\alpha(i,\rm{mx}_{a;b}(i))}$ and it is sent to zero otherwise, see \cite{ES}.
On the other hand, the horizontal map sends $x_i$ to $x_i$ which, by Lemma \ref{cohomology_realization}, equals $(-1)^{\alpha(i,{\rm mx}_{a;b}(i))}x_{{\rm mx}_{a;b}(i)}$ which is then sent to the same basis vector as above, since $i$ and ${\rm mx}_{a;b}(i)$ lie on the same component. Since the $x_i$ generate $H(X)/J_a$ as ring, the claim follows.
\end{proof}

The modules $P(a):=\bigoplus_{b\in\mB_k}{}_b(\mK_k)_{a}$ for $a\in\mB_k$ form the indecomposable left $\mK_k$-modules with endomorphism rings ${}_a(\mK_k)_a$. The centre $Z(\mK_k)$ is therefore the subalgebra of $D:=\oplus_{a\in\mB_k} {}_a(\mK_k)_a=\prod_{a\in\mB_k} {}_a(\mK_k)_a$ given by all elements $(z_a)_{a\in\mB_k}$ such that $z_af=fz_b$ for all $f\in\HOM_{\mK_k}(P(a),P(b))={}_a(\mK_k)_b$, $b\in\mB_k$. Since $_a(\mK_k)_b$ is a cyclic $_a(\mK_k)_a$-module generated by ${}_a \mathbbm{1}_b$, and $_a(\mK_k)_a$ is commutative, \cite{ES}, it follows that $Z(\mK_k) = {\rm Eq}(\gamma)$, where
\begin{eqnarray} \label{Eqgamma}
 {\rm Eq}(\gamma)=\left\{(z_a)_{a\in\mB_k} \in D
 \begin{array}{c|c}
 &\gamma_{a;a,b}({}_a \mathbbm{1}_a z {}_a \mathbbm{1}_a) = \gamma_{b;a,b}({}_b \mathbbm{1}_b z {}_b \mathbbm{1}_b),\\
  &\text{for $a,b\in\mB_k$, $a\not=b$}
  \end{array}
  \right\},
\end{eqnarray}
i.e. the equalizer of $\gamma = \sum_{a \neq b} \gamma_{a;a,b} + \gamma_{b;a,b}:\prod_a {}_a(\mK_k)_a \longrightarrow \prod_{a \neq b} {}_a(\mK_k)_b.$ From \eqref{equ} we obtain therefore an algebra isomorphism
\begin{equation}
\label{cEq}
Z(\mK_k) \cong {\rm Eq}(\psi),
\end{equation}
where ${\rm Eq}(\psi)$ is defined analogously to \eqref{Eqgamma}. Since the inclusion of $S_a \cap S_b$ into $\widetilde{S}$ is the same regardless  if it is viewed as a subset of $S_a$ or of $S_b$, the canonical map $\tau':H(\widetilde{S})\longrightarrow \prod_a H(S_a)$ factors through ${\rm Eq}(\psi)$ and yields a homomorphism of algebras
$$ \tau:\quad H(\widetilde{S})\longrightarrow {\rm Eq}(\psi).$$

Our goal in Section \ref{cell_decomp} will be to show

\begin{theorem} \label{tau_is_iso}
The map $\tau$ is an isomorphism.
\end{theorem}


To show Theorem \ref{tau_is_iso} we follow closely \cite{Khovanov_Springer}. However, we have to reprove most of the involved steps since there are a few difficulties when transferring the arguments to type $\rm D$.
We start by introducing a partial order on $\mB_k$ and a cell decomposition of $\widetilde{S}$ and show that it behaves nicely with respect to the partial order.

\subsection{Partial Order on Diagrams}
Define a partial order on $\mB_k$ as chains of arrows as follows: For $a,b \in \mB_k$ we write $a \rightarrow b$ if $a$ and $b$ only differ locally by one of the following moves. (The diagrams might have many more cups even separating the involved cups; we only depict the ones which change in the move from $a$, depicted in the first column, to $b$,  depicted in the second column. The move always exists whenever both sides of the diagram give an allowed cup diagram. For instance there can't be a dotted cup between the two cups of $a$ in case (I), since then $b$ would have a nested dotted cup.
\begin{equation}
\label{moves}
\begin{array}{c|c}
\begin{tikzpicture}[thick]

\node at (0.75,1) {$a$};
\draw[->] (1.75,1) -- +(1,0);
\node at (3.75,1) {$b$};

\node at (-.5,0) {I)};

\draw (0,0) node[above]{$i$} .. controls +(0,-.5) and +(0,-.5) .. +(.5,0) node[above]{$j$};
\draw (1,0) node[above]{$k$} .. controls +(0,-.5) and +(0,-.5) .. +(.5,0) node[above]{$l$};
\draw[->] (1.75,-.2) -- +(1,0);
\draw (3,0) node[above]{$i$} .. controls +(0,-1) and +(0,-1) .. +(1.5,0) node[above]{$l$};
\draw (3.5,0) node[above]{$j$} .. controls +(0,-.5) and +(0,-.5) .. +(.5,0) node[above]{$k$};
\end{tikzpicture}
&
\begin{tikzpicture}[thick]

\node at (0.75,1) {$a$};
\draw[->] (1.75,1) -- +(1,0);
\node at (3.75,1) {$b$};

\node at (-.5,0) {II)};

\draw (0,0) node[above]{$i$} .. controls +(0,-1) and +(0,-1) .. +(1.5,0) node[above]{$l$};
\draw (.5,0) node[above]{$j$} .. controls +(0,-.5) and +(0,-.5) .. +(.5,0) node[above]{$k$};
\draw[->] (1.75,-.2) -- +(1,0);
\draw (3,0) node[above]{$i$} .. controls +(0,-.5) and +(0,-.5) .. +(.5,0) node[above]{$j$};
\fill (3.25,-.36) circle(2.5pt);
\draw (4,0) node[above]{$k$} .. controls +(0,-.5) and +(0,-.5) .. +(.5,0) node[above]{$l$};
\fill (4.25,-.36) circle(2.5pt);
\end{tikzpicture}
\\
\begin{tikzpicture}[thick]

\node at (-.5,0) {III)};

\draw (0,0) node[above]{$i$} .. controls +(0,-.5) and +(0,-.5) .. +(.5,0) node[above]{$j$};
\fill (.25,-.36) circle(2.5pt);
\draw (1,0) node[above]{$k$} .. controls +(0,-.5) and +(0,-.5) .. +(.5,0) node[above]{$l$};
\draw[->] (1.75,-.2) -- +(1,0);
\draw (3,0) node[above]{$i$} .. controls +(0,-1) and +(0,-1) .. +(1.5,0) node[above]{$l$};
\fill (3.75,-.74) circle(2.5pt);
\draw (3.5,0) node[above]{$j$} .. controls +(0,-.5) and +(0,-.5) .. +(.5,0) node[above]{$k$};
\end{tikzpicture}
&
\begin{tikzpicture}[thick]
\node at (-.5,0) {IV)};

\draw (0,0) node[above]{$i$} .. controls +(0,-1) and +(0,-1) .. +(1.5,0) node[above]{$l$};
\draw (.5,0) node[above]{$j$} .. controls +(0,-.5) and +(0,-.5) .. +(.5,0) node[above]{$k$};
\fill (0.75,-.74) circle(2.5pt);
\draw[->] (1.75,-.2) -- +(1,0);
\draw (3,0) node[above]{$i$} .. controls +(0,-.5) and +(0,-.5) .. +(.5,0) node[above]{$j$};
\draw (4,0) node[above]{$k$} .. controls +(0,-.5) and +(0,-.5) .. +(.5,0) node[above]{$l$};
\fill (4.25,-.36) circle(2.5pt);
\end{tikzpicture}
\\
\begin{tikzpicture}[thick]
\node at (-.5,0) {I')};

\draw (0,0) node[above]{$i$} .. controls +(0,-.5) and +(0,-.5) .. +(.5,0) node[above]{$j$};
\draw (1,0) node[above]{$k$} -- +(0,-.8);
\draw[->] (1.25,-.2) -- +(1,0);
\draw (2.5,0) node[above]{$i$} -- +(0,-.8);
\draw (3,0) node[above]{$j$} .. controls +(0,-.5) and +(0,-.5) .. +(.5,0) node[above]{$k$};
\end{tikzpicture}
&
\begin{tikzpicture}[thick]
\node at (-.5,0) {II')};

\draw (0,0) node[above]{$i$} -- +(0,-.8);
\draw (.5,0) node[above]{$j$} .. controls +(0,-.5) and +(0,-.5) .. +(.5,0) node[above]{$k$};
\draw[->] (1.25,-.2) -- +(1,0);
\draw (2.5,0) node[above]{$i$} .. controls +(0,-.5) and +(0,-.5) .. +(.5,0) node[above]{$j$};
\fill (2.75,-.36) circle(2.5pt);
\draw (3.5,0) node[above]{$k$} -- +(0,-.8);
\fill (3.5,-.4) circle(2.5pt);
\end{tikzpicture}
\\
\begin{tikzpicture}[thick]
\node at (-.5,0) {III')};

\draw (0,0) node[above]{$i$} .. controls +(0,-.5) and +(0,-.5) .. +(.5,0) node[above]{$j$};
\fill (.25,-.36) circle(2.5pt);
\draw (1,0) node[above]{$k$} -- +(0,-.8);
\draw[->] (1.25,-.2) -- +(1,0);
\draw (2.5,0) node[above]{$i$} -- +(0,-.8);
\fill (2.5,-.4) circle(2.5pt);
\draw (3,0) node[above]{$j$} .. controls +(0,-.5) and +(0,-.5) .. +(.5,0) node[above]{$k$};
\end{tikzpicture}
&
\begin{tikzpicture}[thick]
\node at (-.5,0) {IV')};

\draw (0,0) node[above]{$i$} -- +(0,-.8);
\fill (0,-.4) circle(2.5pt);
\draw (.5,0) node[above]{$j$} .. controls +(0,-.5) and +(0,-.5) .. +(.5,0) node[above]{$k$};
\draw[->] (1.25,-.2) -- +(1,0);
\draw (2.5,0) node[above]{$i$} .. controls +(0,-.5) and +(0,-.5) .. +(.5,0) node[above]{$j$};
\draw (3.5,0) node[above]{$k$} -- +(0,-.8);
\fill (3.5,-.4) circle(2.5pt);
\end{tikzpicture}
\end{array}
\end{equation}

 We extend this to a partial order on $\mB_k$ by setting $a \prec b$ if there exists a finite chain of arrows $a \rightarrow c_1 \rightarrow \ldots \rightarrow c_r \rightarrow b$. Since the local moves listed above always preserve the parity of the number of dots, an element from $\mB_k^{\rm even}$ is never $\prec$-comparable with an element from $\mB_k^{\rm odd}$. On the other hand, it is easy to see that the arrows induce connected graphs on $\mB_k^{\rm even}$ and $\mB_k^{\rm odd}$. From now on we will work with the two sets separately (each of them could be seen as the analog of \cite{Khovanov_Springer}), and fix a refinement of these two partial orders to a total order, denoted by $<$.

\begin{remark}
{\rm The graphs on $\mB_k^{\rm even}$ resp. $\mB_k^{\rm odd}$ appear after doubling the arrows as the Ext-quiver of the category of perverse sheaves on isotropic Grassmannians, constructible with respect to the Schubert stratification and describe the microlocal geometry there, see \cite{ES} and \cite{Braden}.
}
\end{remark}

\begin{definition}
Given a cup diagram a cup is called {\it outer} if it is not nested in any other cup and does not contain any dotted cup to the right. We call it {\it inner} if it is not outer. The {\it degree of nesting} of a cup is defined inductively: it is zero if it is outer and $d$ if it is outer after removing all cups with degree of nesting strictly smaller than $d$. (Note that the local moves $a\rightarrow b$ never decrease the degree of nesting.)
\end{definition}

\begin{ex}
\label{excupdiag}
Every cup diagram has at least one outer cup, i.e. a cup with degree of nesting equal to zero, for instance the outer cups for \eqref{introcup} are the cup $(3,4)$ in i), the cup $(1,4)$ in ii), and both cups $(1,2)$ and $(3,4)$ in iii). In addition we have here cups with degree of nesting equal to $1$, namely the cup $(1,2)$ in i) and the cup $(2,3)$ in ii).
\end{ex}

\begin{ex}\label{coolpicture}
The partial ordering on $\mB_3^{\rm even}$, with outer cups in (dashed) red, is displayed below. For  $\mB_3^{\rm odd}$ just add or remove (whatever is possible) a dot on the cup passing through the leftmost point.

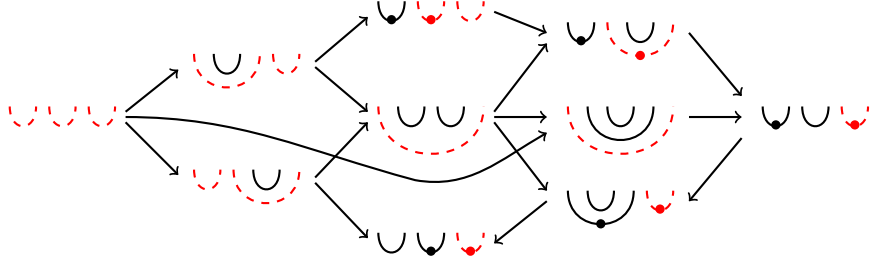
\begin{equation*}
\usetikzlibrary{arrows}
\begin{tikzpicture}[thick,scale=.7]

\draw (0,0)[red,dashed] .. controls +(0,-.5) and +(0,-.5) .. +(.5,0);
\draw (.75,0)[red,dashed] .. controls +(0,-.5) and +(0,-.5) .. +(.5,0);
\draw (1.5,0)[red,dashed] .. controls +(0,-.5) and +(0,-.5) .. +(.5,0);

\draw[->] (2.2, -.1) to +(1,.8);
\draw[->] (2.2, -.2) to [out=0,in=165] +(5.5,-1.2) to [out=-10,in=-150] +(2.5,.9);
\draw[->] (2.2, -.3) to +(1,-1);

\begin{scope}[xshift=3.5cm,yshift=1cm]
\draw (0,0)[red,dashed] .. controls +(0,-.85) and +(0,-.85) .. +(1.25,0);
\draw (.375,0) .. controls +(0,-.5) and +(0,-.5) .. +(.5,0);
\draw (1.5,0)[red,dashed] .. controls +(0,-.5) and +(0,-.5) .. +(.5,0);
\end{scope}

\begin{scope}[xshift=3.5cm,yshift=-1.2cm]
\draw (0,0)[red,dashed] .. controls +(0,-.5) and +(0,-.5) .. +(.5,0);
\draw (.75,0)[red,dashed] .. controls +(0,-.85) and +(0,-.85) .. +(1.25,0);
\draw (1.125,0) .. controls +(0,-.5) and +(0,-.5) .. +(.5,0);
\end{scope}

\draw[->] (5.8, .85) to +(1,.85);
\draw[->] (5.8, .75) to +(1,-.85);
\draw[->] (5.8, -1.35) to +(1,1.05);
\draw[->] (5.8, -1.45) to +(1,-1.05);

\begin{scope}[xshift=7cm,yshift=2cm]
\draw (0,0) .. controls +(0,-.5) and +(0,-.5) .. +(.5,0);
\fill (.25,-.35) circle(2.5pt);
\draw (.75,0)[red,dashed] .. controls +(0,-.5) and +(0,-.5) .. +(.5,0);
\fill (1,-.35)[red] circle(2.5pt);
\draw (1.5,0)[red,dashed] .. controls +(0,-.5) and +(0,-.5) .. +(.5,0);
\end{scope}

\begin{scope}[xshift=7cm]
\draw (0,0)[red,dashed] .. controls +(0,-1.2) and +(0,-1.2) .. +(2,0);
\draw (.375,0) .. controls +(0,-.5) and +(0,-.5) .. +(.5,0);
\draw (1.125,0) .. controls +(0,-.5) and +(0,-.5) .. +(.5,0);
\end{scope}

\begin{scope}[xshift=7cm,yshift=-2.4cm]
\draw (0,0) .. controls +(0,-.5) and +(0,-.5) .. +(.5,0);
\draw (.75,0) .. controls +(0,-.5) and +(0,-.5) .. +(.5,0);
\fill (1,-.35) circle(2.5pt);
\draw (1.5,0)[red,dashed] .. controls +(0,-.5) and +(0,-.5) .. +(.5,0);
\fill (1.75,-.35)[red] circle(2.5pt);
\end{scope}

\draw[->] (9.2, 1.8) to +(1,-.4);
\draw[->] (9.2, -.1) to +(1,1.3);
\draw[->] (9.2, -.2) to +(1,0);
\draw[->] (9.2, -.3) to +(1,-1.3);
\draw[<-] (9.2, -2.6) to +(1,.8);

\begin{scope}[xshift=10.6cm,yshift=1.6cm]
\draw (0,0) .. controls +(0,-.5) and +(0,-.5) .. +(.5,0);
\fill (.25,-.35) circle(2.5pt);
\draw (.75,0)[red,dashed] .. controls +(0,-.85) and +(0,-.85) .. +(1.25,0);
\fill (1.375,-.63)[red] circle(2.5pt);
\draw (1.125,0) .. controls +(0,-.5) and +(0,-.5) .. +(.5,0);
\end{scope}

\begin{scope}[xshift=10.6cm]
\draw (0,0)[red,dashed] .. controls +(0,-1.2) and +(0,-1.2) .. +(2,0);
\draw (.375,0) .. controls +(0,-.85) and +(0,-.85) .. +(1.25,0);
\draw (.75,0) .. controls +(0,-.5) and +(0,-.5) .. +(.5,0);
\end{scope}

\begin{scope}[xshift=10.6cm,yshift=-1.6cm]
\draw (0,0) .. controls +(0,-.85) and +(0,-.85) .. +(1.25,0);
\fill (.625,-.63) circle(2.5pt);
\draw (.375,0) .. controls +(0,-.5) and +(0,-.5) .. +(.5,0);
\draw (1.5,0)[red,dashed] .. controls +(0,-.5) and +(0,-.5) .. +(.5,0);
\fill (1.75,-.35)[red] circle(2.5pt);
\end{scope}

\draw[->] (12.9, 1.4) to +(1,-1.2);
\draw[->] (12.9, -.2) to +(1,0);
\draw[<-] (12.9, -1.8) to +(1,1.2);

\begin{scope}[xshift=14.3cm]
\draw (0,0) .. controls +(0,-.5) and +(0,-.5) .. +(.5,0);
\fill (.25,-.35) circle(2.5pt);
\draw (.75,0) .. controls +(0,-.5) and +(0,-.5) .. +(.5,0);
\draw (1.5,0)[red,dashed] .. controls +(0,-.5) and +(0,-.5) .. +(.5,0);
\fill (1.75,-.35)[red] circle(2.5pt);
\end{scope}
\end{tikzpicture}
\end{equation*}
\end{ex}

\subsection{The cell decomposition} \label{cell_decomp}
We will construct a cell decomposition of the topological Springer fibre by defining a paving of each $S_a$, $a\in \mB_k$, by even dimensional real spaces which we then show to behave well under intersections. To define a cell decomposition of $S_a$ for $a \in \mB_k$, we first attach to each  cup diagram $a\in\mB_k$ a directed graph $\Gamma_a$ as follows: The set of vertices
\begin{equation*}
\cV(\Gamma_a) = {\rm cups}(a)
\end{equation*}
is the set of cups in $a$; while the set of edges $\cE(\Gamma_a)$ is defined as follows: For two vertices $(i_1,j_1),(i_2,j_2) \in \cV(\Gamma_a)$ we put an arrow $(i_1,j_1) \rightarrow (i_2,j_2)$
if there exists a cup diagram $b \in \mB_k$ with $b \rightarrow a$
such that $a$ is obtained from $b$ by a local move of type \textbf{I)}-\textbf{IV)} at the positions $i_1,i_2,j_1,j_2$ and furthermore we demand that the degree of nesting of $(i_2,j_2)$ is greater than the one of $(i_1,j_1)$.

\begin{ex}
An example of a cup diagram $a$ with associated graph $\Gamma_a$:
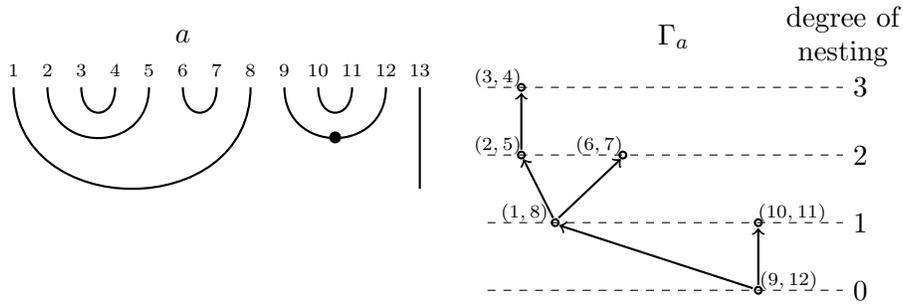
\begin{center}
\begin{tikzpicture}[thick,scale=0.9]
\begin{scope}
\node at (2.5,.75) {$a$};
\node at (0,.25) {\tiny $1$};
\node at (.5,.25) {\tiny $2$};
\node at (1,.25) {\tiny $3$};
\node at (1.5,.25) {\tiny $4$};
\node at (2,.25) {\tiny $5$};
\node at (2.5,.25) {\tiny $6$};
\node at (3,.25) {\tiny $7$};
\node at (3.5,.25) {\tiny $8$};
\node at (4,.25) {\tiny $9$};
\node at (4.5,.25) {\tiny $10$};
\node at (5,.25) {\tiny $11$};
\node at (5.5,.25) {\tiny $12$};
\node at (6,.25) {\tiny $13$};

\draw (0,0) .. controls +(0,-2) and +(0,-2) .. +(3.5,0) ;
\draw (.5,0) .. controls +(0,-1) and +(0,-1) .. +(1.5,0);
\draw (1,0) .. controls +(0,-.5) and +(0,-.5) .. +(.5,0);
\draw (2.5,0) .. controls +(0,-.5) and +(0,-.5) .. +(.5,0);

\draw (4,0) .. controls +(0,-1) and +(0,-1) .. +(1.5,0);
\draw (4.5,0) .. controls +(0,-.5) and +(0,-.5) .. +(.5,0);
\fill (4.75,-.74) circle(2.5pt);

\draw (6,0) -- +(0,-1.5);

\end{scope}

\begin{scope}[xshift=6.25cm]
\node at (3.5,.75) {$\Gamma_a$};
\node at (6,1) {degree of};
\node at (6,.5) {nesting};

\draw[dashed, thin] (.75,0) -- +(5.25,0) node[right]{$3$};
\draw[dashed, thin] (.75,-1) -- +(5.25,0) node[right]{$2$};
\draw[dashed, thin] (.75,-2) -- +(5.25,0) node[right]{$1$};
\draw[dashed, thin] (.75,-3) -- +(5.25,0) node[right]{$0$};

\draw (1.25,0) circle (.05);
\node at (.9,.15) {\tiny $(3,4)$};
\draw[<-] (1.25,-.08) -- +(0,-.84);

\draw (1.25,-1) circle (.05);
\node at (.9,-.85) {\tiny $(2,5)$};
\draw[<-] (1.28,-1.06) -- +(.44,-.87);
\draw (2.75,-1) circle (.05);
\node at (2.4,-.85) {\tiny $(6,7)$};
\draw[<-] (2.72,-1.06) -- +(-.92,-.87);

\draw (1.75,-2) circle (.05);
\node at (1.3,-1.85) {\tiny $(1,8)$};
\draw[<-] (1.81,-2.04) -- +(2.86,-.93);
\draw (4.75,-2) circle (.05);
\node at (5.25,-1.85) {\tiny $(10,11)$};
\draw[<-] (4.75,-2.08) -- +(0,-.84);

\draw (4.75,-3) circle (.05);
\node at (5.2,-2.85) {\tiny $(9,12)$};
\end{scope}
\end{tikzpicture}
\end{center}
\end{ex}

By construction $\Gamma_a$ is a forest, i.e. a union of trees. The roots $\cR(\Gamma_a)$ of $\Gamma_a$ are precisely the outer cups. Any cup is either a root or it is the target of a unique arrow, hence $\#\cR(\Gamma_a)+\#\cE(\Gamma_a) = \lfloor k/2 \rfloor$ and we have a natural bijection between $\cR\cup\cE$ and ${\rm cups}(a)$.

Now assign to each subset $J \subset \cR(\Gamma_a) \cup \cE(\Gamma_a)$ the subspace $C_J'$ of $S^{{\rm cups}(a)}$ given by all elements $(x_c)_{c \in {\rm cups}(a)}$ satisfying the following
\begin{enumerate}[(C1)]
\item If $c \in \cR(\Gamma_a) \cap J$ then $x_c = \pt$,
\item if $c \in \cR(\Gamma_a)$ but $c \not\in J$ then $x_c \neq \pt$,
\item if $(c \rightarrow d) \in \cE(\Gamma_a) \cap J$ then $x_c = x_d$,
\item if $(c \rightarrow d) \not\in \cE(\Gamma_a) \cap J$ then $x_c \neq x_d$.
\end{enumerate}

\begin{ex}
Consider the following cup diagram $a$ with its graph $\Gamma_a$
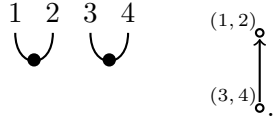
\begin{eqnarray*}
\begin{tikzpicture}[thick]
\draw (0,0) node[above]{$1$} .. controls +(0,-.5) and +(0,-.5) .. +(.5,0) node[above]{$2$};
\fill (0.25,-.36) circle(2.5pt);
\draw (1,0) node[above]{$3$} .. controls +(0,-.5) and +(0,-.5) .. +(.5,0) node[above]{$4$};
\fill (1.25,-.36) circle(2.5pt);

\draw (3.25,0) circle (.05);
\node at (2.9,.15) {\tiny $(1,2)$};
\draw[<-] (3.25,-.08) -- +(0,-.84);

\draw (3.25,-1) circle (.05);
\node at (2.9,-.85) {\tiny $(3,4)$};
\end{tikzpicture}.
\end{eqnarray*}
We have one root $\alpha$ and one edge $e$. The possible subsets $J$ are thus $\emptyset$, $\{\alpha\}$, $\{e\}$, and $\{\alpha,e\}$. The corresponding cells in $S \times S\cong S_a$ are then
$$\begin{array}{r||c|c}
J & C_J'\subset S\times S& C_j=(\Psi_a)^{-1}(C_J') \vspace{0.1cm}\\
\hline
\emptyset &\{(x,y)\mid x\neq y, y\neq \pt\}& \left\lbrace (x,-x,-y,y) \mid y \neq \pt,x \neq y \right\rbrace,\\
\{\alpha\} &\{(x,\pt)\mid y\not=\pt\}& \left\lbrace (x,-x,-\pt,\pt) \mid x \neq \pt \right\rbrace, \\
\{e\} &\{(x,x)\mid x\not=p\}& \left\lbrace (x,-x,-x,x) \mid x \neq \pt \right\rbrace,\\
\{\alpha,e\} &\{(\pt,\pt)\}& \left\lbrace (\pt,-\pt,-\pt,\pt) \right\rbrace.
\end{array}$$
\end{ex}

\begin{lemma} \label{decomp_disjoint}
With the notations above there is a decomposition
$$S^{{\rm cups}(a)} = \bigsqcup_{J \subset \cR(\Gamma_a)\cup \cE(\Gamma_a)} C_J'$$
 into disjoint affine cells $C_J'$ homeomorphic to $\mR^{2(n-|J|)}$.
\end{lemma}
\begin{proof}
The above decomposition is by construction disjoint. To see it is exhaustive we construct to each point $P\in S^{{\rm cups}(a)}$ a set $J$ such that $P\in C_J'$. First note that conditions (C1) and (C2) precisely define $J \cap \cR(\Gamma_a)$. Now we move upwards along the edges which uniquely determines $J \cap \cE(\Gamma)$ via conditions (C3) and (C4) by comparing the values at the two ends of each edge. The construction gives spaces homeomorphic to $\mR^{2(n-|J|)}$ and it is one verifies that $$\overline{C_J'} = \bigcup_{J \subset I} C_I'.$$
The lemma follows.
\end{proof}

Pushing forward along \eqref{local_coordinates} gives a cell decomposition
\begin{eqnarray}
\label{celldec}
S_a = \bigsqcup_{J \subset \cR(\Gamma_a)\cup \cE(\Gamma_a)} C_J,
\end{eqnarray}
with $C_J = (\Psi_a)^{-1}(C_J')$.

\subsection{Compatibility with intersections}
\begin{definition}
\label{distance}
For $a, b\in\mB_k$ we define the \emph{distance} $d(a,b)$ between $a$ and $b$ as the minimal length of a non-directed chain of arrows connecting $a$ and $b$. If no such chain exists we put $d(a,b)=\infty$.
\end{definition}

\begin{remark}{\rm
Since all local moves preserve parity we obviously have $d(a,b) = \infty$ if $a$ and $b$ have different parity. Otherwise $d(a,b)$ is  always finite, since the graphs for $\mB_k^{even}$ and  $\mB_k^{odd}$ are finite and connected.
}\end{remark}

\begin{lemma} \label{lower_bound_exists}
Let $a,b \in \mB_k$ with $d(a,b) < \infty$. Then there exists $c$ such that
$d(a,b) = d(a,c) + d(c,b)$ and $a \succ c \prec b$.
\end{lemma}

\begin{proof}
Let $(a = a_0,a_1,\ldots,a_r=b)$ be a minimal non-directed chain from $a$ to $b$. If $a_0\succ a_1\succ \cdots\succ a_i\prec a_{i+1}\prec\cdots\prec a_r$ for some $i$ then we set $c=a_i$ and nothing is to do. Otherwise there exists some subsequence $a_{i-1} \rightarrow a_i \leftarrow a_{i+1}$ and the idea of the proof is to successively eliminate all such subsequences without changing the length of the chain. From the set of all such subsequences, choose one (not necessarily uniquely defined) such that its middle term $a_i$ is maximal in the order.

Due to the minimality of the chain we have $a_{i-1} \neq a_{i+1}$ and there exists by \cite[Lemma 1.8.3]{Braden} with the identification from \cite{ES} a diamond


\begin{eqnarray*}
\begin{tikzpicture}[thick,scale=0.5]
\draw[<-] (1,-1) node[above]{\quad $a_i$} -- +(-1,-1);
\draw[<-] (1.5,-1) -- +(1,-1);
\draw[<->] (0,-2.7) node[above]{$a_{i-1}$} -- +(1,-1) node[below]{\quad $a'$};
\draw[<->] (2.5,-2.7) node[above]{$a_{i+1}$} -- +(-1,-1);
\end{tikzpicture}
\end{eqnarray*}
with the property that either $a'>a_{i-1}$ or $a' > a_{i+1}$ (indicated by the double arrows). We now substitute $a_i$ with $a'$ in our chain of arrows. If this procedure creates a subsequence such that $a_{i-2} \rightarrow a_{i-1} \leftarrow a'$ or $a' \rightarrow a_{i+1} \leftarrow a_{i+2}$, then both $a_{i-1}$ and $a_{i+1}$ are lower in the order and we proceed by induction.
In the end this will produce a sequence of the form $a \leftarrow a_1'
\leftarrow \ldots \leftarrow a_i' \rightarrow \ldots \rightarrow a_{r-1}'
\rightarrow b$, which is still minimal and $a_i'$ satisfies the conditions
for the element $c$. (Note that $a_i'$ can be equal to $a$ or $b$.)
\end{proof}

As in type $\rm A$, the distance function has a nice diagrammatical interpretation also showing the consistency of our notation $d(a,b)$:
\begin{lemma}
\label{mindegree}
Let $a,b\in\mB_k$ with $d(a,b)<\infty$ and let $m=\lfloor \frac{k}{2}\rfloor$. If $b^*a$ can be oriented then the two definitions of $d(a,b)$ (Definition \ref{min} and \ref{distance}) agree, and $m-d(a,b)$ equals the numbers  of circles in $b^*a$.
\end{lemma}

\begin{proof}
We first claim that the minimal degree equals $m-\gamma=\lfloor \frac{k}{2}\rfloor-\gamma$, where $\gamma$ is the number of circles in  $b^*a$.
We argue by induction on $k$. The base cases $k=1$ or $k=2$ are easy to check. If $a=b$ then $a^*a$ has $m$ circles each containing exactly one cup and cap. Hence orienting each of them anti-clockwise gives the minimal degree $0$ and the claim follows. If $a\not=b$ then the diagram $a^*b$ contains at least one kink. If we can find an undecorated kink we remove it by the following straightening rules:
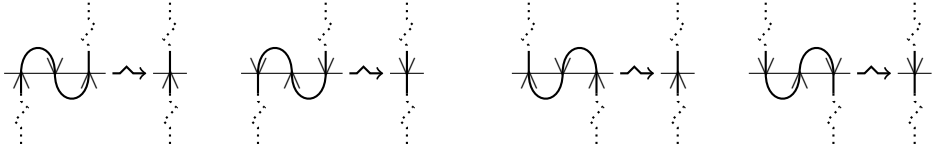
\begin{eqnarray*}
\begin{tikzpicture}[scale=0.9,thick,decoration={zigzag,pre length=1mm,post length=1mm}]
\begin{scope}
\draw (0,0) -- +(0,-.3);
\draw[dotted, decorate] (0,-.3) -- +(0,-.8);
\draw (0,0) .. controls +(0,.5) and +(0,.5) .. +(.5,0);
\draw (.5,0) .. controls +(0,-.5) and +(0,-.5) .. +(.5,0);
\draw (1,0) -- +(0,.3);
\draw[dotted, decorate] (1,.3) -- +(0,.8);

\draw[thin] (-.25,0) -- +(1.5,0);

\node at (.005,-.1) {$\up$};
\node at (1.005,-.1) {$\up$};
\node at (.505,.1) {$\down$};

\draw[->, decorate] (1.35,0) -- +(0.5,0);

\begin{scope}[xshift=-0.8cm]
\draw[thin] (2.75,0) -- +(.5,0);
\draw (3,0) -- +(0,.3);
\draw[dotted, decorate] (3,.3) -- +(0,.8);
\draw (3,0) -- +(0,-.3);
\draw[dotted, decorate] (3,-.3) -- +(0,-.8);
\node at (3.005,-.1) {$\up$};
\end{scope}
\end{scope}

\begin{scope}[xshift=3.5cm]
\draw (0,0) -- +(0,-.3);
\draw[dotted, decorate] (0,-.3) -- +(0,-.8);
\draw (0,0) .. controls +(0,.5) and +(0,.5) .. +(.5,0);
\draw (.5,0) .. controls +(0,-.5) and +(0,-.5) .. +(.5,0);
\draw (1,0) -- +(0,.3);
\draw[dotted, decorate] (1,.3) -- +(0,.8);

\draw[thin] (-.25,0) -- +(1.5,0);

\node at (.005,.1) {$\down$};
\node at (1.005,.1) {$\down$};
\node at (.505,-.1) {$\up$};

\draw[->, decorate] (1.35,0) -- +(0.5,0);

\begin{scope}[xshift=-0.8cm]
\draw[thin] (2.75,0) -- +(.5,0);
\draw (3,0) -- +(0,.3);
\draw[dotted, decorate] (3,.3) -- +(0,.8);
\draw (3,0) -- +(0,-.3);
\draw[dotted, decorate] (3,-.3) -- +(0,-.8);
\node at (3.005,.1) {$\down$};
\end{scope}
\end{scope}

\begin{scope}[xshift=7.5cm]
\draw (0,0) -- +(0,.3);
\draw[dotted, decorate] (0,.3) -- +(0,.8);
\draw (0,0) .. controls +(0,-.5) and +(0,-.5) .. +(.5,0);
\draw (.5,0) .. controls +(0,.5) and +(0,.5) .. +(.5,0);
\draw (1,0) -- +(0,-.3);
\draw[dotted, decorate] (1,-.3) -- +(0,-.8);

\draw[thin] (-.25,0) -- +(1.5,0);

\node at (.005,-.1) {$\up$};
\node at (1.005,-.1) {$\up$};
\node at (.505,.1) {$\down$};

\draw[->, decorate] (1.35,0) -- +(0.5,0);
\begin{scope}[xshift=-0.8cm]

\draw[thin] (2.75,0) -- +(.5,0);
\draw (3,0) -- +(0,.3);
\draw[dotted, decorate] (3,.3) -- +(0,.8);
\draw (3,0) -- +(0,-.3);
\draw[dotted, decorate] (3,-.3) -- +(0,-.8);
\node at (3.005,-.1) {$\up$};
\end{scope}
\end{scope}

\begin{scope}[xshift=11cm]
\draw (0,0) -- +(0,.3);
\draw[dotted, decorate] (0,.3) -- +(0,.8);
\draw (0,0) .. controls +(0,-.5) and +(0,-.5) .. +(.5,0);
\draw (.5,0) .. controls +(0,.5) and +(0,.5) .. +(.5,0);
\draw (1,0) -- +(0,-.3);
\draw[dotted, decorate] (1,-.3) -- +(0,-.8);

\draw[thin] (-.25,0) -- +(1.5,0);

\node at (.005,.1) {$\down$};
\node at (1.005,.1) {$\down$};
\node at (.505,-.1) {$\up$};

\draw[->, decorate] (1.35,0) -- +(0.5,0);

\begin{scope}[xshift=-0.8cm]
\draw[thin] (2.75,0) -- +(.5,0);
\draw (3,0) -- +(0,.3);
\draw[dotted, decorate] (3,.3) -- +(0,.8);
\draw (3,0) -- +(0,-.3);
\draw[dotted, decorate] (3,-.3) -- +(0,-.8);
\node at (3.005,.1) {$\down$};
\end{scope}
\end{scope}
\end{tikzpicture}
\end{eqnarray*}

This decreases $m$ and the degree function $d(a,b)$ by one, but keeps the number of circles and we are done by induction. If there is no such undecorated kink, then there is at least one decorated kink. The kink can be followed on both sides by rays which means it is part of a subdiagram from the following list (recalling that dotted cups can't be to the right of rays):

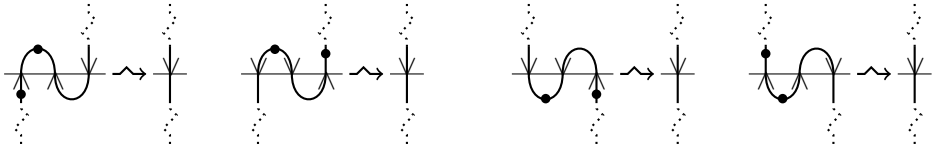
\begin{eqnarray*}
\begin{tikzpicture}[scale=0.9,thick,decoration={zigzag,pre length=1mm,post length=1mm}]

\begin{scope}
\draw[thin] (-.25,0) -- +(1.5,0);
\draw (0,0) -- +(0,-.4);
\draw[dotted, decorate] (0,-.4) -- +(0,-.7);
\fill (0,-.3) circle(2pt);
\draw (0,0) .. controls +(0,.5) and +(0,.5) .. +(.5,0);
\fill (.25,.365) circle(2pt);
\draw (.5,0) .. controls +(0,-.5) and +(0,-.5) .. +(.5,0);
\draw (1,0) -- +(0,.4);
\draw[dotted, decorate] (1,.4) -- +(0,.7);
\node at (.005,-.1) {$\up$};
\node at (.51,-.1) {$\up$};
\node at (1.005,.1) {$\down$};

\draw[->, decorate] (1.35,0) -- +(0.5,0);

\draw[thin] (1.95,0) -- +(.5,0);
\draw (2.2,-.4) -- +(0,0.8);
\draw[dotted, decorate] (2.2,.4) -- +(0,.7);
\draw[dotted, decorate] (2.2,-.4) -- +(0,-.7);
\node at (2.215,.1) {$\down$};
\end{scope}

\begin{scope}[xshift=3.5cm]
\draw[thin] (-.25,0) -- +(1.5,0);
\draw (0,0) -- +(0,-.4);
\draw[dotted, decorate] (0,-.4) -- +(0,-.7);
\draw (0,0) .. controls +(0,.5) and +(0,.5) .. +(.5,0);
\fill (.25,.365) circle(2pt);
\draw (.5,0) .. controls +(0,-.5) and +(0,-.5) .. +(.5,0);
\draw (1,0) -- +(0,.4);
\draw[dotted, decorate] (1,.4) -- +(0,.7);
\fill (1,.3) circle(2pt);
\node at (.01,.1) {$\down$};
\node at (.505,.1) {$\down$};
\node at (1.01,-.1) {$\up$};

\draw[->, decorate] (1.35,0) -- +(0.5,0);

\draw[thin] (1.95,0) -- +(.5,0);
\draw (2.2,-.4) -- +(0,0.8);
\draw[dotted, decorate] (2.2,.4) -- +(0,.7);
\draw[dotted, decorate] (2.2,-.4) -- +(0,-.7);
\node at (2.205,.1) {$\down$};
\end{scope}

\begin{scope}[xshift=7.5cm]
\draw[thin] (-.25,0) -- +(1.5,0);
\draw (0,0) -- +(0,.4);
\draw[dotted, decorate] (0,.4) -- +(0,.7);
\draw (0,0) .. controls +(0,-.5) and +(0,-.5) .. +(.5,0);
\draw (.5,0) .. controls +(0,.5) and +(0,.5) .. +(.5,0);
\fill (.25,-.365) circle(2pt);
\draw (1,0) -- +(0,-.4);
\draw[dotted, decorate] (1,-.4) -- +(0,-.7);
\fill (1,-.3) circle(2pt);
\node at (.005,.1) {$\down$};
\node at (.51,.1) {$\down$};
\node at (1.005,-.1) {$\up$};

\draw[->, decorate] (1.35,0) -- +(0.5,0);

\draw[thin] (1.95,0) -- +(.5,0);
\draw (2.2,-.4) -- +(0,0.8);
\draw[dotted, decorate] (2.2,.4) -- +(0,.7);
\draw[dotted, decorate] (2.2,-.4) -- +(0,-.7);
\node at (2.205,.1) {$\down$};
\end{scope}

\begin{scope}[xshift=11.cm]
\draw[thin] (-.25,0) -- +(1.5,0);
\draw (0,0) -- +(0,.4);
\draw[dotted, decorate] (0,.4) -- +(0,.7);
\fill (0,.3) circle(2pt);
\draw (0,0) .. controls +(0,-.5) and +(0,-.5) .. +(.5,0);
\fill (.25,-.365) circle(2pt);
\draw (.5,0) .. controls +(0,.5) and +(0,.5) .. +(.5,0);
\draw (1,0) -- +(0,-.4);
\draw[dotted, decorate] (1,-.4) -- +(0,-.7);
\node at (.005,-.1) {$\up$};
\node at (.505,-.1) {$\up$};
\node at (1.005,.1) {$\down$};

\draw[->, decorate] (1.35,0) -- +(0.5,0);

\draw[thin] (1.95,0) -- +(.5,0);
\draw (2.2,-.4) -- +(0,0.8);
\draw[dotted, decorate] (2.2,.4) -- +(0,.7);
\draw[dotted, decorate] (2.2,-.4) -- +(0,-.7);
\node at (2.21,.1) {$\down$};
\end{scope}

\end{tikzpicture}
\end{eqnarray*}

Then we apply the ``dot removal trick'', that means we remove the two dots and adjust the orientation in the unique possible way. Observe that the degree of the diagram is not changed. The result contains an undecorated kink and we can argue by induction as above.

In all other cases, the kink is part of a subdiagram from the following list (recalling the assumption that there is no undecorated kink and also the rules for placing the dots):
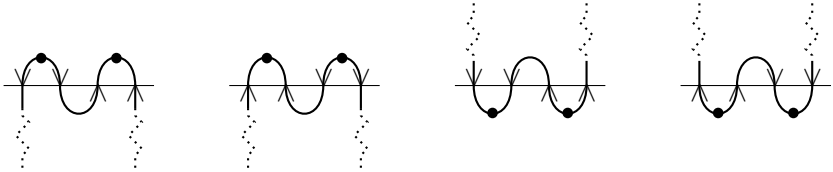
\begin{eqnarray*}
\begin{tikzpicture}[thick,decoration={zigzag,pre length=1mm,post length=1mm}]
\begin{scope}
\draw (0,0) -- +(0,-.3);
\draw[dotted,decorate] (0,-.3) -- +(0,-.8);
\draw (1.5,0) -- +(0,-.3);
\draw[dotted,decorate] (1.5,-.3) -- +(0,-.8);
\draw (0,0) .. controls +(0,.5) and +(0,.5) .. +(.5,0);
\fill (.25,.365) circle(2pt);
\draw (.5,0) .. controls +(0,-.5) and +(0,-.5) .. +(.5,0);
\draw (1,0) .. controls +(0,.5) and +(0,.5) .. +(.5,0);
\fill (1.25,.365) circle(2pt);

\draw[thin] (-.25,0) -- +(2,0);

\node at (.005,.1) {$\down$};
\node at (.505,.1) {$\down$};
\node at (1.005,-.1) {$\up$};
\node at (1.505,-.1) {$\up$};
\end{scope}

\begin{scope}[xshift=3cm]
\draw (0,0) -- +(0,-.3);
\draw[dotted,decorate] (0,-.3) -- +(0,-.8);
\draw (1.5,0) -- +(0,-.3);
\draw[dotted,decorate] (1.5,-.3) -- +(0,-.8);
\draw (0,0) .. controls +(0,.5) and +(0,.5) .. +(.5,0);
\fill (.25,.365) circle(2pt);
\draw (.5,0) .. controls +(0,-.5) and +(0,-.5) .. +(.5,0);
\draw (1,0) .. controls +(0,.5) and +(0,.5) .. +(.5,0);
\fill (1.25,.365) circle(2pt);

\draw[thin] (-.25,0) -- +(2,0);

\node at (.005,-.1) {$\up$};
\node at (.505,-.1) {$\up$};
\node at (1.005,.1) {$\down$};
\node at (1.505,.1) {$\down$};

\end{scope}

\begin{scope}[xshift=6cm]
\draw (0,0) -- +(0,.3);
\draw[dotted,decorate] (0,.3) -- +(0,.8);
\draw (1.5,0) -- +(0,.3);
\draw[dotted,decorate] (1.5,.3) -- +(0,.8);
\draw (0,0) .. controls +(0,-.5) and +(0,-.5) .. +(.5,0);
\fill (.25,-.365) circle(2pt);
\draw (.5,0) .. controls +(0,.5) and +(0,.5) .. +(.5,0);
\draw (1,0) .. controls +(0,-.5) and +(0,-.5) .. +(.5,0);
\fill (1.25,-.365) circle(2pt);

\draw[thin] (-.25,0) -- +(2,0);

\node at (.005,.1) {$\down$};
\node at (.505,.1) {$\down$};
\node at (1.005,-.1) {$\up$};
\node at (1.505,-.1) {$\up$};
\end{scope}

\begin{scope}[xshift=9cm]
\draw (0,0) -- +(0,.3);
\draw[dotted,decorate] (0,.3) -- +(0,.8);
\draw (1.5,0) -- +(0,.3);
\draw[dotted,decorate] (1.5,.3) -- +(0,.8);
\draw (0,0) .. controls +(0,-.5) and +(0,-.5) .. +(.5,0);
\fill (.25,-.365) circle(2pt);
\draw (.5,0) .. controls +(0,.5) and +(0,.5) .. +(.5,0);
\draw (1,0) .. controls +(0,-.5) and +(0,-.5) .. +(.5,0);
\fill (1.25,-.365) circle(2pt);

\draw[thin] (-.25,0) -- +(2,0);

\node at (.005,-.1) {$\up$};
\node at (.505,-.1) {$\up$};
\node at (1.005,.1) {$\down$};
\node at (1.505,.1) {$\down$};

\end{scope}
\end{tikzpicture}
\end{eqnarray*}

Apply again the dot removal trick, so remove the two dots and adjust the orientation such that the two outer labels, thus the type (clockwise or anticlockwise) as well as the degree of the diagram is preserved. Again, the result contains an undecorated kink and we can argue by induction as above. This proves the claim.

Starting from $a^*a$ we apply the minimal path to $a$ to obtain $b$. In each step
the number of circles decreases at most by one, hence the distance $d(a,b)$ is at least $m$ minus the number of circles. To show equality consider $a^*b$.
If $a=b$ then there is nothing to do, so assume $a\not=b$. If there is a circle containing two outer cups then we can apply a move $b\rightarrow b'$ or $b\leftarrow b'$  which  splits the circle into two. Otherwise, we find two cups on the same circle paired by an edge such that one of the cups is outer. Again we can apply some move to split the original circle. Repeating the procedure we finally give a path $d(a,b)$ of the required length starting at $a^*b$ and changing successively $b$. It ends in $a^*a$ since in each step one new circle was created (resulting in $m$ circles) without changing the cap diagram $a^*$.
\end{proof}

\begin{lemma} \label{distance_and_intersection}
Let $c$ be such that $d(a,b)=d(a,c) + d(c,b)$ then
$$S_a \cap S_b = S_a \cap S_{c} \cap S_b.$$
\end{lemma}
\begin{proof}
If $S_a \cap S_b = \emptyset$ there is nothing to show, so we assume $S_a \cap S_b \neq \emptyset$. In case $a=b$ we have $c=a=b$ and the claim is again trivial. Now we argue by some twofold induction, namely on $d(a,b)$ starting with $d(a,b)=1$ and on $d(a,c)$ starting with $d(a,c)=1$.

In case $d(a,b)=1$, there is again nothing to show since either $c=a$ or $c=b$, so we assume now  $d(a,c)=1$ (and $a\not=b$ arbitrary). Since $S_a \cap S_b \neq \emptyset$ the circle diagram $a^* b$ can be oriented and contains $\lfloor \frac{k}{2}\rfloor-d(a,b)$ circles by Lemma \ref{mindegree}. Since $c$ lies on a minimal length path between $b$ and $a$ there is a sequence of moves from $a^*b$ to $a^*c$ and then to $a^*a$ which in each step creates a new circle. In particular, every orientation of $a^*b$ gives rise to an orientation of $a^*c$ and then obviously also of $a^*a$ and therefore $S_a \cap S_b \subset (S_a \cap S_c)\cap S_b$.

It remains to show the statement for $d(a,b)=r$, $d(a,c)=s>1$ with $c \neq b$. We assume that it holds for all $a'$, $b'$, $c'$ such that $d(a',b')=d(a',c') + d(c',b') < r$ and all $x$ such that $d(a,b)=d(a,x)+d(x,b)$ and $d(a,x)<s$.

Choose $x \neq a,c$ on a minimal path connecting $a$ and $c$ thus satisfying $d(x,b) = d(x,c) + d(c,b)<r$. Then by assumption,
$$S_x \cap S_b = S_x \cap S_c \cap S_b.$$
Since $d(a,x)<s$ and $d(a,b) = d(a,x)+d(x,b)$ we also have
$$S_a \cap S_b = S_a \cap S_b \cap S_x.$$
Altogether this gives
$$S_a \cap S_b = S_a \cap (S_b \cap S_x) = S_a \cap S_b \cap S_x \cap S_c = S_a \cap S_b \cap S_c$$
and the lemma follows.
\end{proof}

Let $S_{<a} = \bigcup_{b<a} S_b$ and $S_{\leq a} = \bigcup_{b \leq a} S_b$.

\begin{lemma} \label{reduction_for_intersection}
For any $a\in\mB_k$ we have
\begin{eqnarray}
\label{eqintersectionscells}
S_{<a} \cap S_a &=& \bigcup_{b \rightarrow a} (S_b \cap S_a).
\end{eqnarray}
\end{lemma}
\begin{proof}
The right hand side, say $R$, of \eqref{eqintersectionscells} is by definition contained in the left hand side $L$. For the other inclusion, $L\subset R$, we have to show that $S_b\cap S_a\in R$  for $b < a$, $b\not=a$. Using Lemma \ref{lower_bound_exists} (and its proof) we can chose $c$ such that $d(a,c)+d(c,b)=d(a,b)$ and $c\rightarrow a$
and thus $(S_a \cap S_b) \subset (S_a \cap S_b \cap S_c)$ by Lemma \ref{distance_and_intersection}. Hence $(S_a\cap S_b)\subset (S_a \cap S_c)$ and the claim follows.
\end{proof}

The cell decompositions behave well under pairwise intersections:

\begin{prop} \label{intersection_and_cells}
Let $a,b \in \mB_k$ and $b \rightarrow a$. Then the following holds:
\begin{enumerate}[(i)]
\item If $b \rightarrow a$ is of type \textbf{I)}-\textbf{IV)} then, by
definition of $\Gamma_a$, the local move determines a unique edge $e \in
\cE(\Gamma_a)$ and we have
$$S_a \cap S_b = \bigcup_{J \subset \cR(\Gamma_a) \cup \cE(\Gamma_a), e
\in J} C_J.$$
\item If $b \rightarrow a$ is of type \textbf{I')}-\textbf{IV')}, there
exists a unique cup $\alpha \in {\rm cups}(a)$ such that $\alpha \notin
{\rm cups}(b)$. Furthermore $\alpha \in \cR(\Gamma_a)$ and
$$S_a \cap S_b = \bigcup_{J \subset \cR(\Gamma_a) \cup \cE(\Gamma_a),
\alpha \in J} C_J.$$
\end{enumerate}
\end{prop}

\begin{proof}
This follows directly from the moves \eqref{moves}: in (i), the cups involved in the move determine a new edge in $\Gamma_a$; in case (ii), the move creates a new cup which is outer, hence contained in $\cR(\Gamma_a)$. Now any point in $S_a \cap S_b$ is contained in one of the cells on the right
by construction and (C1)-(C4). Conversely, points from other cells are not contained in the intersection.
\end{proof}

We call the unique cup in $\cR(\Gamma_a)$ obtained from some fixed $b\rightarrow a$ via the construction from Proposition \ref{intersection_and_cells} (ii) \emph {special}. Let $\cR(\Gamma_a)^{\rm sp}$ be the set of special cups in $a$.

\begin{remark}{\rm
For $a \in \mB_k$ with $k$ odd, the moves \textbf{I')}-\textbf{IV')} in \eqref{moves} imply that all outer cups are special if the ray involved in $a$ is dotted whereas only outer cups to the right of the ray are special if the ray is undotted.
}\end{remark}

\begin{corollary} \label{decreasing_intersection_and_cells}
Let $a \in \mB_k$.
\begin{enumerate}[(i)]
\item If $k$ is even then
\begin{eqnarray*}
S_{<a} \cap S_a&= &\underset{\small \begin{array}{c} J \subset \cR(\Gamma_a) \cup \cE(\Gamma_a) \\ J \cap \cE(\Gamma_a) \neq \emptyset \end{array}}{\bigcup} C_J.
\end{eqnarray*}
\item If $k$ is odd then
\begin{eqnarray*}
S_{<a} \cap S_a&=& \underset{\small \begin{array}{c} J \subset \cR(\Gamma_a) \cup \cE(\Gamma_a)  \\ J \cap (\cE(\Gamma_a) \cup \cR(\Gamma_a)^{\rm sp})\neq \emptyset \end{array}}{\bigcup} C_J.
\end{eqnarray*}
\end{enumerate}
\end{corollary}

\begin{proof}
This follows directly from the definition of $S_{<a}$ from Proposition \ref{intersection_and_cells} using the definition of special cups from the moves \eqref{moves}.
\end{proof}

\subsection{Proof of Theorem \ref{tau_is_iso} and combinatorial dimension formula}
The following cohomology vanishing is completely analogous to \cite{Khovanov_Springer}.

\begin{prop}
\label{kunterbunt}
Let $a,b\in \mB_k$.
\begin{enumerate}
\item  The subset $S_{<a} \cap S_a$ has cohomology in even degrees only. The map $H(S_a) \longrightarrow H(S_{<a} \cap S_a)$, induced by the inclusion, is surjective.
    \label{one}
\item The homomorphism
$$H(S_{<a} \cap S_a) \longrightarrow \bigoplus_{b < a} H(S_b \cap S_a) $$
induced by the inclusion is injective.
\label{two}
\item
\label{three}
The subset $S_{\leq a}$ has cohomology in even degrees only. The Mayer-Vietoris sequence for $(S_{<a},S_a)$ breaks down into short exact sequences
\begin{eqnarray*}
&\quad 0 \rightarrow H^{2l}(S_{\leq a}) \rightarrow \begin{array}{c}
 H^{2l}(S_{<a})\\ \oplus \\H^{2l}(S_a)
 \end{array}
 \rightarrow H^{2l}(S_{<a} \cap S_a) \rightarrow 0,&
 \end{eqnarray*}
for $0 \leq l \leq k$.
\end{enumerate}
\end{prop}

\begin{proof} For the proof of part \eqref{one} we refer to {\cite[Proposition 4]{Khovanov_Springer}}, for part \eqref{two} see {\cite[Lemma 5]{Khovanov_Springer}} and for part \eqref{three} see {\cite[Proposition 3]{Khovanov_Springer}}. In all cases the arguments directly apply to our situation.
\end{proof}

Now we are prepared to give the proof of Theorem \ref{tau_is_iso}.
\begin{proof}[Proof of Theorem \ref{tau_is_iso}]
From Proposition \ref{kunterbunt} we deduce as in \cite[Pro\-po\-si\-tion 3]{Khovanov_Springer} that the following sequence is exact
$$ 0 \longrightarrow H(S_{\leq a}) \overset{\phi}{\longrightarrow} \bigoplus_{b \leq a} H(S_{b}) \overset{\psi^-}{\longrightarrow} \bigoplus_{b < c \leq a} H(S_b \cap S_c),$$
where $\phi$ is induced by the inclusion $S_b \subset S_{\leq a}$, and
$ \phi^- := \sum_{b<c \leq a}(\phi_{b,c} - \phi_{c,b}),$
where $\phi_{b,c}:H(S_b) \longrightarrow H(S_b \cap S_c)$,
again induced by the corresponding inclusion.
When $a$ is chosen maximal in $\mB_k^{\rm even}$ we obtain that
$$ 0 \longrightarrow H(\widetilde{S}^{\rm even}) \longrightarrow \bigoplus_{b \in \mB_k^{\rm even}}H(S_b) \longrightarrow \bigoplus_{b,c \in \mB_k^{\rm even}, b < c}H(S_b \cap S_c)$$
is exact; analogously for $\widetilde{S}^{\rm odd}$. This is equivalent to Theorem~\ref{tau_is_iso}.
\end{proof}

The dimension of $H(\widetilde{S})$ can be computed now purely combinatorially.

\begin{prop}
\label{dimensionformula}
The following hold
\begin{eqnarray*}
&{\rm dim} \, H(\widetilde{S}) = 2\, {\rm dim} \, H(\widetilde{S}^{\rm even}) = 2\, {\rm dim} \, H(\widetilde{S}^{\rm odd}) = 2^{k}.&
\end{eqnarray*}
\end{prop}

\begin{proof}
We have the cell decomposition on $S_a\cap S_{<a}$ by Corollary \ref{intersection_and_cells}. To determine the dimension of the cohomology
it is enough to count for each $a\in \mB_k$ the cells contained in $S_a$ but not in any intersection $S_a \cap S_b$ with $b < a$ by considering the induced cell partition, see \cite[Lemma 6]{Khovanov_Springer}, and add all of these numbers up.

We start with the case of even $k$. For fixed $a\in\mB_k$, the  number of cells that do not lie in an intersection of the form $S_a \cap S_b$ for $b < a$ is equal to $2^{{\rm outer}(a)}$ by Lemma \ref{intersection_and_cells}, where ${\rm outer}(a)$ denotes the number of outer cups of $a$. To count these we proceed by induction. The case $k=2$ is obvious.
For arbitrary even $k$, consider the cup $C$ in $a$ which contains the
last vertex.
Say it connects vertex $k$ with vertex $k-(2j-1)$ (with $1 \leq j \leq
m=k/2$).
If this cup is dotted then it is the only outer cup of the diagram and the
number of cells to be counted for such $a$ is just twice the number of
diagrams with a cup of this form, which gives
\begin{equation*}
2 \left[ \binom{2(m-j)}{m-j} \cdot \frac{1}{j} \binom{2(j-1)}{j-1} \right],
\end{equation*}
where the second factor is the Catalan number $\mathcal{C}_{j-1}$ counting the,
necessarily undecorated, diagrams inside our fixed cup $C$ of width
$2(j-1)$, while the first factor is the number
of diagrams to the left of the fixed cup, see Remark~\ref{counting}.
We claim that the sum over $j=0,\ldots, m$ of these terms gives $\tbinom{2m}{m}$. Indeed using induction hypothesis in \eqref{indhyp}, the equality $\binom{2(m-j)}{m-j} = (4-\frac{2}{m-j}) \binom{2(m-1-j)}{m-1-j}$ in \eqref{komisch}, and the Segner recurrence $\sum_{j=0}^{m-2} \mathcal{C}_j\mathcal{C}_{m-2-j}=\mathcal{C}_{m-1}$ in \eqref{Segre} we obtain

\begin{eqnarray}
&&\sum_{j=1}^{m} 2 \left[\tbinom{2(m-j)}{m-j} \cdot \tfrac{1}{j}
\tbinom{2(j-1)}{j-1} \right]
= \tfrac{2}{m} \tbinom{2(m-1)}{m-1} + \sum_{j=1}^{m-1} 2
\left[\tbinom{2(m-j)}{m-j} \cdot \tfrac{1}{j} \tbinom{2(j-1)}{j-1} \right]\nonumber\\
&=& \tfrac{2}{m} \tbinom{2(m-1)}{m-1} + 4 \sum_{j=1}^{m-1} 2
\left[\tbinom{2(m-1-j)}{m-1-j} \cdot \tfrac{1}{j} \tbinom{2(j-1)}{j-1}
\right] \label{komisch}\\
&&\qquad - 4 \sum_{j=0}^{m-2} \left[ \tfrac{1}{j+1} \tbinom{2j}{j}
\tfrac{1}{m-1-j} \tbinom{2(m-2-j)}{m-2-j} \right] \nonumber\\
&{=}& \tfrac{2}{m} \tbinom{2(m-1)}{m-1} + 4
\tbinom{2(m-1)}{m-1} - 4 \sum_{j=0}^{m-2} \left[ \tfrac{1}{j+1} \tbinom{2j}{j}
\tfrac{1}{m-1-j} \tbinom{2(m-2-j)}{m-2-j} \right] \label{indhyp}\\
&{=}& \tfrac{2}{m}
\tbinom{2(m-1)}{m-1} + 4 \tbinom{2(m-1)}{m-1} - \frac{4}{m}
\tbinom{2(m-1)}{m-1} \label{Segre}\\
&=& \frac{4m-2}{m}\tbinom{2(m-1)}{m-1} = \tbinom{2m}{m}\nonumber
\end{eqnarray}
If the cup $C$ is undotted it is an outer cup, but the outer cups to the
left still contribute to the number of cells, giving us
$$
2 \left[ 2^{k-2j} \cdot \frac{1}{j} \binom{2(j-1)}{j-1} \right],
$$
where we have again the number of diagrams inside our fixed cup $C$
times the contribution from the part to the left of $C$, which is
$2^{k-2j}$ by induction. Altogether we get
\begin{eqnarray*}
&&  \sum_{j=1}^{m} 2 \left[ \tbinom{2(m-j)}{m-j} \cdot \tfrac{1}{j}
\tbinom{2(j-1)}{j-1}\right] + \sum_{j=1}^{m} 2 \left[ 2^{2m-2j}
\tfrac{1}{j} \tbinom{2(j-1)}{j-1} \right]\\
&=& \tbinom{2m}{m} + \tfrac{2}{m} \tbinom{2(m-1)}{m-1} +
4 \sum_{j=1}^{m-1} 2 \left[ 2^{2(m-1)-2j} \tfrac{1}{j}
\tbinom{2(j-1)}{j-1} \right] \\
&\underset{\text{(Induction)}}{=}& \tbinom{2m}{m} + \tfrac{2}{m}
\tbinom{2(m-1)}{m-1} + 4 \left[2^{2(m-1)} - \tbinom{2(m-1)}{m-1}\right] = 2^{2m} = 2^k.
\end{eqnarray*}

Now let $k=2m+1$ be odd. We distinguish the case where the ray has a dot from the case where it
has not. In the first case the diagram only contributes a single
cell since all outer cups are special, thus we only have to count the
possible diagrams when the dotted ray is at position $k-2j$ for a fixed
$0 \leq j \leq m$ giving us
$$
\binom{2(m-j)}{m-j} \cdot \frac{1}{j+1} \binom{2j}{j},
$$
with the first factor being the number of possible decorated cup
diagrams to the left of the ray and the second factor the number of
undecorated cup diagrams to the right. If instead the ray is undotted, then only outer cups to the
right of the ray are special and we have to count the cells arising from the outer cups to the left, which gives us $2^{2(m-j)} \cdot \frac{1}{j+1} \binom{2j}{j}$.

Hence it is enough to verify
$\sum_{j=0}^m \left[ \tbinom{2(m-j)}{m-j} + 2^{2(m-j)} \right] \cdot
\tfrac{1}{j+1} \tbinom{2j}{j}=2^{2m+1}$
which is clear for $m=0$. The left hand side however equals
\begin{eqnarray*}
&&\tfrac{2}{m+1} \tbinom{2m}{m} + \sum_{j=0}^{m-1} \left[
\tbinom{2(m-j)}{m-j} + 2^{2(m-j)} \right] \cdot \tfrac{1}{j+1}
\tbinom{2j}{j} \\
&=& \tfrac{2}{m+1} \tbinom{2m}{m} + \sum_{j=0}^{m-1} \left[
\tbinom{2(m-j-1)}{m-j-1} + 2^{2(m-j-1)} \right] \cdot \tfrac{1}{j+1}
\tbinom{2j}{j} \\
&& \qquad + \sum_{j=0}^{m-1} \left[ \tbinom{2(m-j)}{m-j} -
\tbinom{2(m-j-1)}{m-j-1} + 3 \cdot 2^{2(m-j-1)} \right]  \cdot
\tfrac{1}{j+1} \tbinom{2j}{j}
\end{eqnarray*}
\begin{eqnarray*}
&\underset{\text{(Induction)}}{=}& \tfrac{2}{m+1} \tbinom{2m}{m} + 2^{2m-1}\\
&& \qquad + \sum_{j=0}^{m-1} \left[ \tfrac{3(m-j)-2}{m-j}
\tbinom{2(m-j-1)}{m-j-1} + 3 \cdot 2^{2(m-j-1)} \right]  \cdot
\tfrac{1}{j+1} \tbinom{2j}{j}\\
&{=}& \tfrac{2}{m+1} \tbinom{2m}{m} + 2^{2m-1} + 3 \sum_{j=0}^{m-1}
\left[\tbinom{2(m-j-1)}{m-j-1} + 2^{2(m-j-1)} \right]  \cdot
\tfrac{1}{j+1} \tbinom{2j}{j} \\
&& \qquad -2 \sum_{j=0}^{m-1} \tfrac{1}{m-j} \tbinom{2(m-j-1)}{m-j-1}
\cdot \tfrac{1}{j+1} \tbinom{2j}{j}
\end{eqnarray*}
\begin{eqnarray*}
\\
&\underset{\text{(Induction)}}{=}& \tfrac{2}{m+1} \tbinom{2m}{m} +
2^{2m-1} + 3 \cdot 2^{2m-1} -2 \sum_{j=0}^{m-1} \tfrac{1}{m-j}
\tbinom{2(m-j-1)}{m-j-1} \cdot \tfrac{1}{j+1} \tbinom{2j}{j}\\
&\underset{\text{(Segner rec.)}}{=}& \tfrac{2}{m+1} \tbinom{2m}{m} +
2^{2m+1} - \tfrac{2}{m+1} \tbinom{2m}{m} = 2^{2m+1}\\
\end{eqnarray*}
and so the claim holds. Hence the lemma is proved.
\end{proof}

\section{Domino Tableaux and combinatorial bijections}
\label{Section5}
In the following let $\lambda$ be a partition of $2k$, viewed as a Young diagram of shape $\lambda$ (i.e., the parts of $\lambda$ are the length of the rows of the diagram). We will picture them using "English notation".

\subsection{(Signed admissible) domino tableaux and clusters}
We now introduce the combinatorics of domino diagrams and tableaux, following \cite{vanLeeuwen}, and connect them afterwards with the geometry of Springer fibres and our cup diagram combinatorics. The cup diagrams provide then an elementary description of the geometry, in particular pairwise intersections of components.

\begin{definition}
A \emph{domino diagram} $T$ with $k$ dominoes of shape $\lambda$ is a Young diagram $\overline{T}$ of shape $\lambda$ together with a partitioning of the boxes into subsets of order 2, called \emph{dominoes}, such that the paired boxes share a common edge (vertically or horizontally). The set of all domino diagrams with shape $\lambda$ will be denoted by $DY(\lambda)$.
\end{definition}

When picturing a domino diagram we omit the common edge of a pair:

\begin{ex}
Here is an example of a domino diagram of shape $\lambda = (4,3,1)$, the left diagram shows the 2-subsets, which are then turned into dominoes in the right diagram:
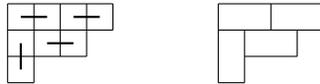
\begin{eqnarray*}
\begin{tikzpicture}[scale=0.7]
\draw (0,0)  -- +(0,-1.5);
\draw (.5,-1)  -- +(0,-.5);
\draw (1.5,-.5)  -- +(0,-.5);
\draw (2,0)  -- +(0,-.5);

\draw (0,0)  -- +(2,0);
\draw (1.5,-.5)  -- +(.5,0);
\draw (.5,-1)  -- +(1,0);
\draw (0,-1.5)  -- +(.5,0);

\draw (0,-.5)  -- +(1.5,0);
\draw (1,0)  -- +(0,-.5);
\draw (.5,-.5)  -- +(0,-.5);
\draw (.5,0)  -- +(0,-.5);
\draw (1.5,0)  -- +(0,-.5);
\draw (1,-.5)  -- +(0,-.5);
\draw (0,-1)  -- +(.5,0);

\draw[thick] (.25,-.25)  -- +(.5,0);
\draw[thick] (1.25,-.25)  -- +(.5,0);
\draw[thick] (.75,-.75)  -- +(.5,0);
\draw[thick] (.25,-.75)  -- +(0,-.5);

\begin{scope}[xshift=4cm]
\draw (0,0)  -- +(0,-1.5);
\draw (.5,-1)  -- +(0,-.5);
\draw (1.5,-.5)  -- +(0,-.5);
\draw (2,0)  -- +(0,-.5);

\draw (0,0)  -- +(2,0);
\draw (1.5,-.5)  -- +(.5,0);
\draw (.5,-1)  -- +(1,0);
\draw (0,-1.5)  -- +(.5,0);

\draw (0,-.5)  -- +(1.5,0);
\draw (1,0)  -- +(0,-.5);
\draw (.5,-.5)  -- +(0,-.5);

\end{scope}
\end{tikzpicture}
\end{eqnarray*}
\end{ex}

Recall from Section \ref{Section2} the notion of admissible partition which was used to label conjugacy classes of nilpotent elements. We now define admissible domino tableaux:

\begin{definition}
An \emph{admissible domino tableau} $t$ is a domino diagram from $DY(\lambda)$ together with a filling of the $2k$ boxes in the diagram with the integers $1, \ldots, k$ such that
\begin{enumerate}[(\text{ADT}1)]
\item Each integer occurs exactly twice and the two boxes it appears in form a domino.
\item The entries in each row and column are weakly increasing from left to right and top to bottom.
\item The sequence of partitions starting with $t$ and then obtained by successively eliminating the domino labelled with the largest integer consists only of admissible partitions as in Definition \ref{defpartition}.
\end{enumerate}
We denote the \emph{set of all admissible domino tableaux} of shape $\lambda$ by $ADT(\lambda)$. If we drop the last condition (ADT3) then we get more generally the set $DT(\la)$ of \emph{standard domino tableaux} of shape $\la$.
\end{definition}

We again abbreviate admissible domino tableaux by drawing dominoes (instead of pairs of boxes) with only one number per domino (instead of twice the same number if the corresponding two boxes form a pair).

\begin{ex} \label{admissible_domino_example}
There are two admissible domino tableaux of shape $(3,3)$,
\begin{eqnarray}
\label{ADTex}
\begin{tikzpicture}[scale=.7]
\node at (-3.5,.5) {$ADT((3,3))=$};
\node[yscale=1.7] at (-1.6,.5) {\{};
\node at (-.75,.5) {$t_1=$};
\draw (0,0) -- +(0,1);
\draw (1.5,0) -- +(0,1);
\draw (0,0) -- +(1.5,0);
\draw (0,1) -- +(1.5,0);
\node at (.25,.5) {\tiny $1$};
\node at (.75,.5) {\tiny $2$};
\node at (1.25,.5) {\tiny $3$};
\draw (1,0) -- +(0,1);
\draw (.5,0) -- +(0,1);
\node at (2.25,0) {$,$};
\begin{scope}[xshift=4cm]
\node at (-.75,.5) {$t_2=$};
\draw (0,0) -- +(0,1);
\draw (1.5,0) -- +(0,1);
\draw (0,0) -- +(1.5,0);
\draw (0,1) -- +(1.5,0);
\node at (.25,.5) {\tiny $1$};
\node at (1,.75) {\tiny $2$};
\node at (1,.25) {\tiny $3$};
\draw (.5,.5) -- +(1,0);
\draw (.5,0) -- +(0,1);
\end{scope}
\node[yscale=1.7] at (6.1,.5) {\}};
\end{tikzpicture}
\end{eqnarray}

Indeed, we first fill the Young diagram of shape $(3,3)$ with the integers $1,2,3$, each occurring twice such that the fillings are weakly increasing in rows and columns:
\begin{eqnarray*}
{\tiny\Yvcentermath1 t_1=\young(123,123) \quad \quad t_2=\young(122,133) \quad \quad t_3=\young(113,223) \quad \quad t_4=\young(112,233)}
\end{eqnarray*}
To satisfy (ADT1) we have to remove $t_4$, since the two boxes labelled $2$ do not form a domino. Since $t_3$ violates (ADT3) we are left with $t_1$ and $t_2$ which indeed satisfy (ADT1)-(ADT3).
\end{ex}

From now on we restrict ourself to partitions $\lambda = (\lambda_1,\lambda_2)$ with only 2 parts. Note that the definition of an admissible domino tableau forces horizontal dominoes to occur with the left box in an even column. In particular we must have a vertical domino in the first column (e.g. $t_3$ is not possible).

\begin{definition}
 Let $T \in ADT((\lambda_1,\lambda_2))$ be an admissible domino tableau and $p$ a domino in $T$. We define the \emph{type} of $p$ as
$${\rm type}(p) = \left\lbrace
\begin{array}{l}
{\rm V}^0 \text{ if } p \text{ is vertical and in an even column,}\\
{\rm V}^1 \text{ if } p \text{ is vertical and in an odd column,}\\
{\rm H} \text{ if } p \text{ is horizontal.}
\end{array}
\right. $$
\end{definition}

\begin{remark}{\rm
In \cite{vanLeeuwen} dominoes of type ${\rm V}^0,{\rm V}^1,\rm{H}$ are called of type ${\rm I}-,{\rm I}+,{\rm N}$ respectively.
}\end{remark}

 To get a labelling set for the irreducible components of the Springer fibre we need to introduce an additional data in form of a sign for the dominoes of type ${\rm V}^1$.

\begin{definition}
A \emph{signed domino tableau} is an admissible domino tableau with signs (i.e., an element in $\{+,-\}$) attached to each domino of type ${\rm V}^1$.
\end{definition}
The set of signed domino tableaux of shape $\lambda$ is denoted $ADT_{\rm sgn}(\lambda)$.

\begin{ex}
There are six signed domino tableaux of shape $(3,3)$ (with the underlying domino tableaux from Example \ref{admissible_domino_example}):
\begin{eqnarray*}
\begin{tikzpicture}[scale=0.7]
\node at (-1,.5) {$t_{1,++}=$};
\draw (0,0)  -- +(0,1);
\draw (1.5,0)  -- +(0,1);
\draw (0,0)  -- +(1.5,0);
\draw (0,1)  -- +(1.5,0);

\node at (.25,.75) {\tiny $1$};
\node at (.25,.25) {\tiny $+$};
\node at (.75,.5) {\tiny $2$};
\node at (1.25,.75) {\tiny $3$};
\node at (1.25,.25) {\tiny $+$};
\draw (1,0)  -- +(0,1);
\draw (.5,0)  -- +(0,1);

\begin{scope}[xshift=4cm]
\node at (-1,.5) {$t_{1,+-}=$};
\draw (0,0)  -- +(0,1);
\draw (1.5,0)  -- +(0,1);
\draw (0,0)  -- +(1.5,0);
\draw (0,1)  -- +(1.5,0);

\node at (.25,.75) {\tiny $1$};
\node at (.25,.25) {\tiny $+$};
\node at (.75,.5) {\tiny $2$};
\node at (1.25,.75) {\tiny $3$};
\node at (1.25,.25) {\tiny $-$};
\draw (1,0)  -- +(0,1);
\draw (.5,0)  -- +(0,1);
\end{scope}

\begin{scope}[xshift=8cm]
\node at (-1,.5) {$t_{1,-+}=$};
\draw (0,0)  -- +(0,1);
\draw (1.5,0)  -- +(0,1);
\draw (0,0)  -- +(1.5,0);
\draw (0,1)  -- +(1.5,0);

\node at (.25,.75) {\tiny $1$};
\node at (.25,.25) {\tiny $-$};
\node at (.75,.5) {\tiny $2$};
\node at (1.25,.75) {\tiny $3$};
\node at (1.25,.25) {\tiny $+$};
\draw (1,0)  -- +(0,1);
\draw (.5,0)  -- +(0,1);
\end{scope}

\begin{scope}[xshift=12cm]
\node at (-1,.5) {$t_{1,--}=$};
\draw (0,0)  -- +(0,1);
\draw (1.5,0)  -- +(0,1);
\draw (0,0)  -- +(1.5,0);
\draw (0,1)  -- +(1.5,0);

\node at (.25,.75) {\tiny $1$};
\node at (.25,.25) {\tiny $-$};
\node at (.75,.5) {\tiny $2$};
\node at (1.25,.75) {\tiny $3$};
\node at (1.25,.25) {\tiny $-$};
\draw (1,0)  -- +(0,1);
\draw (.5,0)  -- +(0,1);
\end{scope}

\begin{scope}[yshift=-1.5cm,xshift=4cm]
\node at (-1,.5) {$t_{2,+}=$};
\draw (0,0)  -- +(0,1);
\draw (1.5,0)  -- +(0,1);
\draw (0,0)  -- +(1.5,0);
\draw (0,1)  -- +(1.5,0);

\node at (.25,.75) {\tiny $1$};
\node at (.25,.25) {\tiny $+$};
\node at (1,.75) {\tiny $2$};
\node at (1,.25) {\tiny $3$};
\draw (.5,.5)  -- +(1,0);
\draw (.5,0)  -- +(0,1);
\end{scope}

\begin{scope}[yshift=-1.5cm,xshift=8cm]
\node at (-1,.5) {$t_{2,-}=$};
\draw (0,0)  -- +(0,1);
\draw (1.5,0)  -- +(0,1);
\draw (0,0)  -- +(1.5,0);
\draw (0,1)  -- +(1.5,0);

\node at (.25,.75) {\tiny $1$};
\node at (.25,.25) {\tiny $-$};
\node at (1,.75) {\tiny $2$};
\node at (1,.25) {\tiny $3$};
\draw (.5,.5)  -- +(1,0);
\draw (.5,0)  -- +(0,1);
\end{scope}
\end{tikzpicture}
\end{eqnarray*}
\end{ex}

The admissibility condition (ADT3) implies that dominoes of type $H$ have their left part in an even column which means that our tableaux are built up from two basic building blocks, called clusters.

A \emph{closed cluster} is a connected part of an admissible domino tableau containing all dominoes between a domino of type ${\rm V}^1$ and the next domino to its right of type ${\rm V}^0$, including the two vertical dominoes.
Hence after removing the filling, it is a part of the underlying Young diagram  of the following form (starting at an odd column with an even number of dominoes of type ${\rm H}$ arranged in two rows of the same length between the two vertical dominoes)
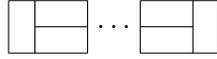
\begin{center}
\begin{tikzpicture}[scale=0.7]
\draw (0,0)  -- +(0,1);
\draw (1.5,0)  -- +(0,1);
\draw (0,0)  -- +(1.5,0);
\draw (0,1)  -- +(1.5,0);

\draw (.5,.5)  -- +(1,0);
\draw (.5,0)  -- +(0,1);
\node at (2,.5) {$\ldots$};

\draw (2.5,0)  -- +(0,1);
\draw (4,0)  -- +(0,1);
\draw (2.5,0)  -- +(1.5,0);
\draw (2.5,1)  -- +(1.5,0);

\draw (2.5,.5)  -- +(1,0);
\draw (3.5,0)  -- +(0,1);
\end{tikzpicture}
\end{center}
An \emph{open cluster} is a connected part of an admissible domino tableau given by a domino $D$ of type ${\rm V}^1$ with no domino of type ${\rm V}^0$ to its right and all dominoes of type ${\rm H}$ to the right, i.e. a diagram of the form
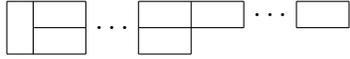
\begin{center}
\begin{tikzpicture}[scale=0.7]
\draw (0,0)  -- +(0,1);
\draw (1.5,0)  -- +(0,1);
\draw (0,0)  -- +(1.5,0);
\draw (0,1)  -- +(1.5,0);

\draw (.5,.5)  -- +(1,0);
\draw (.5,0)  -- +(0,1);
\node at (2,.5) {$\ldots$};

\draw (2.5,0)  -- +(0,1);
\draw (2.5,0)  -- +(1,0);
\draw (2.5,1)  -- +(2,0);

\draw (2.5,.5)  -- +(2,0);
\draw (3.5,0)  -- +(0,1);

\draw (4.5,.5)  -- +(0,.5);

\node at (5,.75) {$\ldots$};

\draw (5.5,.5)  -- +(0,.5);
\draw (6.5,.5)  -- +(0,.5);
\draw (5.5,1)  -- +(1,0);
\draw (5.5,.5)  -- +(1,0);
\end{tikzpicture}.
\end{center}

Clearly every (signed) admissible standard domino tableau decomposes uniquely into a finite disjoint union of closed clusters and possibly one extra open cluster. Such an open cluster exists if and only if the shape is $(\la_1,\la_2)$ with $\la_1$ (and then also $\la_2$) odd.


\begin{ex}
The following admissible domino tableau
\begin{eqnarray*}
\begin{tikzpicture}[scale=0.7]
\node at (-.75,.5) {$t=$};
\draw (0,0)  -- +(0,1);
\draw (8.5,0)  -- +(0,1);
\draw (10.5,.5)  -- +(0,.5);
\draw (8.5,.5)  -- +(2,0);
\draw (0,0)  -- +(8.5,0);
\draw (0,1)  -- +(10.5,0);

\node at (.25,.5) {\tiny $1$};
\node at (1,.75) {\tiny $2$};
\node at (1,.25) {\tiny $4$};
\node at (2,.75) {\tiny $3$};
\node at (2,.25) {\tiny $5$};
\node at (2.75,.5) {\tiny $6$};
\node at (3.25,.5) {\tiny $7$};
\node at (3.75,.5) {\tiny $8$};
\node at (4.25,.5) {\tiny $9$};
\node at (5,.75) {\tiny $10$};
\node at (5,.25) {\tiny $11$};
\node at (5.75,.5) {\tiny $12$};
\node at (6.25,.5) {\tiny $13$};
\node at (7,.75) {\tiny $14$};
\node at (7,.25) {\tiny $15$};
\node at (8,.75) {\tiny $16$};
\node at (8,.25) {\tiny $17$};
\node at (9,.75) {\tiny $18$};
\node at (10,.75) {\tiny $19$};
\draw (.5,0)  -- +(0,1);
\draw (.5,.5)  -- +(1,0);
\draw (1.5,0)  -- +(0,1);
\draw (1.5,.5)  -- +(1,0);
\draw (2.5,0)  -- +(0,1);
\draw (3,0)  -- +(0,1);
\draw (3.5,0)  -- +(0,1);
\draw (4,0)  -- +(0,1);
\draw (4.5,0)  -- +(0,1);
\draw (4.5,.5)  -- +(1,0);
\draw (5.5,0)  -- +(0,1);
\draw (6,0)  -- +(0,1);
\draw (6.5,0)  -- +(0,1);
\draw (6.5,.5)  -- +(2,0);
\draw (7.5,0)  -- +(0,1);
\draw (9.5,.5)  -- +(0,.5);

\end{tikzpicture},
\end{eqnarray*}
has three closed clusters $C_1(t)=\{1,2,3,4,5,6\}$, $C_2(t)=\{7,8\}$, and $C_3(t)=\{9,10,11,12\}$ and one open cluster $C_4(t)=\{13,14,15,16,17,18,19\}$.
\end{ex}

Since by definition each of the clusters includes exactly one domino of type ${\rm V}^1$ we can speak of the \emph{sign of a cluster} for a signed domino tableau.

\begin{definition}
\label{orbitalvarieties}
Let $t,t' \in ADT_{\rm sgn}(\lambda)$ be two signed domino tableaux that coincide if we forget the signs. Then $t \thicksim_{\rm cl} t'$ if the signs of all closed clusters of $t$ and $t'$ coincide. This is obviously an equivalence relation. We denote the equivalence classes for $\thicksim_{\rm cl}$ by $ADT_{\rm sgn,cl}(\lambda)$.
\end{definition}

Since in our case there is at most one open cluster for a given admissible domino tableau, the equivalence classes consist of either one element, if there is no open cluster, or of two elements, if there is an open cluster.

\subsection{Combinatorial bijections}
Given a cluster $\cC$ in a (signed) domino tableau we call the part consisting of the dominoes of type ${\rm H}$ the \emph{standard tableau part} of $\cC$, since viewing the dominoes as ordinary boxes of a tableau we have in fact a standard tableau $T(\cC)$. Recall the well-known bijection
\begin{eqnarray}
\label{ordinary}
&\{\text{(ordinary) standard tableaux of shape $(r,s)$} \}&\nonumber\\
&\oneone&\\
&\{\text{undecorated cup diagrams on $r+s$ vertices with $s$ cups}\nonumber\}&
\end{eqnarray}
sending a standard tableau $T$ to the unique cup diagram $C(T)$ where the $s$ numbers appearing in the bottom row mark the right endpoints of the cups and the number in the top row are the left endpoint of cups or endpoints of rays, see e.g. \cite[Proposition 3]{SW}.

\begin{ex}
\label{ex1} The following five standard tableaux $T$
$$\young(125,34)\hspace{1.2cm}\young(135,24)\hspace{1.2cm}
\young(134,25)\hspace{1.2cm} \young(124,35)\hspace{1.2cm}\young(123,45)$$
correspond to the undecorated cup diagrams $C(T)$:

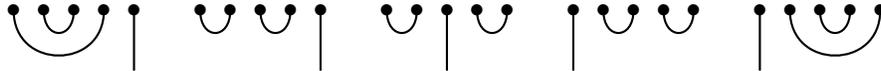
\begin{tikzpicture}[thick, scale=0.8]
\node at (0,0) {$\bullet$};
\node at (.5,0) {$\bullet$};
\node at (1,0) {$\bullet$};
\node at (1.5,0) {$\bullet$};
\node at (2,0) {$\bullet$};
\draw (0,0) .. controls +(0,-1) and +(0,-1) .. +(1.5,0);
\draw (.5,0) .. controls +(0,-.5) and +(0,-.5) .. +(.5,0);
\draw (2,0) -- +(0,-1);

\begin{scope}[xshift=3.1cm]
\node at (0,0) {$\bullet$};
\node at (.5,0) {$\bullet$};
\node at (1,0) {$\bullet$};
\node at (1.5,0) {$\bullet$};
\node at (2,0) {$\bullet$};
\draw (0,0) .. controls +(0,-.5) and +(0,-.5) .. +(.5,0);
\draw (1,0) .. controls +(0,-.5) and +(0,-.5) .. +(.5,0);
\draw (2,0) -- +(0,-1);
\end{scope}

\begin{scope}[xshift=6.2cm]
\node at (0,0) {$\bullet$};
\node at (.5,0) {$\bullet$};
\node at (1,0) {$\bullet$};
\node at (1.5,0) {$\bullet$};
\node at (2,0) {$\bullet$};
\draw (0,0) .. controls +(0,-.5) and +(0,-.5) .. +(.5,0);
\draw (1,0) -- +(0,-1);
\draw (1.5,0) .. controls +(0,-.5) and +(0,-.5) .. +(.5,0);
\end{scope}

\begin{scope}[xshift=9.3cm]
\node at (0,0) {$\bullet$};
\node at (.5,0) {$\bullet$};
\node at (1,0) {$\bullet$};
\node at (1.5,0) {$\bullet$};
\node at (2,0) {$\bullet$};
\draw (0,0) -- +(0,-1);
\draw (.5,0) .. controls +(0,-.5) and +(0,-.5) .. +(.5,0);
\draw (1.5,0) .. controls +(0,-.5) and +(0,-.5) .. +(.5,0);
\end{scope}

\begin{scope}[xshift=12.4cm]
\node at (0,0) {$\bullet$};
\node at (.5,0) {$\bullet$};
\node at (1,0) {$\bullet$};
\node at (1.5,0) {$\bullet$};
\node at (2,0) {$\bullet$};
\draw (0,0) -- +(0,-1);
\draw (.5,0) .. controls +(0,-1) and +(0,-1) .. +(1.5,0);
\draw (1,0) .. controls +(0,-.5) and +(0,-.5) .. +(.5,0);
\end{scope}
\end{tikzpicture}
\end{ex}

We extend the bijection \eqref{ordinary} and assign to a signed closed cluster $\cC$ with standard tableau part containing $2m$ dominoes of type ${\rm H}$ a cup diagram on $2m+2$ vertices with $m+1$ cups by adding an additional vertex to the left and right of $C(T(\cC))$ connected by a cup. This cup is dotted if the sign of $\cC$ is $-$ and undotted if the sign is $+$. If the cluster $\cC$ is open we again take the cup diagram $C(T(\cC))$ associated with its standard tableau part and add an additional vertex with a ray to the left of $C(T(\cC))$. This ray is dotted if the sign of $\cC$ is $-$ and undotted if the sign is $+$. Since every signed standard tableau is a concatenation of closed and possibly one open cluster, our assignment

\usetikzlibrary{arrows}

\begin{tikzpicture}[thick,font=\scriptsize,scale=.45]

\node at (.85,1.4) {closed};
\draw (0,0) rectangle +(.5,1);
\draw (.5,.5) rectangle +(1,.5);
\draw (.5,0) rectangle +(1,.5);
\draw (1.5,.5) rectangle +(1,.5);
\draw (1.5,0) rectangle +(1,.5);
\node at (3,.5) {$\dots$};
\draw (3.5,.5) rectangle +(1,.5);
\draw (3.5,0) rectangle +(1,.5);
\draw (4.5,0) rectangle +(.5,1);
\draw[|->] (5.3,.5) -- +(1.5,0);

\node[xscale=2,yscale=4] at (7.5,.5) {\{};
\begin{scope}[xshift=8.5cm,yshift=.8cm]
\draw (0,.9) .. controls +(0,-1) and +(0,-1) .. +(2,0);
\node at (1.05,.65) {$C(T)$};
\fill (1,.15) circle(4pt);
\node at (3,.5) {if $-$};
\end{scope}
\begin{scope}[xshift=8.5cm,yshift=-.8cm]
\draw (0,.9) .. controls +(0,-1) and +(0,-1) .. +(2,0);
\node at (1.05,.65) {$C(T)$};
\node at (3,.5) {if +};
\end{scope}

\begin{scope}[xshift=14cm,yshift=2.9cm]
\node at (.65,-1.7) {open};
\draw (0,-3) rectangle +(.5,1);
\draw (.5,-2.5) rectangle +(1,.5);
\draw (.5,-3) rectangle +(1,.5);
\draw (1.5,-2.5) rectangle +(1,.5);
\draw (2.5,-2.5) rectangle +(1,.5);
\draw[|->] (4.3,-2.5) -- +(1.5,0);
\node[xscale=2,yscale=4] at (6.5,-2.5) {\{};
\begin{scope}[xshift=7.5cm,yshift=-2.2cm]
\draw (0,0) -- (0,1);
\fill (0,.5) circle(4pt);
\node at (1,.5) {$c(T)$};
\node at (3,.5) {if $-$};
\end{scope}
\begin{scope}[xshift=7.5cm,yshift=-3.8cm]
\draw (0,0) -- (0,1);
\node at (1,.5) {$c(T)$};
\node at (3,.5) {if +};
\end{scope}

\end{scope}
\end{tikzpicture}

extends to a map $F$ which assigns to each admissible tableau a cup diagram
obtained by putting the cup diagrams associated to the single clusters next to each other (in the same order as the clusters).

\begin{lemma}
\label{Bijection}
Let $r,s$ be natural numbers, such that $(r,s)$ is an admissible shape. The assignment $F$ defines a bijection
\begin{eqnarray*}
ADT_{\rm sgn}((r,s))
&\oneone&\left\{ \text{cup diagrams on $\frac{r+s}{2}$ vertices with $\left\lfloor{ \frac{s}{2}}\right\rfloor$ cups}\right\}
\end{eqnarray*}
\end{lemma}

\begin{proof}
The map is indeed well-defined, since no closed cluster appears to the right of an open cluster thus no dotted cup to the right of a ray and of course no dotted cup nested inside another cup. The inverse map is given as follows: take a cup diagram $a$ of the above form and first ignore all dots. Then create a domino tableau with vertical dominoes exactly in the rows labelling the endpoints of outer cups to the left of all the rays in $a$ and the leftmost ray in $a$. Fill them with their column number and add additionally a sign in case the column number is odd. This sign is $+$ if the corresponding cup/ray was originally undotted and $-$ in case it was dotted. Then forget these outer cups and the leftmost ray and add dominoes of type ${\rm H}$ in between and to the right of these vertical dominoes to get the correct shape. Turn them into standard tableaux parts of the clusters using the bijection \eqref{ordinary}. For the filling the numbers labelling the rows between the corresponding two vertical dominoes or to the right of the rightmost vertical dominoes should be used (so the bijection \eqref{ordinary} has to be slightly adjusted by shifting the numbers for the filling by the column number of the next vertical domino to the left). Since inside an outer cup or to the right of the leftmost ray we have only undotted cups, the assignment makes sense. We omit the straightforward calculations to verify that this is indeed the inverse map.
\end{proof}

\begin{corollary}
\label{signeddominoestocups}
The bijection from Lemma \ref{Bijection} restricts to bijections
\begin{eqnarray*}
&\{T\in ADT_{\rm sgn}((r,s))\mid \text{the number of $-$ signs is even (resp. odd)}\}&\\
&\oneone&\\
&\left\{
\begin{array}{c}
\text{cup diagrams on $\frac{r+s}{2}$ vertices with $\lfloor \frac{s}{2}\rfloor$ cups} \\
\text{and an even (resp. odd) number of dots}
\end{array}
\right\}&
\end{eqnarray*}
and gives in the special case $r=s=k$ the following two bijections
\begin{eqnarray*}
\{T\in ADT_{\rm sgn}((k,k))\mid \text{the number of $-$ signs is even}\}
&\oneone&\mB_k^{\rm even}\\
\{T\in ADT_{\rm sgn}((k,k))\mid \text{the number of $-$ signs is odd}\}
&\oneone&\mB_k^{\rm odd}
\end{eqnarray*}
\end{corollary}
\begin{proof}
This follows directly from the definitions and Lemma \ref{Bijection}.
\end{proof}

\begin{remark}
Lemma \ref{Bijection} transfers the notion of closed and open clusters to cup diagrams: a \emph{closed cluster} consists there of a cup $c$ together with all cups contained in $c$, with the property that $c$ is not contained in any other cup and located to the left of all rays, whereas an \emph{open cluster} consists of the leftmost ray and all cups and rays to the right of it.
\end{remark}

A \emph{cycle move} on a closed cluster is the move illustrated in the following diagram changing a closed cluster $C$ into a (not admissible) standard domino tableau $\op{Cyc}(C)$ containing only horizontal dominoes:

\begin{eqnarray}
\label{CYC}
\begin{tikzpicture}[thick,font=\small,scale=0.8]
\node at +(-1.5,0.5) {$\op{Cycle}:$};
\draw (-.5,0) rectangle +(.5,1) node at +(.25,.5) {a};
\foreach \x/\l in {0/b,1/c,2/d,3/e} {\draw (\x,.5) rectangle +(1,.5) node at +(.5,.25) {\l};};
\foreach \x/\l in {0/f,1/g,2/h,3/i} {\draw (\x,0) rectangle +(1,.5) node at +(.5,.25) {\l};};
\draw (4,0) rectangle +(.5,1) node at +(.25,.5) {j};

\draw[|->] (5,.5) -- +(1,0);

\begin{scope}[xshift=6.5cm]
\foreach \x/\l in {0/a,1/b,2/c,3/d,4/e} {\draw (\x,.5) rectangle +(1,.5) node at +(.5,.25) {\l};};
\foreach \x/\l in {0/f,1/g,2/h,3/i,4/j} {\draw (\x,0) rectangle +(1,.5) node at +(.5,.25) {\l};};
\end{scope}

\end{tikzpicture}
&&
\end{eqnarray}
Note that neither the shape nor the set of numbers used for the filling is changed. The numbers are consecutive from a set of $2m$ numbers for the shape $(2m,2m)$. In contrast, it is important to observe that the filling in the subdiagrams of shape $(2m',2m')$ using the first $2m'$ dominoes for $m'<m$ is not a set of consecutive numbers.

 The concept of cycle moves goes back to \cite{Garfinkle}. We use them here to establish a bijection between signed domino tableaux of a fixed shape and ordinary domino tableaux of the same shape:

\begin{lemma}
\label{cyclemovebij}
Applying the cycle move \eqref{CYC} to each closed cluster with sign $+$, and forgetting afterwards all signs determines a bijection
\begin{eqnarray*}
\op{Cyc}:\quad ADT_{\rm sgn,cl}((r,s))&\oneone&DT((r,s))
\end{eqnarray*}
for any admissible shape $(r,s)$.
\end{lemma}

\begin{proof}
Obviously, the map is well-defined and injective. For surjectivity take a standard domino tableau $S\in DT((r,s))$. If there is no horizontal domino starting at an odd column number then just place minus signs at all vertical dominoes in odd columns to obtain a preimage of $S$. Otherwise take the first such horizontal domino $D$ in the first row. We claim it is the top left domino of a unique connected subdiagram $S'$ of $S$ of the form shown on the right hand side of \eqref{CYC}, in particular it is the smallest such rectangular shaped diagram containing consecutive numbers as filling. Hence the uniqueness is clear and also the existence in case $S$ has two rows of the same length. Otherwise, $S$ has shape $(r,s)$ with $r, s$ odd in which case there is at least one vertical domino to the right of $D$ and the existence is again clear. Now use \eqref{CYC} and apply $\op{Cycle}^{-1}$ to $S'$ and assign a $+$ sign to the resulting cluster. Repeat this procedure with the part of the diagram to the right of $S'$ observing that this is again of admissible shape.
The result is an admissible signed domino tableau except that probably some signs are missing. Finally insert $-$'s for all missing signs to get $\op{Cyc}^{-1}(S)$.
\end{proof}

\begin{ex}
\label{bigexample}
The bijection from Corollary \ref{signeddominoestocups} assigns to the cup diagrams from Example~\ref{coolpicture} (read from left to right columnwise from top to bottom) the following signed domino tableaux
\begin{equation*}
\begin{tikzpicture}[thick,font=\scriptsize,scale=.65]

\foreach \x/\l in {1,2,3,4,5,6} {\draw (\x/2,0) rectangle +(.5,1) node at +(.25,.5) {\l};};
\draw (.75,.15) node {\tiny $+$} +(1,0) node {\tiny $+$} +(2,0) node {\tiny $+$};

\begin{scope}[xshift=4cm]
\foreach \x/\l in {1,4,5,6} {\draw (\x/2,0) rectangle +(.5,1) node at +(.25,.5) {\l};};
\draw (.75,.15) node {\tiny $+$} +(2,0) node {\tiny $+$};
\draw (1,.5) rectangle +(1,.5) node at +(.5,.25) {2};
\draw (1,0) rectangle +(1,.5) node at +(.5,.25) {3};
\end{scope}

\begin{scope}[xshift=8cm]
\foreach \x/\l in {1,2,3,6} {\draw (\x/2,0) rectangle +(.5,1) node at +(.25,.5) {\l};};
\draw (.75,.15) node {\tiny $+$} +(1,0) node {\tiny $+$};
\draw (2,.5) rectangle +(1,.5) node at +(.5,.25) {4};
\draw (2,0) rectangle +(1,.5) node at +(.5,.25) {5};
\end{scope}

\begin{scope}[xshift=12cm]
\foreach \x/\l in {1,2,3,4,5,6} {\draw (\x/2,0) rectangle +(.5,1) node at +(.25,.5) {\l};};
\draw (.75,.15) node {\tiny $-$} +(1,0) node {\tiny $-$} +(2,0) node {\tiny $+$};
\end{scope}

\begin{scope}[xshift=16cm]
\foreach \x/\l in {1,6} {\draw (\x/2,0) rectangle +(.5,1) node at +(.25,.5) {\l};};
\foreach \x/\l in {1/2,2/4} {\draw (\x,.5) rectangle +(1,.5) node at +(.5,.25) {\l};};
\foreach \x/\l in {1/3,2/5} {\draw (\x,0) rectangle +(1,.5) node at +(.5,.25) {\l};};
\draw (.75,.15) node {\tiny $+$};
\end{scope}

\begin{scope}[yshift=-2cm]

\begin{scope}[xshift=0cm]
\foreach \x/\l in {1,2,3,4,5,6} {\draw (\x/2,0) rectangle +(.5,1) node at +(.25,.5) {\l};};
\draw (.75,.15) node {\tiny $+$} +(1,0) node {\tiny $-$} +(2,0) node {\tiny $-$};
\end{scope}

\begin{scope}[xshift=4cm]
\foreach \x/\l in {1,2,3,6} {\draw (\x/2,0) rectangle +(.5,1) node at +(.25,.5) {\l};};
\draw (.75,.15) node {\tiny $-$} +(1,0) node {\tiny $-$};
\draw (2,.5) rectangle +(1,.5) node at +(.5,.25) {4};
\draw (2,0) rectangle +(1,.5) node at +(.5,.25) {5};
\end{scope}

\begin{scope}[xshift=8cm]
\foreach \x/\l in {1,6} {\draw (\x/2,0) rectangle +(.5,1) node at +(.25,.5) {\l};};
\foreach \x/\l in {1/2,2/3} {\draw (\x,.5) rectangle +(1,.5) node at +(.5,.25) {\l};};
\foreach \x/\l in {1/4,2/5} {\draw (\x,0) rectangle +(1,.5) node at +(.5,.25) {\l};};
\draw (.75,.15) node {\tiny $+$};
\end{scope}

\begin{scope}[xshift=12cm]
\foreach \x/\l in {1,4,5,6} {\draw (\x/2,0) rectangle +(.5,1) node at +(.25,.5) {\l};};
\draw (.75,.15) node {\tiny $-$} +(2,0) node {\tiny $-$};
\draw (1,.5) rectangle +(1,.5) node at +(.5,.25) {2};
\draw (1,0) rectangle +(1,.5) node at +(.5,.25) {3};
\end{scope}

\begin{scope}[xshift=16cm]
\foreach \x/\l in {1,2,3,4,5,6} {\draw (\x/2,0) rectangle +(.5,1) node at +(.25,.5) {\l};};
\draw (.75,.15) node {\tiny $-$} +(1,0) node {\tiny $+$} +(2,0) node {\tiny $-$};
\end{scope}

\end{scope}

\end{tikzpicture}
\end{equation*}
They correspond via Lemma \ref{cyclemovebij} to the following standard domino tableaux
\begin{equation*}
\begin{tikzpicture}[thick,font=\scriptsize,scale=.65]

\foreach \x/\l in {0/1,1/3,2/5} {\draw (\x,.5) rectangle +(1,.5) node at +(.5,.25) {\l};};
\foreach \x/\l in {0/2,1/4,2/6} {\draw (\x,0) rectangle +(1,.5) node at +(.5,.25) {\l};};

\begin{scope}[xshift=4cm]
\foreach \x/\l in {0/1,1/2,2/5} {\draw (\x,.5) rectangle +(1,.5) node at +(.5,.25) {\l};};
\foreach \x/\l in {0/3,1/4,2/6} {\draw (\x,0) rectangle +(1,.5) node at +(.5,.25) {\l};};
\end{scope}

\begin{scope}[xshift=8cm]
\foreach \x/\l in {0/1,1/3,2/4} {\draw (\x,.5) rectangle +(1,.5) node at +(.5,.25) {\l};};
\foreach \x/\l in {0/2,1/5,2/6} {\draw (\x,0) rectangle +(1,.5) node at +(.5,.25) {\l};};
\end{scope}

\begin{scope}[xshift=11.5cm]
\foreach \x/\l in {1,2,3,4} {\draw (\x/2,0) rectangle +(.5,1) node at +(.25,.5) {\l};};
\draw (2.5,.5) rectangle +(1,.5) node at +(.5,.25) {5};
\draw (2.5,0) rectangle +(1,.5) node at +(.5,.25) {6};
\end{scope}

\begin{scope}[xshift=16cm]
\foreach \x/\l in {0/1,1/2,2/4} {\draw (\x,.5) rectangle +(1,.5) node at +(.5,.25) {\l};};
\foreach \x/\l in {0/3,1/5,2/6} {\draw (\x,0) rectangle +(1,.5) node at +(.5,.25) {\l};};
\end{scope}

\begin{scope}[yshift=-2cm]

\begin{scope}[xshift=-.5cm]
\draw (.5,.5) rectangle +(1,.5) node at +(.5,.25) {1};
\draw (.5,0) rectangle +(1,.5) node at +(.5,.25) {2};
\foreach \x/\l in {3,4,5,6} {\draw (\x/2,0) rectangle +(.5,1) node at +(.25,.5) {\l};};
\end{scope}

\begin{scope}[xshift=3.5cm]
\foreach \x/\l in {1,2,3} {\draw (\x/2,0) rectangle +(.5,1) node at +(.25,.5) {\l};};
\draw (2,.5) rectangle +(1,.5) node at +(.5,.25) {4};
\draw (2,0) rectangle +(1,.5) node at +(.5,.25) {5};
\draw (3,0) rectangle +(.5,1) node at +(.25,.5) {6};
\end{scope}

\begin{scope}[xshift=8cm]
\foreach \x/\l in {0/1,1/2,2/3} {\draw (\x,.5) rectangle +(1,.5) node at +(.5,.25) {\l};};
\foreach \x/\l in {0/4,1/5,2/6} {\draw (\x,0) rectangle +(1,.5) node at +(.5,.25) {\l};};
\end{scope}

\begin{scope}[xshift=11.5cm]
\draw (.5,0) rectangle +(.5,1) node at +(.25,.5) {1};
\draw (1,.5) rectangle +(1,.5) node at +(.5,.25) {2};
\draw (1,0) rectangle +(1,.5) node at +(.5,.25) {3};
\foreach \x/\l in {4,5,6} {\draw (\x/2,0) rectangle +(.5,1) node at +(.25,.5) {\l};};
\end{scope}

\begin{scope}[xshift=16cm]
\draw (0,0) rectangle +(.5,1) node at +(.25,.5) {1};
\draw (.5,0) rectangle +(.5,1) node at +(.25,.5) {2};
\draw (1,.5) rectangle +(1,.5) node at +(.5,.25) {3};
\draw (1,0) rectangle +(1,.5) node at +(.5,.25) {4};
\draw (2,0) rectangle +(.5,1) node at +(.25,.5) {5};
\draw (2.5,0) rectangle +(.5,1) node at +(.25,.5) {6};
\end{scope}

\end{scope}
\end{tikzpicture}
\end{equation*}
\end{ex}

\begin{remark}
\label{Pietrahobij}
{\rm
Another bijection between signed domino tableaux of a fixed shape and certain standard domino tableaux using cycle moves was established already in \cite{PietrahoSpringer}. This bijection differs slightly from ours, since we never apply a cycle move to open clusters, thus the shape stays unchanged; we fix instead the parity of the total number of minus signs.
}
\end{remark}

\begin{definition}
A \emph{$2$-row character of type ${\rm D}_k$} is an unordered pair $\{\la,\mu\}$ of 1-row partitions  $\la,\mu$ with a total number of $k$ boxes. A \emph{special bitableau of type ${\rm D}_k$}  is a $2$-row character of type ${\rm D}_k$ filled with the numbers $\{1,\ldots, k\}$ increasing in the two rows. Let $\Omega_{a,b}=\Omega({D_k})_{a,b}$ be the set of such special bitableaux where $\la$ has $a$ boxes and $\mu$ has $b$ boxes.
\end{definition}

The irreducible representations of the Weyl group $W(B_k)$ are indexed by ordered pairs of partitions $(\la,\mu)$, see e.g. \cite{MacDonald} or \cite[8.2]{Serre}. Their restrictions $S^{(\la,\mu)}$ to the index $2$ subgroup $W(D_k)$ stay either irreducible, but two  irreducible representations become isomorphic (namely the one indexed by $(\la,\mu)$ and $(\mu,\la)$ in case $\la\not=\mu$) or split into two non- isomorphic irreducible summands  $S^{\la,\mu,+}, S^{\la,\mu,-}$ (if $\la=\mu$). Let $s_i$, $0\leq i\leq k-1$ be the Coxeter generators of $W(D_k)$ labelled by the vertices of the Dynkin diagram. We choose the labelling such that there is an edge between $0$ and $2$ and between $i$ and $i+1$ for $1\leq i\leq k-2$.  Then $\{s_i \mid 1\leq i\leq k-1\}$ generate a subgroup isomorphic to $S_k$. The $2$-row characters appear naturally as summands in the following induced module:

 \begin{lemma}
 \label{inducedmodule}
 The induced trivial representation $\op{Ind}^{W(D_k)}_{S_k} \mathbbm{1}$
 of $S_k$ decomposes into a multiplicity free sum of $2$-row characters.
 \end{lemma}

 \begin{proof}
 Let $A=(\mZ / 2\mZ)^{k-1}$. Since $W(D_k)\cong A \rtimes S_k$ we can apply the Wigner-Mackey little group argument, \cite[8.2]{Serre} to construct the irreducible characters of $W(D_k)$. The $S_k$-action on the group of characters of $A$ has exactly $\lfloor k/2 \rfloor+1$ orbits and we can chose representatives
 $\chi_0,\ldots,\chi_{\lfloor k/2 \rfloor}$ such that the
 stabiliser $H_i$ of $\chi_i$ equals $S_i \times S_{k-i}$ if $i<k/2$ and $\mZ/2\mZ \ltimes (S_{k/2} \times S_{k/2})$ if $i=k/2$.
 An irreducible character of $W(D_k)$ is of the form ${\Ind}_{A \rtimes
 H_i}^{W(D_k)} (\chi_i \otimes \rho)$ for some $i$ and $\rho$ an
 irreducible character of $H_i$. Using Frobenius reciprocity and the Mackey formula we obtain
 \begin{eqnarray*}
 &&{\rm Hom}_{W(D_k)}\left({\rm Ind}_{A \rtimes H_i}^{W(D_k)} (\chi_i
 \otimes \rho),\op{Ind}^{W(D_k)}_{S_k} \mathbbm{1}\right)\\
 &\cong&{\rm Hom}_{S_k}\left({\rm Res}^{W(D_k)}_{S_k}{\rm Ind}_{A \rtimes
 H_i}^{W(D_k)} (\chi_i \otimes \rho),\mathbbm{1}\right)\\
 &\cong&{\rm Hom}_{S_k}\left({\rm Ind}^{S_k}_{H_i}{\rm Res}_{H_i}^{A
 \rtimes H_i} (\chi_i \otimes \rho),\mathbbm{1}\right)={\rm
 Hom}_{S_k}\left({\rm Ind}^{S_k}_{H_i} \rho,\mathbbm{1}\right).
\end{eqnarray*}
The latter is non-trivial if and only if $\rho$ is the trivial representation of $H_i$, i.e. the induced representation is of 2-row type and in case of $i=k/2$ equal to $S^{(k/2),(k/2),+}$ or $S^{(k/2),(k/2),-}$, depending on our choice $W(D_k)\cong A \rtimes S_k$ of isomorphism. As the dimension is at most $1$, the claim follows.
\end{proof}

We finish this section with a bijection between cup diagrams and a labelling set of the basis elements of our irreducible representations of $W(D_k)$. Given a cup diagram $a$ on $k$ vertices consider all the cups $c$ which have no ray to the left and are not contained in any other cup. If $c$ is undotted (resp. dotted), mark its left (right) endpoint and also the left (right) endpoints of all the cups contained in $c$. Then mark the vertices at rays (if they exist) and all left endpoints of cups which are to the right of at least one ray. Let $x$ be the number of marked vertices. The special bitableau $\Omega(c)$ of type ${\rm D}_k$ associated with $c$ is then the pair $(\la,\mu)$, where $\la$ has $x$ boxes and $\mu$ has $k-x$ boxes filled in increasing order with the number of the marked (resp. not marked) vertices. The following bijection identifies cup diagrams with special bitableaux, explaining the counting formula from Remark \ref{counting}:

\begin{lemma}
\label{bijchi}
The assignment $c\mapsto \Omega(c)$ defines for odd $r,s$ bijections
\begin{eqnarray*}
\left\{
\begin{array}{c}
\text{cup diagrams on $k=\frac{r+s}{2}$ vertices with $\lfloor \frac{s}{2}\rfloor$}\\
\text{cups and an even (resp. odd) number of dots}
\end{array}
\right\}
\;\oneone\;
\Omega_{\frac{r+1}{2},\frac{s-1}{2}}
\end{eqnarray*}
and a bijection $\mB_k\oneone \Omega_{\frac{k}{2},\frac{k}{2}}$ for even $r=s=k$.
\end{lemma}

\begin{proof}
The number of marked points equals the number of cups plus the number of rays.
\end{proof}

\begin{remark} \label{rem:weyl_action}
{\rm
\begin{enumerate}[a.)]
\item
The induced sign representation was identified in \cite[Theorem 5.17]{LS} with a vector space with basis all cup diagrams on $k$ vertices together with a diagrammatical action of $W(D_k)$, and the decomposition into irreducible modules given by the span of cup diagrams with a fixed number of cups, \cite[Remark 5.24]{LS}. We get the refined decompositions of Lemma~\ref{inducedmodule}:
\begin{eqnarray}
\label{decomprefined}
\op{Ind}^{W(D_k)}_{S_k}\mathbbm{1}&\cong&
\bigoplus_{l=0}^{\frac{k-\epsilon}{2}} S^{((k-l),(l))}
\end{eqnarray}
  where $\epsilon=1$ for $k$ odd and $\epsilon=0$ for $k$ even with $S^{((k/2),(k/2))}$ then denoting one out of the two $S^{((k/2),(k/2)),\pm}$ depending on the choices made, as in the proof of Lemma \ref{inducedmodule}.
  The summands can be identified with Kazhdan-Lusztig cell modules. Lemma \ref{bijchi} does not induce a $W(D_k)$- equivariant isomorphism (even after twisting with the sign representation) of the corresponding irreducible representations of $W(D_k)$, but only compares the dimensions. To establish an explicit isomorphism between the two representations one has to link the canonical basis for the cell module (the cup diagrams) with the analog of a Specht basis (the bitableau) using for instance \cite{NaruseTypeC}, as was done successfully for type $\rm A$ in \cite{Naruse}. Note however that outside of the class of $2$-row characters the cell modules are not necessarily irreducible.

\item
The bijection from Lemma \ref{bijchi} induces via Lemmas \ref{Bijection} and \ref{cyclemovebij} bijections
$DT(r,s)\oneone\Omega_{\frac{r+1}{2},\frac{s-1}{2}}$ if $r, s$ odd and
$DT(r,s)\oneone\Omega_{\frac{k}{2},\frac{k}{2}}$ if $r=s=k$ even. They are in fact special cases of known bijections between standard domino tableaux and pairs of partitions with standard fillings established in \cite{StantonWhite}. These Stanton-White bijections were used for instance in the study of decompositions of tensor products of representations of $\gl_n$ and generalized Littlewood-Richardson theory, see e.g. \cite{CarreLeclerc} and in the representation theory of the Weyl group $W(B_k)$, see e.g. \cite{PietrahoRS}.
\end{enumerate}
}
\end{remark}

\begin{ex}
\label{extable}
Here is an example summarising our bijections: (we only display half of the bitableaux, since the second half is then determined)
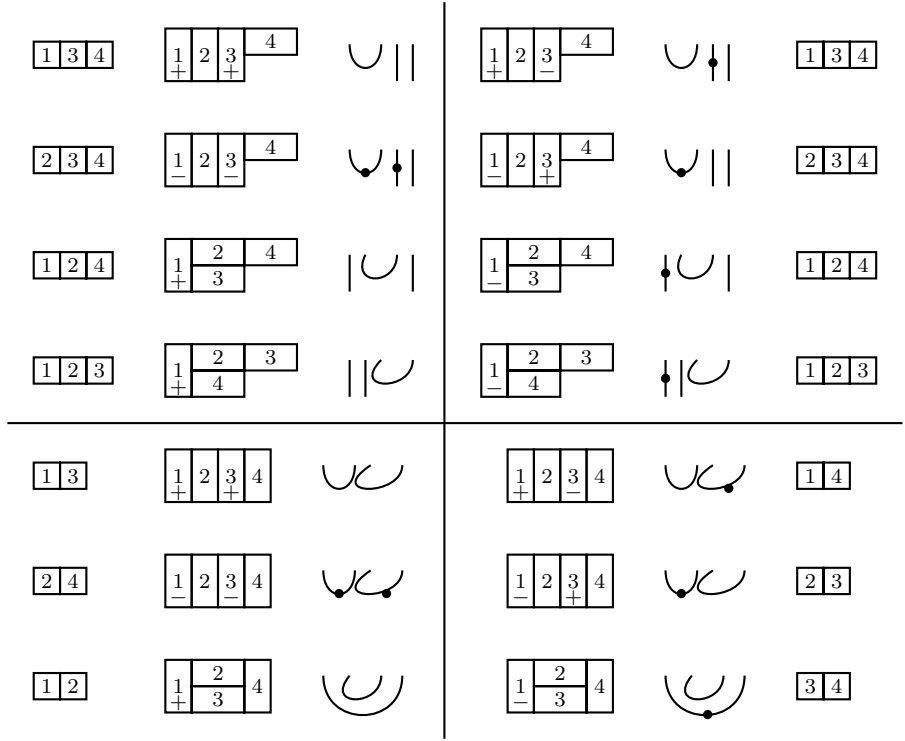
\begin{equation*}
\usetikzlibrary{arrows}
\begin{tikzpicture}[thick,font=\scriptsize,scale=.7]

\draw (7.8,1) -- +(0,-14);
\draw (-.5,-7) -- +(17,0);

\draw (0,-.25) rectangle +(.5,.5) node at +(.25,.25) {1} +(.5,0) rectangle +(1,.5) node at +(.75,.25) {3} +(1,0) rectangle +(1.5,.5) node at +(1.25,.25) {4};
\foreach \x/\l in {1,2,3} {\draw (\x/2+2,-.5) rectangle +(.5,1) node at +(.25,.5) {\l};};
\draw (2.5,-.5) node at +(.25,.18) {$+$} node at +(1.25,.18) {$+$} +(2.5,.5) rectangle +(1.5,1) node at +(2,.75) {4};
\draw (6,.2) .. controls +(0,-.6) and +(0,-.6) .. +(.6,0) +(.9,0) -- +(.9,-.7) +(1.2,0) -- +(1.2,-.7);

\draw (0,-2.25) rectangle +(.5,.5) node at +(.25,.25) {2} +(.5,0) rectangle +(1,.5) node at +(.75,.25) {3} +(1,0) rectangle +(1.5,.5) node at +(1.25,.25) {4};
\foreach \x/\l in {1,2,3} {\draw (\x/2+2,-2.5) rectangle +(.5,1) node at +(.25,.5) {\l};};
\draw (2.5,-2.5) node at +(.25,.18) {$-$} node at +(1.25,.18) {$-$} +(2.5,.5) rectangle +(1.5,1) node at +(2,.75) {4};
\draw (6,-1.8) .. controls +(0,-.6) and +(0,-.6) .. +(.6,0) +(.9,0) -- +(.9,-.7) +(1.2,0) -- +(1.2,-.7);
\fill (6.3,-2.24) circle(2.5pt);
\fill (6.9,-2.15) circle(2.5pt);

\draw (0,-4.25) rectangle +(.5,.5) node at +(.25,.25) {1} +(.5,0) rectangle +(1,.5) node at +(.75,.25) {2} +(1,0) rectangle +(1.5,.5) node at +(1.25,.25) {4};
\draw (2.5,-4.5) rectangle +(.5,1) node at +(.25,.18) {$+$} node at +(.25,.5) {1} +(.5,.5) rectangle +(1.5,1) node at +(1,.75) {2} +(.5,0) rectangle +(1.5,.5) node at +(1,.25) {3} +(2.5,.5) rectangle +(1.5,1) node at +(2,.75) {4};
\draw (6,-3.8) -- +(0,-.7) +(.3,0)  .. controls +(0,-.6) and +(0,-.6) .. +(.9,0) +(1.2,0) -- +(1.2,-.7);

\draw (0,-6.25) rectangle +(.5,.5) node at +(.25,.25) {1} +(.5,0) rectangle +(1,.5) node at +(.75,.25) {2} +(1,0) rectangle +(1.5,.5) node at +(1.25,.25) {3};
\draw (2.5,-6.5) rectangle +(.5,1) node at +(.25,.18) {$+$} node at +(.25,.5) {1} +(.5,.5) rectangle +(1.5,1) node at +(1,.75) {2} +(.5,0) rectangle +(1.5,.5) node at +(1,.25) {4} +(2.5,.5) rectangle +(1.5,1) node at +(2,.75) {3};
\draw (6,-5.8) -- +(0,-.7) +(.3,0) -- +(.3,-.7) +(.6,0)  .. controls +(0,-.6) and +(0,-.6) .. +(1.2,0);

\begin{scope}[xshift=8.5cm]
\draw (6,-.25) rectangle +(.5,.5) node at +(.25,.25) {1} +(.5,0) rectangle +(1,.5) node at +(.75,.25) {3} +(1,0) rectangle +(1.5,.5) node at +(1.25,.25) {4};
\foreach \x/\l in {1,2,3} {\draw (\x/2-.5,-.5) rectangle +(.5,1) node at +(.25,.5) {\l};};
\draw (0,-.5) node at +(.25,.18) {$+$} node at +(1.25,.18) {$-$} +(2.5,.5) rectangle +(1.5,1) node at +(2,.75) {4};
\draw (3.5,.2) .. controls +(0,-.6) and +(0,-.6) .. +(.6,0) +(.9,0) -- +(.9,-.7) +(1.2,0) -- +(1.2,-.7);
\fill (4.4,-.15) circle(2.5pt);

\draw (6,-2.25) rectangle +(.5,.5) node at +(.25,.25) {2} +(.5,0) rectangle +(1,.5) node at +(.75,.25) {3} +(1,0) rectangle +(1.5,.5) node at +(1.25,.25) {4};
\foreach \x/\l in {1,2,3} {\draw (\x/2-.5,-2.5) rectangle +(.5,1) node at +(.25,.5) {\l};};
\draw (0,-2.5) node at +(.25,.18) {$-$} node at +(1.25,.18) {$+$} +(2.5,.5) rectangle +(1.5,1) node at +(2,.75) {4};
\draw (3.5,-1.8) .. controls +(0,-.6) and +(0,-.6) .. +(.6,0) +(.9,0) -- +(.9,-.7) +(1.2,0) -- +(1.2,-.7);
\fill (3.8,-2.24) circle(2.5pt);

\draw (6,-4.25) rectangle +(.5,.5) node at +(.25,.25) {1} +(.5,0) rectangle +(1,.5) node at +(.75,.25) {2} +(1,0) rectangle +(1.5,.5) node at +(1.25,.25) {4};
\draw (0,-4.5) rectangle +(.5,1) node at +(.25,.15) {$-$} node at +(.25,.5) {1} +(.5,.5) rectangle +(1.5,1) node at +(1,.75) {2} +(.5,0) rectangle +(1.5,.5) node at +(1,.25) {3} +(2.5,.5) rectangle +(1.5,1) node at +(2,.75) {4};
\draw (3.5,-3.8) -- +(0,-.7) +(.3,0)  .. controls +(0,-.6) and +(0,-.6) .. +(.9,0) +(1.2,0) -- +(1.2,-.7);
\fill (3.5,-4.15) circle(2.5pt);

\draw (6,-6.25) rectangle +(.5,.5) node at +(.25,.25) {1} +(.5,0) rectangle +(1,.5) node at +(.75,.25) {2} +(1,0) rectangle +(1.5,.5) node at +(1.25,.25) {3};
\draw (0,-6.5) rectangle +(.5,1) node at +(.25,.15) {$-$} node at +(.25,.5) {1} +(.5,.5) rectangle +(1.5,1) node at +(1,.75) {2} +(.5,0) rectangle +(1.5,.5) node at +(1,.25) {4} +(2.5,.5) rectangle +(1.5,1) node at +(2,.75) {3};
\draw (3.5,-5.8) -- +(0,-.7) +(.3,0) -- +(.3,-.7) +(.6,0)  .. controls +(0,-.6) and +(0,-.6) .. +(1.2,0);
\fill (3.5,-6.15) circle(2.5pt);
\end{scope}

\begin{scope}[yshift=-8cm]
\draw (0,-.25) rectangle +(.5,.5) node at +(.25,.25) {1} +(.5,0) rectangle +(1,.5) node at +(.75,.25) {3};
\foreach \x/\l in {1,2,3,4} {\draw (\x/2+2,-.5) rectangle +(.5,1) node at +(.25,.5) {\l};};
\draw (2.5,-.5) node at +(.25,.18) {$+$} node at +(1.25,.18) {$+$} +(2.5,.5);
\draw (5.5,.2) .. controls +(0,-.6) and +(0,-.6) .. +(.6,0) +(.9,0) .. controls +(0,-.6) and +(0,-.6) .. +(1.5,0);

\draw (0,-2.25) rectangle +(.5,.5) node at +(.25,.25) {2} +(.5,0) rectangle +(1,.5) node at +(.75,.25) {4};
\foreach \x/\l in {1,2,3,4} {\draw (\x/2+2,-2.5) rectangle +(.5,1) node at +(.25,.5) {\l};};
\draw (2.5,-2.5) node at +(.25,.18) {$-$} node at +(1.25,.18) {$-$} +(2.5,.5);
\draw (5.5,-1.8) .. controls +(0,-.6) and +(0,-.6) .. +(.6,0) +(.9,0) .. controls +(0,-.6) and +(0,-.6) .. +(1.5,0);
\fill (5.8,-2.24) circle(2.5pt);
\fill (6.7,-2.24) circle(2.5pt);

\draw (0,-4.25) rectangle +(.5,.5) node at +(.25,.25) {1} +(.5,0) rectangle +(1,.5) node at +(.75,.25) {2};
\draw (2.5,-4.5) rectangle +(.5,1) node at +(.25,.18) {$+$} node at +(.25,.5) {1} +(.5,.5) rectangle +(1.5,1) node at +(1,.75) {2} +(.5,0) rectangle +(1.5,.5) node at +(1,.25) {3} +(1.5,0) rectangle +(2,1) node at +(1.75,.5) {4};
\draw (5.5,-3.8) .. controls +(0,-1) and +(0,-1) .. +(1.5,0) +(.5,0) .. controls +(0,-.6) and +(0,-.6) .. +(1.1,0);
\end{scope}

\begin{scope}[xshift=9cm,yshift=-8cm]
\draw (5.5,-.25) rectangle +(.5,.5) node at +(.25,.25) {1} +(.5,0) rectangle +(1,.5) node at +(.75,.25) {4};
\foreach \x/\l in {1,2,3,4} {\draw (\x/2-.5,-.5) rectangle +(.5,1) node at +(.25,.5) {\l};};
\draw (0,-.5) node at +(.25,.18) {$+$} node at +(1.25,.18) {$-$} +(2.5,.5);
\draw (3,.2) .. controls +(0,-.6) and +(0,-.6) .. +(.6,0) +(.9,0) .. controls +(0,-.6) and +(0,-.6) .. +(1.5,0);
\fill (4.2,-.24) circle(2.5pt);

\draw (5.5,-2.25) rectangle +(.5,.5) node at +(.25,.25) {2} +(.5,0) rectangle +(1,.5) node at +(.75,.25) {3};
\foreach \x/\l in {1,2,3,4} {\draw (\x/2-.5,-2.5) rectangle +(.5,1) node at +(.25,.5) {\l};};
\draw (0,-2.5) node at +(.25,.18) {$-$} node at +(1.25,.18) {$+$} +(2.5,.5);
\draw (3,-1.8) .. controls +(0,-.6) and +(0,-.6) .. +(.6,0) +(.9,0) .. controls +(0,-.6) and +(0,-.6) .. +(1.5,0);
\fill (3.3,-2.24) circle(2.5pt);

\draw (5.5,-4.25) rectangle +(.5,.5) node at +(.25,.25) {3} +(.5,0) rectangle +(1,.5) node at +(.75,.25) {4};
\draw (0,-4.5) rectangle +(.5,1) node at +(.25,.18) {$-$} node at +(.25,.5) {1} +(.5,.5) rectangle +(1.5,1) node at +(1,.75) {2} +(.5,0) rectangle +(1.5,.5) node at +(1,.25) {3} +(1.5,0) rectangle +(2,1) node at +(1.75,.5) {4};
\draw (3,-3.8) .. controls +(0,-1) and +(0,-1) .. +(1.5,0) +(.5,0) .. controls +(0,-.6) and +(0,-.6) .. +(1.1,0);
\fill (3.8,-4.54) circle(2.5pt);
\end{scope}

\end{tikzpicture}
\end{equation*}
Note once more the different behaviour for the shape $(r,s)$ depending on $r\not=s$ or $r=s=k$.
\end{ex}

\section{Irreducible components of the algebraic Springer fibre}
\label{Section6}
We consider again the algebraic Springer fibre of type ${\rm D}_k$ from Definition~ \ref{Springerfibre}. The irreducible components of $\cF_N$ for classical $G$ were parameterized and partly described by Spaltenstein in \cite{SpaltensteinLNM}. We briefly recall the constructions needed in our case.

\subsection{Classification of irreducible components}
Let $\lambda=\lambda_N$ be the partition corresponding to the nilpotent element $N\in\mg$ and let $F_\bullet\in\cF_N$. Then $N$ defines a sequence of induced endomorphisms on $F_{n-i}/F_i=(F_i)^\perp/F_i$, for $k\geq i\geq 0$. Their Jordan types give us a sequence of domino diagrams where successive shapes differ by exactly one domino. Putting the number $i$ into the $i$th domino added in this way defines an admissible domino tableau $\cS_N(F_\bullet)$ of shape $\la_N$, thus an assignment
\begin{eqnarray}
\label{SN}
\cS_N:&& \cF_N \longrightarrow ADT(\lambda_N)
\end{eqnarray}
For details we refer to \cite{SpaltensteinLNM} or \cite{vanLeeuwen}.

The analogous construction in type $\rm A$ (to each ordinary full flag $F_\bullet$ one assigns an ordinary standard tableau $\cS^A_N(F_\bullet)$ encoding the Jordan types of the restriction of $N$ to the $i$th part of the flag) gives rise to the well-known Spaltenstein-Vargas classification of irreducible components of Springer fibres for $SL(n)$:
\begin{theorem}[\cite{SpaltensteintypeA}, \cite{Vargas}]
Let $N\in\mathfrak{sl}(n,\mC)$ be nilpotent and let $\cF^A$ be the flag variety for $\op{SL}(n,\mC)$ containing the Springer fibre $\cF^A_N$.
The map $\cS^A_N$ defines a surjection onto the set  $oSYT(\lambda_N)$ of ordinary standard Young tableaux  which  separates the irreducible components
$\textup{Irr}(\mathcal{F}^A_N)$ of $\mathcal{F}^A_N$.  That is, it defines a bijection
\begin{eqnarray*}
\cS^A_N:\quad\text{Irr}(\mathcal{F}^A_N) \longrightarrow oSYT(\lambda_N).
\end{eqnarray*}
Given $T\in oSYT(\lambda_N)$, the irreducible component $X^A_T=(\cS^A_N)^{-1}(T)$ equals the closure of $(X^A_T)^0=\{F_\bullet\mid\cS^A_N(F_\bullet)=T\}$.
\end{theorem}

In type $\rm D$, admissible tableaux do not
fully separate the components of $\cF_N$.  It is still true that $F_\bullet,G_\bullet\in\cF_N$ with    $\cS_N(F_\bullet)\not=\cS_N(F_\bullet)$ lie in different
irreducible components, but the converse fails, since of course $(X_T)^0=\{F_\bullet\mid\cS^A_N(F_\bullet)=T\}$ and its closure $X_T$
need not be connected. The signed domino tableaux were introduced to account for this disconnectedness in form of the following classification of Van Leeuwen \cite[Lemmas 3.2.3 and 3.3.3]{vanLeeuwen}:

\begin{theorem}\label{paramatrization}
Let $N\in\mg$ be nilpotent of admissible Jordan type $\la_N=(\la_1,\la_2)$ and $T\in ADT(\lambda_N)$.
\begin{enumerate}
\item $X_T$ is a union of  irreducible components of $\cF_N$, labelled by all  signed domino tableaux which equal $T$ after forgetting the signs.
\item All irreducible components of $\cF_N$ are obtained in this way.
\item There is a bijection between the set of admissible signed standard domino tableaux of shape
$\lambda_N$ and the set of irreducible components of $\cF_N$,
\begin{eqnarray}
ADT_{\rm sgn}(\lambda_N)&\oneone&\op{Irr}(\cF_N).
\end{eqnarray}
\end{enumerate}
\end{theorem}

\begin{remark}{\rm
For general Jordan type the labelling set of irreducible components is given by certain equivalence classes of signed domino tableaux, \cite[Definition 3.3.2]{vanLeeuwen}, but in the special case of 2-row shapes the equivalence relation becomes trivial, since every cluster contains only one domino of type $V^1$.
}\end{remark}

If $N$ has Jordan type $(r,s)$, irreducible components are of dimension $\lfloor \frac{s}{2}\rfloor$, \cite[Lemma 2.3.4 and p.14]{vanLeeuwen}. Hence the dimension equals the number of cups in the cup diagram associated to the admissible signed domino tableau from Theorem \ref{paramatrization} via the bijection from Lemma \ref{Bijection}.

\subsection{Components as iterated $\mP^1(\mC)$-bundles}
The rest of this section gives an inductive construction of the irreducible components and proves the following refinement of Theorem \ref{paramatrization}:
\begin{definition}
A topological space $X_1$ is an iterated fibre bundle of base type $(B_1,\ldots,B_l)$ if there exist spaces $X_1,B_1,X_2,B_2, \ldots,X_l ,B_l ,X_{l+1}=\op{pt}$ and maps $p_1,\ldots,p_l$ such that $p_j : X_j\rightarrow B_j$ is a fibre bundle
with typical fibre $X_{j+1}$.
\end{definition}

\begin{theorem}
\label{maintheorem}
Let $N\in\mg$ nilpotent of type $\la_N=(\la_1,\la_2)$ and $T\in ADT(\lambda_N)$.
\begin{enumerate}
\label{maintheorem1}
\item The assignment $\cS_N$ from \eqref{SN} induces a bijection
\begin{eqnarray}
S_N^{\rm sgn}:\quad ADT(\lambda_N)&\oneone&\op{Irr}(\cF_N).
\end{eqnarray}
where $(S_N^{\rm sgn})^{-1}(T)$ is given by the algorithm in Section~ \ref{algorithm}.
\item Every irreducible component $Y$ is an iterated fibre bundle  of base type $(\mP^1(\mC), \mP^1(\mC),\ldots,\mP^1(\mC))$ where the number of components equals the number of cups in the cup diagram corresponding to $(S_N^{\rm sgn})^{-1}(T)$ via Lemma \ref{Bijection}.
\end{enumerate}
\end{theorem}
The proof of the theorem will be given in the next subsections.

\subsection{Explicit linear algebra conventions}
Let $W=W_{a,a}$ be a complex vector space of dimension $n=2a$. We fix a basis $w_{1},\ldots, w_{a}, w_{-1},\ldots, w_{-a}$ and set $w_0=0=w_{-a-1}$. Define a bilinear form $\beta=\beta_{a}$ on $W$ by
$$ \beta(w_i,w_j) \, = \, \delta_{i,-j},$$
hence the non trivial pairings can be displayed as
\begin{eqnarray*}
\begin{tikzpicture}[thick,font=\tiny,scale=.6]
\draw (0,0) node {$1$};
\draw (1,0) node {$2$};
\draw (2,0) node {$\cdots$};
\draw (3,0) node {a};
\draw (4,0) node {-a};
\draw (5,0) node {$\cdots$};
\draw (6,0) node {$-2$};
\draw (7,0) node {$-1$};
\draw (0,-.2) .. controls +(0,-1) and +(0,-1) .. +(7,0);
\draw (1,-.2) .. controls +(0,-.8) and +(0,-.8) .. +(5,0);
\draw (3,-.2) .. controls +(0,-.5) and +(0,-.5) .. +(1,0);
\end{tikzpicture}
\end{eqnarray*}
Setting $N(w_i)=\sigma(i) w_{i-1}$ for $-a\leq i\leq a$ with $\sigma(i)=-1$ if $i<0$ and $\sigma(i)=1$ if $i>0$, defines a linear endomorphism of $W$. It satisfies $\beta(Nw,w')+\beta(w,Nw')=0$ for $w,w'\in W$, hence $N\in \mathfrak{o}(W,\beta)$.\\

Let $V_{a,b}=\mC^{a+b}$ be a complex vector space of dimension $n=a+b$ with $a, b$ odd. We fix a basis $e_{1},\ldots, e_{b}, e_{-1},\ldots, e_{-a}$, set $e_0=0=e_{-a-1}$ and fix the symmetric bilinear form $\gamma=\gamma_{a,b}$ on $V$ defined by
\begin{eqnarray*}
\gamma(e_i,e_j)&=& (-1)^i \left(\delta_{i+j,b+1} - \delta_{i+j,-(a+1)} \right)\\
\text{ i.e. , graphically} &&\parbox[c]{6cm}
{
\begin{tikzpicture}[thick,font=\tiny,scale=.6]
\draw (0,0) node {$1$};
\draw (1,0) node {$2$};
\draw (2,0) node {$\cdots$};
\draw (3,0) node {b-1};
\draw (4,0) node {b};
\draw (5,0) node {$-1$};
\draw (6,0) node {$-2$};
\draw (7,0) node {$\cdots$};
\draw (8,0) node {-(a-1)};
\draw (9,0) node {-a};
\draw (0,-.2) .. controls +(0,-.8) and +(0,-.8) .. +(4,0);
\draw (1,-.2) .. controls +(0,-.6) and +(0,-.6) .. +(2,0);
\draw (5,-.2) .. controls +(0,-.8) and +(0,-.8) .. +(4,0);
\draw (6,-.2) .. controls +(0,-.6) and +(0,-.6) .. +(2,0);
\end{tikzpicture}}.
\end{eqnarray*}
Setting $N(e_i)=e_{i-1}$ for $-a\leq i\leq -1$, $1\leq i\leq b$ defines a linear endomorphism $N$ of $V$ satisfying $\gamma(Nv,v')+\gamma(v,Nv')=0$ for $v,v'\in V$. Hence $N\in \mathfrak{o}(V,\gamma)$.

We identify $V_{a,a}$ with $W_{a,a}$ via the following isomorphism
\begin{lemma}
\label{twobilinears}
For odd $a$, the following defines an isomorphism
\begin{eqnarray}
\varphi:\quad V_{a,a}&\longrightarrow& W_{a,a}\nonumber\\
e_i&\longmapsto&
\begin{cases}
\frac{1}{\sqrt{2}}\left((-1)^{i+1}w_i+w_{i+a+1}\right)&\text{if $i<0$,}\\
\frac{1}{\sqrt{2}}\left((-1)^{i+1}w_{i-a-1}-w_i\right)&\text{if $i>0$.}
\end{cases}
\end{eqnarray}
 of vector spaces satisfying $\beta(\varphi(v),\varphi(v'))=\gamma(v,v')$ for any $v,v'\in V$.
\end{lemma}
\begin{proof}
The claim follows by a straightforward direct calculation. (The inverse map is given by $\varphi^{-1}(w_i)= \frac{1}{\sqrt{2}}(-1)^{i+1}(e_i+e_{a+i+1})$ for $i<0$ and $\varphi^{-1}(w_i)=\frac{1}{\sqrt{2}}(e_{i-a-1}-e_i)$ for $i>0$.)
\end{proof}

\subsection{Inductive construction of the irreducible components}
\label{algorithm}
 In this section we prove Theorem \ref{maintheorem} by constructing
the irreducible component $(S_N^{\rm sgn})^{-1}(T)$ corresponding to a signed domino tableau $T$. This will be done by induction on the number of dominoes. We distinguish the following basic cases (displaying the rightmost cluster and indicating the position of the domino $D$ with the largest entry):
\begin{eqnarray}
\label{Cases}
&
\begin{tikzpicture}[thick,font=\scriptsize,scale=.65]

\node at (.15,1.5) {I)};
\draw (0,0) rectangle +(.5,1) node at +(.25,.5) {1} +(.25,.17) node {\tiny $-$};
\draw (1,0) rectangle +(.5,1) node at +(.25,.5) {1} +(.25,.17) node {\tiny $+$};

\begin{scope}[xshift=1.9cm]
\node at (.8,1.5) {IIa)};
\draw (.5,0) rectangle +(.5,1);
\foreach \x/\l in {1,2} {\draw (\x,.5) rectangle +(1,.5);};
\draw[pattern color=red,pattern=north east lines] (3,.5) rectangle +(1,.5);
\draw (1,0) rectangle +(1,.5);
\end{scope}

\begin{scope}[xshift=5.9cm]
\node at (.8,1.5) {IIb)};
\draw (.5,0) rectangle +(.5,1);
\foreach \x/\l in {1,2,3} {\draw (\x,.5) rectangle +(1,.5);};
\draw (1,0) rectangle +(1,.5);
\draw[pattern color=red,pattern=north east lines] (2,0) rectangle +(1,.5);
\end{scope}

\begin{scope}[xshift=10.4cm]
\node at (.25,1.5) {IIIa)};
\draw (0,0) rectangle +(.5,1);
\draw (.5,.5) rectangle +(1,.5);
\draw (.5,0) rectangle +(1,.5);
\draw[pattern color=red,pattern=north east lines] (1.5,0) rectangle +(.5,1);
\end{scope}

\begin{scope}[xshift=13.4 cm]
\node at (.3,1.5) {IIIb)};
\draw (0,0) rectangle +(.5,1);
\draw (.5,.5) rectangle +(1,.5);
\draw (.5,0) rectangle +(1,.5);
\draw (1.5,0) rectangle +(.5,1);
\draw[pattern color=red,pattern=north east lines] (2,0) rectangle +(.5,1);
\end{scope}

\end{tikzpicture}&\nonumber\\
\label{Casesforinduction}
\end{eqnarray}

\begin{proof}[Proof of Theorem \ref{maintheorem}]
To each admissible signed domino tableau $T$ with say $d$ dominoes, we construct now the corresponding irreducible component.

Our starting point is $d=1$. Then $N$ has two Jordan blocks of size $1$ and we consider $W=W_{1,1}$. The possible signed domino tableaux are displayed in \eqref{Casesforinduction} (I). The $1$-dimensional isotropic flags in $W$ are $F_1=\langle w_{-1}\rangle$, and $F_1=\langle w_1\rangle$ which are automatically in the kernel of $N$. Each of them defines an irreducible component in $\cF_N$. We assign $F_1=\langle w_{-1}\rangle$ to the first (negative) signed tableau and $F_1=\langle w_{1}\rangle$ to the second (positive) signed tableau. The signs indicate the components of $\cF$ with respect to our chosen basis.
(Alternatively we could work with $V=V_{1,1}$ and take $F_1=\langle e_1+e_{-1}\rangle$ and $F_1=\langle e_{-1}-e_1\rangle$ respectively.)

Let now $d>1$ and assume the theorem holds for less than $d$ dominoes. Let $D$ be the domino in $T$ filled with the highest number.

\textbf{Assume: $T$ is of shape $(a,b)$ and $D$ is horizontal:} Consider $V=V_{a,b}$.

A line in the kernel of $N$ is of the form $F_1=\langle \alpha e_{-a}+\beta e_{1}\rangle$, for some $[\alpha:\beta] \in \mP^1(\mC)$, and is automatically isotropic. First focus on a generic point given by $[\alpha : \beta] \notin \{[0:1],[1:0]\}$. As a new basis $B_{\rm new}$ of $V$ pick
\begin{eqnarray*}\arraycolsep=1.4pt
\left\lbrace\begin{array}{ccccccl}
\beta e_{-a}, & \beta e_{-a+1}, & \ldots, &
\beta e_{-b-1}, & \beta e_{-b}+\alpha e_1, & \ldots, & 
\beta e_{-1}+\alpha e_b,\\
& & & & \alpha e_{-a}+\beta e_1, & \ldots, &
\alpha e_{-a+b-1}+\beta e_b
\end{array}\right\rbrace
\end{eqnarray*}
Note that $N$ maps an element of $B_{\rm new}$ to its left neighbour, the two leftmost to zero. Here ${F_1^\perp}$ is spanned by all vectors in $B_{\rm new}$ except $\alpha e_{-a+b-1}+\beta e_b$. Hence the endomorphism induced by $N$ on ${F_1^\perp}/F_1$ is of Jordan type $(a,b-2)$. 

For the point $[0:1]$, i.e. $F_1=\langle e_1 \rangle$, we can use the original basis of $V$ and it holds ${F_1^\perp}=\langle e_j\mid -a\leq j\leq -1, 1\leq j < b\rangle$. Again, the induced endomorphism is of Jordan type $(a,b-2)$. We identify ${F_1^\perp}/F_1$ and $V_{a,b-2}$ by mapping $e_i\in V_{a,b-2}$ to $e_i+F_1$ for $i<0$ and to ${\ii}e_{i+1}+F_1$ for $i>0$. By taking the closure we thus obtain for Case (IIb) a $\mP^1(\mC)$-bundle over the space of flags given by the tableaux $T$ with the domino $D$ removed. Under the identification of domino tableaux and cup diagrams the label of $D$ is exactly the position of the right end of a cup, see \ref{explicitexample} for an example.

For the point $[1:0]$, i.e. $F_1=\langle e_{-a}\rangle$, we see that ${F_1^\perp}=\langle e_j\mid -a\leq j < -1, 1\leq j\leq b\rangle$ and ${F_1^\perp}/F_1$ is identified with $V_{a-2,b}$ by mapping the basis element $e_i$ of $V_{a-2,b}$ to $e_i+F_1$ for $i>0$ and to ${\ii}e_{i-1}+F_1$ for $i<0$. The Jordan type of the endomorphism induced by $N$ on ${F_1^\perp}/F_1$ is $(a-2,b)$ in this case. Hence for Case (IIa) we get a unique choice of a subspace. Under the bijection with cuo diagrams the label of $D$ will be a ray or left end of a cup.

For both cases the statement follows then by induction by looking at the quotient ${F_1^\perp}/F_1$ with the respective identifications.

\textbf{Assume: $T$ of shape $(a,a)$ and $D$ vertical:} Consider the space  $W_{a,a}$.

Now, a line in the kernel of $N$ is of the form $F_1=\langle\alpha w_{-a}+\beta w_1\rangle$ for $[\alpha:\beta]\in \mP^1(\mC)$ and again automatically isotropic. For a generic point given by $[\alpha : \beta] \notin \{[0:1],[1,:0]\}$, we choose the new basis $B_{\rm new}$ of $W$ as
\begin{eqnarray*}\arraycolsep=1.4pt
\left\lbrace\begin{array}{rrrr}
\alpha w_{-a} - \beta w_1, & \alpha w_{-a+1} + \beta w_2, & \ldots, & \alpha w_{-1}+ (-1)^a \beta w_a, \\
\alpha w_{-a} + \beta w_1, & - \alpha w_{-a+1} + \beta w_2, & \ldots, & (-1)^{a+1} \alpha w_{-1} + \beta w_a
\end{array}\right\rbrace
\end{eqnarray*}
Note again, that $N$ sends each vector to its left neighbour (but multiplied by $-1$ in the first row), the leftmost ones to zero. Now $F_1^\perp$ is spanned by all elements in $B_{\rm new}$ except $\alpha w_{-1}+\beta w_a$. In case of $a$ being even, this is the rightmost vector in the first row and thus $N$ induces on $F_1^\perp/{F_1}$ an endomorphism of Jordan type $(a-1,a-1)$ and moreover $F_1^\perp/F_1\cong W_{a-1,a-1}$ with the induced bilinear form. In case of $A$ being odd, $\alpha w_{-1}+\beta w_a$ is a multiple of the rightmost vector in the second row, so the induced endomorphism is of Jordan type $(a,a-2)$.

Hence for Case (IIIa), i.e., $a$ is even, we obtain, as the closure, a $\mP^1(\mC)$-bundle over the space of flags for the tableaux were $D$ is removed. Again the label of $D$ is the position for the right end of a cup.

Looking at the two points $[1:0]$ and $[0:1]$, i.e., $F_1 = \langle w_{-a} \rangle$ respectively $F_1 = \langle w_1 \rangle$, it holds $F_1^\perp=\langle w_i \mid -a\leq i\leq -1, 1\leq i < a \rangle$ and  $F_1^\perp=\langle w_i \mid -a\leq i< -1, 1\leq i\leq a\rangle$ respectively. For both choices $F_1^\perp/F_1\cong W_{a-1,a-1}$ with the induced bilinear form, and $N$ induces an endomorphism of Jordan type $(a-1,a-1)$.

For Case (IIIb), i.e., $a$ is odd, we obtain two choices corresponding to the sign of the domino $D$ in this case. We assign the choice $F_1=\langle w_{-a}\rangle$ to the negative sign in $D$ and $F_1=\langle w_{1}\rangle$ to the positive sign in $D$ (In this case corresponding to a dotted or undotted ray at the label of $D$ or the left end of a cup). Again the statement follows by induction looking at $F_1^\perp/F_1$.

By construction, the first part of Theorem \ref{maintheorem} holds. Moreover, we have shown that the components are iterated $\mP^1(\mC)$-bundles of the correct base type and therefore smooth (see e.g. \cite[Section 8]{Gisa} for a detailed argument).
\end{proof}

\subsection{An explicit example}
\label{explicitexample}
We illustrate the inductive construction of the irreducible components for the admissible shape $(5,3)$ and refer to Example~\ref{extable} for the corresponding cup diagrams. We identify $\mC^{5+3}$ with the vector space $V_{5,3}$ with basis $e_{-5},e_{-4},e_{-3},e_{-2},e_{-1},e_1, e_2,e_3$ and form $\gamma$.

{\it The irreducible components attached to the four diagrams}
\parbox[c]{1.7cm}
{
\begin{tikzpicture}[thick,font=\tiny,scale=.6]
\foreach \x/\l in {1,2,3} {\draw (\x/2+2,-.5) rectangle +(.5,1) node at +(.25,.55) {\l};};
\draw (2.5,-.5) node at +(.25,.2) {$\pm$} node at +(1.25,.2) {$\pm$} +(2.5,.5) rectangle +(1.5,1) node at +(2,.75) {4};
\end{tikzpicture}
}:

The domino labelled $4$ puts us in Case (IIa). This implies $F_1=\langle e_{-5}\rangle$ and removing the domino means passing to $\langle e_i\mid -5\leq i\leq -2, 1\leq i\leq 3\rangle/F_1$ which we identify with $W_{3,3}$ via Lemma \ref{twobilinears} with basis
\begin{equation*}\arraycolsep=1.4pt
\left\lbrace\begin{array}{rclrclrcl}
w_{-3}&=&\frac{1}{\sqrt{2}}(\ii e_{-4}+e_1), \;\;&
w_{-2}&=&\frac{1}{\sqrt{2}}(-\ii e_{-3}-e_2), \;\;&
w_{-1}&=&\frac{1}{\sqrt{2}}(\ii e_{-2}+e_3),\\
w_1&=&\frac{1}{\sqrt{2}}(\ii e_{-4}-e_1),&
w_2&=&\frac{1}{\sqrt{2}}(\ii e_{-3}-e_2),&
w_3&=&\frac{1}{\sqrt{2}}(\ii e_{-2}-e_3)
\end{array}\right\rbrace
\end{equation*}
The domino labelled $3$ belongs to Case~(IIIb), thus $F_2=F_1+\langle w_{-3}\rangle$ if the sign is $-$ and $F_2=F_1+\langle w_{1}\rangle$ if it is $+$. We identify the span of $w_{-2}, w_{-1},w_1,w_2$ (resp.  $w_{-3}, w_{-2},w_2,w_3$) with $F_2^\perp/F_2$. The domino labelled $2$ belongs to Case~(IIIa), hence we consider the generic bases (i.e. $[\alpha:\beta] \notin \{ [1:0],[0:1]\}$) of $F_2^\perp/F_2$
\begin{eqnarray*}\arraycolsep=2pt
\left\lbrace\begin{array}{lr}
\alpha w_{-2}-\beta w_1,&\alpha w_{-1}+\beta w_2,\\
\alpha w_{-2}+\beta w_1,&-\alpha w_{-1}+\beta w_2
\end{array}\right\rbrace
&\text{and}&\arraycolsep=2pt
\left\lbrace\begin{array}{lr}
\alpha w_{-3}-\beta w_2,&  \alpha w_{-2}+\beta w_3,\\
\alpha w_{-3}+\beta w_2,&  -\alpha w_{-2}+\beta w_3
\end{array}\right\rbrace
\end{eqnarray*}
depending on the two choices made previously, together with
$F_3=\langle\alpha w_{-2}+\beta w_1\rangle$ respectively $F_3=\langle\alpha w_{-3}+\beta w_2\rangle$.

The final domino labelled 1 puts us in Case (I) with the choice $F_4=F_3+\langle -\alpha w_{-1}+\beta w_2\rangle$ corresponding to $+$ and $F_4=F_3+\langle \alpha w_{-2}-\beta w_1\rangle$ to $-$; similar for the second choice.
Altogether, we obtain dense subsets of the four components are given by the following generic choices for each step in the complete full flags:

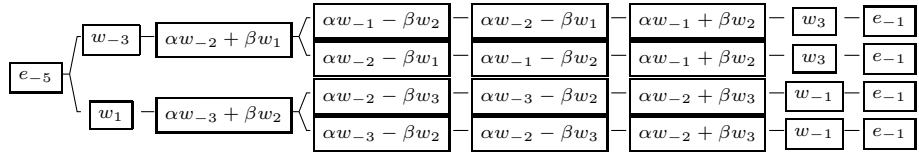
\begin{tikzpicture}[thick,font=\tiny,scale=.5]
\draw (0,0) node {\fbox{$e_{-5}$}};
\draw[thin] (.7,.1) -- +(.2,0);
\draw[thin] (.9,.1) -- +(.2,1) -- +(.3,1);
\draw[thin] (.9,.1) -- +(.2,-1) -- +(.3,-1);
\draw (2,1) node {\fbox{$w_{-3}$}};
\draw (2,-1) node {\fbox{$w_1$}};
\draw[thin] (2.8,1.1) -- +(.35,0);
\draw[thin] (2.8,-.9) -- +(.35,0);
\draw (5,1) node {\fbox{$\alpha w_{-2}+\beta w_1$}};
\draw (5,-1) node {\fbox{$\alpha w_{-3}+\beta w_2$}};
\draw[thin] (6.8,1.1) -- +(.2,0);
\draw[thin] (7,1.1) -- +(.2,.5) -- +(.3,.5);
\draw[thin] (7,1.1) -- +(.2,-.5) -- +(.3,-.5);
\draw[thin] (6.8,-.9) -- +(.2,0);
\draw[thin] (7,-.9) -- +(.2,.5) -- +(.3,.5);
\draw[thin] (7,-.9) -- +(.2,-.5) -- +(.3,-.5);
\draw (9.2,1.5) node {\fbox{$\alpha w_{-1}-\beta w_2$}};
\draw (9.2,.5) node {\fbox{$\alpha w_{-2}-\beta w_1$}};
\draw (9.2,-.5) node {\fbox{$\alpha w_{-2}-\beta w_3$}};
\draw (9.2,-1.5) node {\fbox{$\alpha w_{-3}-\beta w_2$}};
\draw[thin] (11.1,1.6) -- +(.4,0);
\draw[thin] (11.1,.6) -- +(.4,0);
\draw[thin] (11.1,-.4) -- +(.4,0);
\draw[thin] (11.1,-1.4) -- +(.4,0);
\draw (13.4,1.5) node {\fbox{$\alpha w_{-2}-\beta w_1$}};
\draw (13.4,.5) node {\fbox{$\alpha w_{-1}-\beta w_2$}};
\draw (13.4,-.5) node {\fbox{$\alpha w_{-3}-\beta w_2$}};
\draw (13.4,-1.5) node {\fbox{$\alpha w_{-2}-\beta w_3$}};
\draw[thin] (15.3,1.6) -- +(.4,0);
\draw[thin] (15.3,.6) -- +(.4,0);
\draw[thin] (15.3,-.4) -- +(.4,0);
\draw[thin] (15.3,-1.4) -- +(.4,0);
\draw (17.6,1.5) node {\fbox{$\alpha w_{-1}+\beta w_2$}};
\draw (17.6,.5) node {\fbox{$\alpha w_{-1}+\beta w_2$}};
\draw (17.6,-.5) node {\fbox{$\alpha w_{-2}+\beta w_3$}};
\draw (17.6,-1.5) node {\fbox{$\alpha w_{-2}+\beta w_3$}};
\draw[thin] (19.5,1.6) -- +(.4,0);
\draw[thin] (19.5,.6) -- +(.4,0);
\draw[thin] (19.5,-.4) -- +(.4,0);
\draw[thin] (19.5,-1.4) -- +(.4,0);
\draw (20.7,1.5) node {\fbox{$w_{3}$}};
\draw (20.7,.5) node {\fbox{$w_{3}$}};
\draw (20.7,-.5) node {\fbox{$w_{-1}$}};
\draw (20.7,-1.5) node {\fbox{$w_{-1}$}};
\draw[thin] (21.5,1.6) -- +(.4,0);
\draw[thin] (21.5,.6) -- +(.4,0);
\draw[thin] (21.5,-.4) -- +(.4,0);
\draw[thin] (21.5,-1.4) -- +(.4,0);
\draw (22.7,1.5) node {\fbox{$e_{-1}$}};
\draw (22.7,.5) node {\fbox{$e_{-1}$}};
\draw (22.7,-.5) node {\fbox{$e_{-1}$}};
\draw (22.7,-1.5) node {\fbox{$e_{-1}$}};
\end{tikzpicture}

Expressing these as flags in the original basis we get the isotropic flags $F_\bullet$ with $F_j=F_{j-1}+\langle x_j\rangle$, $1\leq j\leq 4$, where the tuple $(x_4,x_3,x_2,x_1)$ is given as follows (depending on the sign configuration):
\small
\begin{equation*}\arraycolsep=3pt
\begin{array}{l|llll}
--
& \alpha(\ii e_{-2}+e_3)-\beta(\ii e_{-3}-e_2),
& \alpha(\ii e_{-3}+e_2)-\beta(\ii e_{-4}-e_1),
& \ii e_{-4}+e_1,
&e_{-5}\\[0.1cm]
+-
& \alpha(\ii e_{-3}+e_2)+\beta(\ii e_{-4}-e_1),
& \alpha(\ii  e_{-3}+e_2)-\beta(\ii e_{-4}-e_1),
&\ii e_{-4}+e_1,
&e_{-5}\\[0.1cm]
-+
& \alpha(\ii e_{-3}+e_2)+\beta(\ii e_{-2}-e_3),
& \alpha(\ii  e_{-4}+e_1)+\beta(\ii e_{-3}-e_2),
&\ii e_{-4}-e_1,
&e_{-5}\\[0.1cm]
++
& \alpha(\ii e_{-4}+e_1)-\beta(\ii e_{-3}-e_2),
& \alpha(\ii e_{-4}+e_1)+\beta(\ii e_{-3}-e_2),
&\ii e_{-4}-e_1,
& e_{-5}
\end{array}
\end{equation*}
\normalsize
(We denoted the flags in reversed order, with $F_1$ on the right, so they fit with the cup diagrams and domino shapes.)

{\it The irreducible components attached to the two diagrams}
\parbox[c]{1.7cm}
{
\begin{tikzpicture}[thick,font=\tiny,scale=.6]
\draw (2.5,-4.5) rectangle +(.5,1) node at +(.25,.2) {$\pm$} node at +(.25,.55) {$1$} +(.5,.5) rectangle +(1.5,1) node at +(1,.75) {$2$} +(.5,0) rectangle +(1.5,.5) node at +(1,.25) {$3$} +(2.5,.5) rectangle +(1.5,1) node at +(2,.75) {4};
\end{tikzpicture}
}:

We have again $F_1=\langle e_{-5}\rangle$ and we pass to $\left\langle e_{-5}, {\ii}e_{-4},{\ii}e_{-3},{\ii}e_{-2}, e_1,e_2,e_3 \right\rangle / F_1$, which we identify with $V_{3,3}$. Case (IIb) uses the basis $B_{\rm new}$, for the generic choice, given by
\begin{equation*}
\left\lbrace
\begin{array}{ccc}
\beta \ii e_{-4}+\alpha e_1, & \beta \ii e_{-3}+\alpha e_2, & \beta \ii e_{-2} + \alpha e_3,\\
\alpha \ii e_{-4} + \beta e_1, & \alpha \ii e_{-3}+\beta e_2, & \alpha \ii e_{-2}+\beta e_3
\end{array} \right\rbrace
\end{equation*}
After applying then the Cases (IIa), (I) we obtain dense subsets of the components are given by the following generic choices of flags:
\small
\begin{eqnarray*}
\begin{array}{l|llll}
-
& -e_{-3}+\ii e_2,
& \ii e_{-4},
& \alpha \ii e_{-4}+\beta e_1,
&e_{-5}\\[0.1cm]
+
& -e_{-3}-\ii e_2,
& \ii e_{-4},
& \alpha \ii e_{-4}+\beta e_1,
&e_{-5}
\end{array}
\end{eqnarray*}
\normalsize

{\it The irreducible components attached to the two diagrams}
\parbox[c]{1.7cm}
{
\begin{tikzpicture}[thick,font=\tiny,scale=.6]
\draw (2.5,-6.5) rectangle +(.5,1) node at +(.25,.2) {$\pm$} node at +(.25,.55) {$1$} +(.5,.5) rectangle +(1.5,1) node at +(1,.75) {$2$} +(.5,0) rectangle +(1.5,.5) node at +(1,.25) {$4$} +(2.5,.5) rectangle +(1.5,1) node at +(2,.75) {3};
\end{tikzpicture}
}:

The dense subsets are given by the following generic flags:
\small
\begin{eqnarray*}
\begin{array}{l|llll}
-
& - e_{-3} + \ii e_2,
& e_{-4},
& e_{-5},
&\alpha e_{-5}+\beta e_1\\[0.1cm]
+
& - e_{-3} - \ii e_2,
& e_{-4},
& e_{-5},
&\alpha e_{-5}+\beta e_1
\end{array}
\end{eqnarray*}
\normalsize

We now compare the intersection behaviour of the irreducible components in both the algebraic and the topological Springer fibre for the case of an odd number of decorations. We look at the following diagrams respectively the corresponding cup diagrams (the components will be labelled by $A$, $B$, $C$ and $D$).
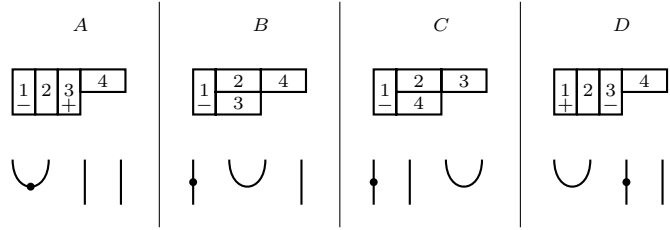
\begin{center}
\begin{tikzpicture}[thick,font=\tiny,scale=.6]
\foreach \x/\l in {1,2,3} {\draw (\x/2+2,-.5) rectangle +(.5,1) node at +(.25,.55) {\l};};
\draw (2.5,-.5) node at +(.25,.2) {$-$} node at +(1.25,.2) {$+$} +(2.5,.5) rectangle +(1.5,1) node at +(2,.75) {4};
\draw[thin] (5.75,2) -- +(0,-5);
\draw (4,1.5) node {$A$};

\begin{scope}[xshift=4cm]
\draw (2.5,-.5) rectangle +(.5,1) node at +(.25,.2) {$-$} node at +(.25,.55) {$1$} +(.5,.5) rectangle +(1.5,1) node at +(1,.75) {$2$} +(.5,0) rectangle +(1.5,.5) node at +(1,.25) {$3$} +(2.5,.5) rectangle +(1.5,1) node at +(2,.75) {$4$};
\draw[thin] (5.75,2) -- +(0,-5);
\draw (4,1.5) node {$B$};
\end{scope}

\begin{scope}[xshift=8cm]
\draw (2.5,-.5) rectangle +(.5,1) node at +(.25,.2) {$-$} node at +(.25,.55) {$1$} +(.5,.5) rectangle +(1.5,1) node at +(1,.75) {$2$} +(.5,0) rectangle +(1.5,.5) node at +(1,.25) {$4$} +(2.5,.5) rectangle +(1.5,1) node at +(2,.75) {$3$};
\draw[thin] (5.75,2) -- +(0,-5);
\draw (4,1.5) node {$C$};
\end{scope}

\begin{scope}[xshift=12cm]
\foreach \x/\l in {1,2,3} {\draw (\x/2+2,-.5) rectangle +(.5,1) node at +(.25,.55) {\l};};
\draw (2.5,-.5) node at +(.25,.2) {$+$} node at +(1.25,.2) {$-$} +(2.5,.5) rectangle +(1.5,1) node at +(2,.75) {$4$};
\draw (4,1.5) node {$D$};
\end{scope}

\begin{scope}[yshift=-1.5cm]
\draw (2.5,0) .. controls +(0,-.8) and +(0,-.8) .. +(.8,0);
\draw (4.1,0) -- +(0,-1);
\draw (4.9,0) -- +(0,-1);
\fill (2.9,-.6) circle(2.5pt);
\end{scope}

\begin{scope}[xshift=4cm,yshift=-1.5cm]
\draw (2.5,0) -- +(0,-1);
\draw (3.3,0) .. controls +(0,-.8) and +(0,-.8) .. +(.8,0);
\draw (4.9,0) -- +(0,-1);
\fill (2.5,-.5) circle(2.5pt);
\end{scope}

\begin{scope}[xshift=8cm,yshift=-1.5cm]
\draw (2.5,0) -- +(0,-1);
\draw (3.3,0) -- +(0,-1);
\draw (4.1,0) .. controls +(0,-.8) and +(0,-.8) .. +(.8,0);
\fill (2.5,-.5) circle(2.5pt);
\end{scope}

\begin{scope}[xshift=12cm,yshift=-1.5cm]
\draw (2.5,0) .. controls +(0,-.8) and +(0,-.8) .. +(.8,0);
\draw (4.1,0) -- +(0,-1);
\draw (4.9,0) -- +(0,-1);
\fill (4.1,-.5) circle(2.5pt);
\end{scope}
\end{tikzpicture}
\end{center}
Starting with the topological Springer fibre, we note that the definition of the components from \eqref{Sa} generalises and one easily checks that the only non-empty intersections (a single point in each case) are: $S_A \cap S_B$, $S_B \cap S_C$, $S_B \cap S_D$, and finally $S_C \cap S_D$.

For the algebraic Springer fibres, we summarise how the irreducible components look like for the cases above by giving the flags $F_\bullet$ in each case.
\begin{align*}
X_A &= \left\lbrace \left.\begin{array}{l}
F_1 = \left\langle e_{-5} \right\rangle, F_2 = F_1 \oplus \left\langle \mathbf{i}e_4-e_{-1}\right\rangle,\\
F_3 = F_2 \oplus \left\langle \alpha(\mathbf{i}e_{-4}+e_1) + \beta(\mathbf{i}e_{-3}-e_2) \right\rangle,\\
F_4=F_3 \oplus \left\langle \alpha(\mathbf{i}e_{-3}+e_2) + \beta(\mathbf{i}e_{-2}-e_3) \right\rangle\\
\end{array}\right| [\alpha:\beta] \in \mP^1(\mC)\right\rbrace\\
X_B &= \left\lbrace \left.\begin{array}{l}
F_1 = \left\langle e_{-5} \right\rangle, F_2 = F_1 \oplus \left\langle \alpha\mathbf{i}e_{-4} + \beta e_1\right\rangle,\\
F_3 = \left\langle e_{-5}, e_{-4},e_1 \right\rangle, F_4=F_3 \oplus \left\langle -e_{-3}+\mathbf{i}e_2 \right\rangle\\
\end{array}\right| [\alpha:\beta] \in \mP^1(\mC)\right\rbrace\\
X_C &= \left\lbrace \left.\begin{array}{l}
F_1 = \left\langle \alpha e_{-5}  + \beta e_1 \right\rangle, F_2 = \left\langle e_{-5},e_1\right\rangle,\\
F_3 = \left\langle e_{-5}, e_{-4},e_1 \right\rangle, \\
F_4=F_3 \oplus \left\langle -e_{-3}+\mathbf{i} e_2 \right\rangle\\
\end{array}\right| [\alpha:\beta] \in \mP^1(\mC)\right\rbrace\\
X_D &= \left\lbrace \left.\begin{array}{l}
F_1 = \left\langle e_{-5} \right\rangle, F_2 = F_1 \oplus \left\langle \mathbf{i}e_{-4}+e_{1}\right\rangle,\\
F_3 = F_2 \oplus \left\langle \alpha(\mathbf{i}e_{-3}+e_2) + \beta(\mathbf{i}e_{-4}-e_1) \right\rangle,\\
F_4= \left\langle e_{-5},e_{-4}, \mathbf{i}e_{-3}+e_2, e_1\right\rangle\\
\end{array}\right| [\alpha:\beta] \in \mP^1(\mC)\right\rbrace
\end{align*}
This description directly implies that $X_C \cap X_D = \emptyset$, while the other components have the same intersection behaviour as the topological version. In the algebraic case we thus have a Kleinian singularity of type ${\rm D}_4$, see \cite{Slodowy}.

\section{Cohomology of algebraic and topological Springer fibre}
\label{Section7}
We give a concrete description of the cohomology ring  the algebraic Springer fibres and prove Theorem \ref{Hintro}.
Let $n=2k$ and $G'=SO(2k,\mathbb{C})$ and $\mg'$ its Lie algebra with Cartan
subalgebra $\mathfrak{h}$. Let $P$ be a maximal parabolic of type $\rm A$
with Lie algebra $\p$ and Levi component $L$. The connected centre $S$ of
$L$ is one-dimensional. Let $\mathfrak{s}$ be its Lie algebra. Explicitly,
we realize $\mg'$ as a Lie algebra of complex $n\times n$-matrices $A$ that
are skew-symmetric
with respect to the anti-diagonal, i.e. $A_{i,j} = - A_{n-j+1,n-i+1}$. It has basis
\begin{align*}
\begin{array}{lclc}
L_{i,j}&=& E_{i,j} - E_{n-j+1,n-i+1},&1\leq i,j\leq k\\
X_{i,j}&=&E_{i,n-j+1} - E_{j,n-i+1},&1\leq i<j\leq k\\
Y_{i,j}&=&E_{n-j+1,i} - E_{n-i+1,j},&1\leq i<j\leq k
\end{array}
\end{align*}
where $E_{i,j}$ denotes the elementary matrix with $1$ at
position $(i,j)$ and zeros elsewhere. We choose $\mathfrak{h}$ to be the
Lie subalgebra given by diagonal matrices on which the Weyl group $W=W(D_k)$ acts
by $w.L_{i,i}=L_{w(i),w(i)}$ for $w\in S_k$ and $s_0.L_{1,1}=-L_{2,2}$,
$s_0.L_{2,2}=-L_{1,1}$, and $s_0.L_{i,i}=L_{i,i}$ for $i\not=1,2$.  The
$L_{i,j}$'s form a Levi subalgebra isomorphic to $\mathfrak{gl}(k,\mC)$ inside the subalgebra $\mp$ generated by the $L_{i,j}$'s and $X_{i,j}$. In this realisation, $\mathfrak{s}$ is the one-dimensional subspace spanned by $x_L=\sum_{i=1}^k L_{i,i}$. (Note $\mg'=\mathfrak{so}(W,\beta)$).

\subsection{Algebraic Springer fibre}
We consider the closed affine subvariety of $\mh \times \mh$ defined as
\begin{eqnarray}
\label{DefZ}
\mathcal{Z}_k &=& \{ (wx,x) \in \mh \times \mh \mid x \in \ms, w \in W \}.
\end{eqnarray}

Its coordinate ring $\mC[\mathcal{Z}_k]$ has, by \cite{KumarProcesi}, an $S(\ms^*)$-algebra structure coming from the projection onto $\ms$, where $S(\ms^*)=\mC[\ms]$ denotes the algebra of regular functions on $\ms$. To determine $\mC[\mathcal{Z}_k]$, and then compute $H(\cF'_N)$, we view $\mathcal{Z}_k$ as a closed subvariety of $\mh \times \ms$ and choose explicit coordinates: We identify $\mC[\ms]$ with $\mC[T]$ such that $T$
evaluated on $x_L$ is equal to $1$, and $\mC[\mh]$ with $\mC[X_1,\ldots,X_k]$ by identifying $X_i$ with the basis element dual to $L_{i,i}$.

For any $I\subset \{1,2,\ldots, k\}$ we denote $X_I=\prod_{i\in I}
X_i\in\mC[X_1,X_2,\ldots,X_k]$.
\begin{theorem}
\label{equivcohothm}
Let $N\in\mg'$ be nilpotent of Jordan type $(k,k)$.
\begin{enumerate}
\item The $S$-equivariant cohomology ring of $\mathcal{F}_N'$ equals
\begin{eqnarray}
\label{equivcoho}
H_S(\mathcal{F}_N')&\cong&
\begin{cases}
\dfrac{\mC[X_1,X_2,\ldots,X_k,T]}{\langle\{X_i^2-T^2,X_I-X_J\}\rangle}
&\text{if $k$ is even,}\\
\dfrac{\mC[X_1,X_2,\ldots,X_k,T]}{\langle\{X_i^2-T^2,X_{I}-X_{J}T\}\rangle}&\text{if
$k$ is odd,}
\end{cases}
\end{eqnarray}
where $1\leq i\leq k$ and $I,J\subset\{1,2,\ldots, k\}$ with $I\cap J=\emptyset$ such that $|I|=\frac{k}{2}=|J|$ for even $k$ and $|I|=\frac{k+1}{2}=|J|+1$ for odd $k$.
\item In particular,
\begin{eqnarray*}
\qquad &&H(\mathcal{F}_N') \cong \left\lbrace
\begin{array}{cl}
\mC[X_1,X_2,\ldots,X_k]/\langle\{X_i^2,X_I-X_J\}\rangle & \text{if $k$ is
even,}\\
\mC[X_1,X_2,\ldots,X_k]/\langle\{X_i^2,X_I\}\rangle & \text{if $k$ is odd}.
\end{array} \right.
\end{eqnarray*}
Explicit bases are given by the images of the $X_I$, $|I|\leq \frac{k-1}{2}$, in case $k$ is odd; by the $X_I$ for all $I$, $|I|<\frac{k}{2}$, in case $k$ is even and $k \notin I$; and by the $X_I$ for all $I$, $|I| \leq \frac{k}{2}$, in case $k$ is even and $k \in I$.
\end{enumerate}
\end{theorem}
\begin{proof}
We first show that $H_S(\mathcal{F}_N')\cong A_S(k)$, where $A_S(k)$ denotes the quotient rings on the right hand side of \eqref{equivcoho}, but assuming \cite[Theorem 1.2]{KumarProcesi}. The comomorphism of the inclusion $\mathcal{Z}_k\subset \mh \times \ms$ as closed subvariety induces a surjection
\begin{eqnarray} \label{surjection_onto_coordinate_ring}
\mC[X_1,\ldots,X_k,T] \longrightarrow \mC[\mathcal{Z}_k].
\end{eqnarray}
By \cite[Theorem 1.2]{KumarProcesi}, its kernel is contained in the kernel of the corresponding map to $H_S(\mathcal{F}_N')$ given by the equivariant Borel morphism on the first factor and the induced map $ \mathbb{C}[\mathcal{Z}_k] \longrightarrow
H_S(\mathcal{F}_N')$ is an isomorphism.

Since $\mathcal{Z}_k$ is a cone, to determine
the vanishing ideal of $\mathcal{Z}_k$, it is enough to look at homogeneous
polynomials vanishing on all points of the form $(w.x_L,x_L)$. The $W$-orbit of $x_L = \sum_{i=1}^k L_{i,i}$ is easy to determine, namely
\begin{equation}
W.x_L=\left\lbrace \sum_{i=1}^k \epsilon_i L_{i,i} \mid \epsilon_i =\pm 1,
\prod_{i=1}^k \epsilon_i = 1 \right\rbrace.
\end{equation}
It is a trivial calculation to check that the polynomials $X_i^2-T^2$ and
$X_I-X_J$, resp. $X_I-X_JT$, as defined above vanish on $W.x_L \times \{x_L\}\subset \ms \times \mh$. Thus the map
\eqref{surjection_onto_coordinate_ring} factors through $A_S(k)$.

Since $\cF_N'$ is equivariantly formal in the sense of \cite{GKMP}, $H_S(\mathcal{F}_N')$ is naturally a free $S(\ms^*)$-module of rank $|W/S_k| = 2^{k-1}$, which is the number of $S$-fixed points, \cite[Lemma 2.1]{KumarProcesi}. On the other hand, $A_S(k)$ is also a free module over $\mC[T]=S(\ms^*)$ with basis given by all $X_I$ for all $I$ such that  $|I| \leq \frac{k-1}{2}$ in case $k$ is odd, and $|I| <\frac{k}{2}$ or $|I|=\frac{k}{2}$ with $k \in I$ in case $k$ is even. The rank equals $\sum_{l=0}^{\frac{k-1}{2}}\binom{k}{l}=2^{k-1}$ if $k$ is odd and $\frac{1}{2}\binom{k}{k/2}+\sum_{l=0}^{\frac{k}{2}-1}\binom{k}{l}
=2^{k-1}$ if $k$ is even. Thus, \eqref{surjection_onto_coordinate_ring} induces an isomorphism between $A_S(k)$ and $H_S(\mathcal{F}_N')$. Using again equivariant formality we get $H(\mathcal{F}_N')\cong H_S(\mathcal{F}_N')\otimes_{\mC[T]}\mC$ and part (2) follows. It remains to verify the surjectivity assumption of \cite{KumarProcesi} which is Proposition \ref{surjectivity} below.
\end{proof}

\begin{prop}
\label{surjectivity}
Let $N\in\mg$ be a nilpotent endomorphism of $\mC^{2k}$ with Jordan type $(k,k)$. Then the canonical map  $H(\mathcal{F})\rightarrow H(\mathcal{F}_N)$ is surjective. The same holds for the  canonical map $H(\mathcal{F}')\rightarrow H(\mathcal{F}_N')$.
\end{prop}

\begin{proof}
We first claim that the top degree is completely contained in the image of the canonical map for any nilpotent of admissible Jordan type $(r,s)$ of type ${\rm D}_k$. Recall from Section \ref{Section2} that for $r=s=k$ even, the component group is trivial. Then the claim is a special case of a general result of Hotta and Springer \cite[Theorem 1.1, Corollary 1.3]{HottaSpringer}. In case of odd $r=s=k$, the component group is $\mZ/2\mZ$ and swaps between the two connected components of $\cF_N$, see \cite[Section 3]{vanLeeuwen}. By Theorem \ref{maintheorem} and Lemma~\ref{bijchi}, the total number of irreducible components equals twice the dimension of the corresponding irreducible representation of $W(D_k)$ attached to $N$ via the Springer correspondence. Hence, when we replace $\cF$ by the set $\cB$ of Borel subgroups and consider the top degree of $H(\cB_N)$ for the corresponding  fixed point variety $\cB_N$, it carries the structure of an irreducible $W(D_k)$-module. Hence the statement follows again by \cite[Corollary 1.3]{HottaSpringer}. Similarly for the remaining case $r\not=s$.

 Let now $N$ be of type $(k,k)$. We claim there is a filtration by subvarieties
\begin{eqnarray}
\label{filtration}
\cY_0\subset\cY_1\subset\cdots\subset \cY_{\frac{k-\epsilon}{2}}=\cF_N
\end{eqnarray}
(with $\epsilon={0}$ or $\epsilon=1$ depending on the parity of $k$) such that $\cY_j= F_{N_j}$, where the Jordan type of $N$ is $(k,k)$  if $j=\frac{k}{2}$
and $(2k-2j-1,2j+1)$ otherwise. To see this consider the labelling set, $\cC_j$, of $\op{Irr}(\cF_{N_j})$ by cup diagrams. The diagrams in $\cC_j$  have exactly one cup less than the ones in $\cC_{j+1}$ and can all be obtained (in a nonunique way) by replacing a cup by two rays in some diagram from $\cC_{j+1}$. Now recall the inductive construction of the components. Each cup corresponded to a $\mP^1(\mC)$-choice. If we interpret replacing a cup by two rays as picking the nongeneric point in $\mP^1(\mC)$ then this gives exactly the construction of the components of the smaller dimensional Springer fibres and thus constructs the required chain of subvarieties. By construction it is independent of the choices and the claim follows.
Now each $\cY_i$ is equidimensional of complex dimension $i$, so that it contributes only to $H^{2i}(\cF_N)$, the degree $2i$ in cohomology. Moreover, by the beginning of the proof the subspace spanned by the classes of the irreducible components is in the image of the canonical map  $H(\mathcal{F})\rightarrow H(\mathcal{F}_N)$. Hence it remains to show that all these (linearly independent) classes span $H(\cF_N)$. We have $\op{dim}H(\mathcal{F}_N)=2^{k}$, since $\cF_N$ has a cell partition with $2\cdot 2^{k-1}$ strata by \cite[Lemma 2.1]{KumarProcesi}.  On the other hand,
the parametrization of irreducible components from Theorem~ \ref{paramatrization} together with Lemmas \ref{Bijection} and \ref{bijchi} gives
\begin{eqnarray*}
\sum_{l=0}^{\frac{k-\epsilon}{2}}|\op{Irr}(\cY_l)|
&=&
\begin{cases}
2\sum_{l=0}^{\frac{k-1}{2}}\binom{k}{l}=\sum_{l=0}^{k}\binom{k}{l}=2^{k}
&\text{if $k$ is odd,}\\
\binom{k}{k/2}+2\sum_{l=0}^{\frac{k}{2}-1}\binom{k}{l}=\sum_{l=0}^{k}\binom{k}{l}
=2^{k}
&\text{if $k$ is even.}
\end{cases}
\end{eqnarray*}
Hence the canonical map is surjective.
\end{proof}

\begin{remark}{\rm
 The last steps of this proof follow more directly from a general result of Lusztig, \cite{Lusztiginduction}, providing an isomorphism  $H(\mathcal{F}'_N)\cong \op{Ind}^{W(D_k)}_{S_k} \mathbf{1}$ as $W(D_k)$-modules. The filtration \eqref{filtration} corresponds then to the decomposition \eqref{decomprefined} into irreducible summands. The two possible choices of embedding $S_k$ into $W(D_k)$ (using $s_i, i\not=0$ or $s_i, i\not=1$) give for even $k$ the two different representations labelled $(k/2),(k/2)$ in top degree.
}
\end{remark}

\begin{remark}
{\rm The action of $s_i$, $i>0$, on $H(\cF_N')$ is given by permuting the variables and $s_0$ acts by swapping $X_1$ and $-X_2$. It is induced from the natural left action on the first factor in \eqref{DefZ}. The different representations  $S^{(k/2),(k/2),\pm}$ arise in top degree depending on the two possible choices of $\mp$ of which we chose one.
}
\end{remark}

\subsection{Cohomology of topological Springer fibre}

We are now ready to prove Theorem \ref{Hintro} from the introduction:
 \begin{theorem}
 \label{H}
 Let $N\in\mg$ be a nilpotent endomorphism of $\mC^{2k}$ with Jordan type $(k,k)$. Then we have isomorphisms of graded rings
 \begin{eqnarray*}
 H(\widetilde{S})\;\cong\; H(\mathcal{F}_N)
 \end{eqnarray*}
\end{theorem}

\begin{proof}
Recall that via the isomorphism $\tau$ from Theorem~\ref{tau_is_iso} the ring $H(\widetilde{S})$ is identified with the centre of the algebra $\mathbb{K}_k$, which in turn is isomorphic to $H(\mathcal{F}_N)$ by \cite{ES}.
\end{proof}

\section{Jordan type $(k,k)$: category $\cO$ and $\mathcal{W}$-algebras}
\label{Section8}

Let now $N\in\mg$ be nilpotent of Jordan type $(k,k)$. In this case the $S$-fixed points $\cF_N^S$ of $\cF_N$ are easy to determine, namely all flags $F_\bullet=F_\bullet(I)$ where for each $1\leq j\leq k$ we have  $F_j=\langle w_l\mid l\in I\rangle$ where $I$ is a union $I=I_-\cup I_+$ of two sets of consecutive indices of the form
\begin{eqnarray*}
I_-=\{-k,-k+1,-k+2,\ldots -k+r+1\},&&  I_+=\{1,2,\ldots, s\}
\end{eqnarray*}
such that $|I|=r+s=j$. (These are obviously fixed points and there are altogether $2\cdot 2^{k-1}=2^k$ such choices which is the desired number by \cite[Lemma 2.1]{KumarProcesi}.)

In the following we will identify $S$-fixed points with combinatorial weights via the following easy observation.
\begin{lemma}
\label{fixedpointstoweights}
There is a bijection between fixed points and combinatorial weights of length $k$ sending $F_\bullet(I)\in\cF_N$ to the combinatorial weight $\la=\la(I)$ which has at vertex $i$ an $\up$ if $-i\in I$ and a $\down$ if $i\in I$.
\end{lemma}

The following generalizes \cite[Lemma 12]{SW}

\begin{prop}
\label{hurray}
Let $N$ be of Jordan type $(k,k)$. Then the following holds
\begin{enumerate}
\item  \label{fixedpointsone}
The fixed points $F_\bullet(I)\in\cF_N^S$ contained in an irreducible component $Y$ are precisely the combinatorial weights $\la=\la(I)$ such that $\la c$ is an oriented cup diagram, where $c$ denotes the cup diagram associated with $Y$ via Theorem~\ref{maintheorem} and Corollary \ref{signeddominoestocups}.
\item
 \label{fixedpointstwo}
 The fixed points $P(I)$ contained in the intersection $Y\cap Y'$ of two irreducible components are precisely the combinatorial weights $\la=\la(I)$ such that $c^*\la d$ is an oriented circle diagram, where $c$ (resp. $d$) denotes the cup diagram associated with $Y$ and $Y'$.
 \end{enumerate}
\end{prop}

From the constructions we directly obtain a refinement of Theorem~\ref{maintheorem}:
\begin{corollary}
Pairwise intersections of components are either empty or again iterated $\mP^1(\mC)$-bundles, namely of base type $(\mP^1(\mC))^c$, where $c$ is the number of circles in the corresponding circle diagram.
\end{corollary}

\begin{remark}
\label{basisvectors}
{\rm Each $P\in\cF_N^S$ defines a vector of $H(\cF_N)$ as follows: take the weight $\la$ of length $k$ identified with $P$. This defines the unique cup diagram $c=c(\la)$ on $k$ vertices such that $\la c$ is oriented and of degree $0$, see \cite{LS}. Using the filtration \eqref{filtration} it defines a subvariety of $\cF_N$ whose class defines a vector in cohomology. By the proof of Theorem \ref{surjectivity} these vectors form indeed a basis.
}
\end{remark}

\begin{ex}
\label{cooltable}
Consider $\cF_N$, where $N$ has Jordan type $(4,4)$. The $S$-fixed points in the pairwise intersections of irreducible components labelled by $\mB^{odd}_4$ are displayed in the following table (the case $\mB^{even}_4$ is similar).
\small
\begin{eqnarray*}
\begin{array}{c||c|c|c}
&
\begin{tabular}{c}
\begin{tikzpicture}[thick]
\draw (0,0) .. controls +(0,-.5) and +(0,-.5) .. +(.5,0);
\fill (0.25,-.36) circle(2.5pt);
\draw (1,0) .. controls +(0,-.5) and +(0,-.5) .. +(.5,0);
\end{tikzpicture}
\end{tabular}
&
\begin{tabular}{c}
\begin{tikzpicture}[thick]
\begin{scope}[xshift=6cm]
\draw (0,0) .. controls +(0,-1) and +(0,-1) .. +(1.5,0);
\fill (0.75,-.74) circle(2.5pt);
\draw (0.5,0) .. controls +(0,-.5) and +(0,-.5) .. +(.5,0);
\end{scope}
\end{tikzpicture}
\end{tabular}
&
\begin{tabular}{c}
\begin{tikzpicture}[thick]
\begin{scope}[xshift=3cm]
\draw (0,0) .. controls +(0,-.5) and +(0,-.5) .. +(.5,0);
\draw (1,0) .. controls +(0,-.5) and +(0,-.5) .. +(.5,0);
\fill (1.25,-.36) circle(2.5pt);
\end{scope}
\end{tikzpicture}
\end{tabular}
\\
\hline
\hline
\begin{tabular}{c}
\begin{tikzpicture}[thick]
\draw (0,0) .. controls +(0,-.5) and +(0,-.5) .. +(.5,0);
\fill (0.25,-.36) circle(2.5pt);
\draw (1,0) .. controls +(0,-.5) and +(0,-.5) .. +(.5,0);
\end{tikzpicture}
\end{tabular}
&\begin{tabular}{cc}
$\up\up\down\up$&
$\down\down\down\up$\\
$\up\up\up\down$&
$\down\down\up\down$
\end{tabular}
&\begin{tabular}{c}
$\up\up\down\up$\\
$\down\down\up\down$
\end{tabular}
&
\begin{tabular}{c}
\text{empty intersection}
\end{tabular}
\\
\hline
\begin{tabular}{c}
\\
\begin{tikzpicture}[thick]
\begin{scope}[xshift=6cm]
\draw (0,0) .. controls +(0,-1) and +(0,-1) .. +(1.5,0);
\fill (0.75,-.74) circle(2.5pt);
\draw (0.5,0) .. controls +(0,-.5) and +(0,-.5) .. +(.5,0);
\end{scope}
\end{tikzpicture}
\end{tabular}
&
\begin{tabular}{c}
$\up\up\down\up$\\
$\down\down\up\down$
\end{tabular}
&
\begin{tabular}{cc}
$\up\down\up\up$&$\down\up\down\down$\\
$\up\down\up\up$&$\down\up\down\down$
\end{tabular}
&
\begin{tabular}{c}
$\up\down\up\up$\\
$\down\up\down\down$
\end{tabular}
\\
\hline
\begin{tabular}{c}
\begin{tikzpicture}[thick]
\begin{scope}[xshift=3cm]
\draw (0,0) .. controls +(0,-.5) and +(0,-.5) .. +(.5,0);
\draw (1,0) .. controls +(0,-.5) and +(0,-.5) .. +(.5,0);
\fill (1.25,-.36) circle(2.5pt);
\end{scope}
\end{tikzpicture}
\end{tabular}
&
\begin{tabular}{c}
\text{empty intersection}
\end{tabular}
&\begin{tabular}{c}
$\up\down\up\up$\\
$\down\up\down\down$
\end{tabular}&
\begin{tabular}{cc}
$\down\up\up\up$&$\up\down\up\up$\\
$\down\up\down\down$&$\up\down\down\down$
\end{tabular}
\end{array}
\end{eqnarray*}
\normalsize
\end{ex}

\begin{proof} [Proof of Proposition \ref{hurray}] For part \eqref{fixedpointsone} we assume first that $k$ is even. Moreover assume that the signed domino tableau attached to $Y$ has only one cluster. By construction of the component $Y$ (Case IIIa) in \eqref{Cases}) we choose the basis  $y_{-k}, \ldots, y_{-1}, y_1,\ldots, y_{k}$, where $y_{-j}=\alpha w_{-j}+(-1)^{j+1}\beta w_{k+1-j}$ and
$y_{j}=(-1)^{j+1}\alpha w_{j-k-1}+\beta w_{j}$ for $1\leq j\leq k$ and have $F_1=\langle y_{1}\rangle$. To describe $F_2$ we apply Case IIb) of \eqref{Cases} and choose a basis of $F_1^\perp/F_1$ of the form
$z_{k-1}, \ldots, z_{-1}, z_1,\ldots, z_{k-1}$,
where
$z_{-j}=(-1)^{j+1}y_{j}+y_{k-j}$ and $z_j=(-1)^{j+1}y_{j-k}-y_{j}$ for $1\leq j\leq k-1$ such that $F_2=F_1+\langle \delta z_{-(k-1)}+\gamma z_1\rangle$ for $\delta,\gamma\in\mC$. Now $\delta z_{-(k-1)}+\gamma z_1=(\gamma+\delta)y_{-(k-1)}+(\delta-\gamma)y_1$, thus choosing $\gamma=\delta$ gives us $y_{-(k-1)}$ and $\gamma=-\delta$ gives us
$y_1$. Choosing now $(\alpha,\beta)\in\{(0,1),(1,0)\}$ gives us finally the required $w$'s. These arguments can be repeated to give the claim for the standard tableau part. For the remaining vertical domino we have  $F_k=F_{k-1}+\langle b+b'\rangle$ if the sign is $-$
and $F_k=F_{k-1}+\langle b-b' \rangle$ if the sign is $+$, where
for $b=\gamma z_{-k/2}+\delta z_{k/2}$ and $b'=\delta z_{-k/2}+\gamma z_{k/2}$, hence $b+b'=(\gamma+\delta) z_{-k/2}+ (\gamma-\delta) z_{k/2}$. and $b-b'=(\gamma-\delta) z_{-k/2}+ (\gamma+\delta) z_{k/2}$. Hence choosing $\gamma=\delta$ or $\gamma=-\delta$ and $\alpha, \beta$ as above shows part \eqref{fixedpointsone} in case there is only one cluster. In case there is more than one cluster, we start as above with the first cluster and then repeat the same arguments for the space $W/F_k$ with the basis given by the remaining $w$'s (i.e. in case the first cluster had shape $(s,s)$ these are the $w_i$'s where $-k+s/2\leq i\leq -1$ and $s/2+1\leq i\leq k$ if the sign was $+$,  and  $-k+s/2+1\leq i\leq -1$ and $s/2\leq i\leq k$ if the sign was $-$) etc. This settles the case for even $k$.
For $k$ odd we argue as for $k+1$ in the even case, but remove afterwards the rightmost vertical domino which destroys half of the fixed points and the claim follows again.

Part \eqref{fixedpointstwo} follows then directly from part \eqref{fixedpointsone} and the definition of oriented circle diagrams.
\end{proof}

\begin{remark}
\label{subtle}
Note that Proposition~\ref{hurray} is in general wrong for Jordan types $(r,s)$ where $r\not=s$. Consider Section \ref{explicitexample} for a counterexample where the intersections of the algebraic Springer are not given by the cup diagram combinatorics.
\end{remark}

\subsection{Connections with the algebra $\mK_k$ and category $\cO$}

As a consequence we obtain Theorem~\ref{boringtheorem} from the introduction:

\begin{theorem}
\label{vectorspaceisoeasy}
Let $N\in\mg$ be nilpotent of Jordan type $(k,k)$. Then there exists an isomorphism of graded vector spaces
\begin{eqnarray}
\mathbb{K}_k=\bigoplus_{(Y_1,Y_2)\in\op{Irr}(\cF_N)\times\op{Irr}
(\cF_N)}H(Y_1\cap Y_2)\langle d(a(Y_1),a(Y_2))\rangle
\end{eqnarray}
where $a(Y_i)$, $i=1,2$ denotes the cup diagram associated with $Y_i$.
\end{theorem}

\begin{proof} By Theorem~\ref{equivcohothm} we have an isomorphism of vector spaces $H(\cF_N)\cong\mC^{|\cF_N^S|}$ with basis naturally labelled by the $S$-fixed points. Hence there is at least an isomorphism of vector spaces by Proposition~\ref{hurray} and the definition of $\mK_k$. The filtration \eqref{filtration} determines the degree of each basis vector and the theorem follows easily from the definition of the grading on $\mK_k$.
\end{proof}

\begin{remark} Since circle diagrams encode the $(D_k,A_{k-1})$-Kazhdan-Lusztig combinatorics, Theorem \ref{vectorspaceisoeasy} implies that the combinatorics of intersection of components in $\cF_N$ for the equal length two row case is controlled by Kazhdan-Lusztig polynomials.
\end{remark}

\subsection{Connections with $\mathcal{W}$-algebras and category $\cO$}

Let $e\in\mg$ be nilpotent of Jordan type $2^k$ for even $k$ and consider the associated $\mathcal{W}$-algebra $\mathcal{W}(e)$ of finite type as introduced originally by Premet, \cite{Premet}. Let $\mathcal{F}in_0(e)$ be the abelian category of finite dimensional $\mathcal{W}(e)$-modules with trivial central character. By a result of Brown, \cite{Brownrectangular}, the simple objects in $\mathcal{F}in_0(e)$  are naturally indexed by the set $\cP_k$ of s-tables of shape $2^k$.

An {\it s-table} is a Young diagram of shape $(2^k)$ with a filling with the numbers $\pm 1, \ldots, \pm k$ such that the columns are strictly decreasing from top to bottom and the rows are strictly increasing from left to right. Moreover the filling needs to be skew-symmetric in the sense that the $i$th entry in the first column equals the negative of the $(k-i+1)$th entry in the second column.

For an $s$-table $P$ we denote by $\cI=\cI(P)$ the set of numbers in the second column and encode the associated type ${\rm D}_k$ weight as in \cite{LS} as the combinatorial weight with the $i$-th symbol an $\up$ if $i\in \cI$ and a $\down$ if $-i\in \cI$, that is the weight $\la=\la(-\cI(P))$ in the notation of Lemma \ref{fixedpointstoweights}. To this weight $\la$ we can assign the associated cup diagram $c(\la)$ as in Remark \ref{basisvectors}. Hence we get an assignment $\cI\mapsto \op{Cup}(\cI)$ sending an $s$-table to a cup diagram. For instance, $\cI=\{1,2,3,4\}$, $\cI=\{1,-2,3,4\}$, respectively $\cI=\{-1,2,-3,4\}$ define the cup diagrams in \eqref{introcup}.
It is easy to check for a weight $\la$ that $c(\la)\in\mB_k$ if and only if strictly to the left of each $\down$ in $\la$ there is at least one more $\up$ than $\down$.

\begin{lemma}
\label{skewpyramids}
The assignment $\cP(\cI)\mapsto \op{Cup}(\cI)$ defines a bijection $\cP_k\oneone\mathbb{B}_k$.
\end{lemma}

\begin{proof} Let $P\in\cP_k$. If there is a $\down$ in $\op{Cup}(P)$ at position $i$ then $P$ has $i$ in column $1$ and fewer numbers bigger than $i$ in column $1$ than in $2$. Hence the map is well-defined and obviously injective. To see it is surjective we have to show that for any $D\in\mB_k$ the associated filling $P$ defines an $s$-table. This is clear if $D$ has only dotted cups, as  then $I=\{1,2,\ldots, k\}$.  Otherwise choose an undotted cup $c$ connecting vertices $i$ and $i+1$. Increase the numbers below $-(i+1)$ in the first column by $2$ and decrease the numbers above $i+1$ in the second column by $2$ and then remove $\pm i$ and $\pm(i+1)$ from the filling. The result is the filling associated to $D$ with $c$ removed. The claim follows then by induction observing that the base holds for even $k$.
\end{proof}

Let $\mp$ and $\mp'$ be the two parabolic subalgebras in $\mg$ with Weyl group generated by $s_i$, $i\not=0$ and $i\not=1$ respectively. Let $\cO_0^\mp(\mg)$ and $\cO_0^{\mp'}(\mg)$ be the corresponding principal blocks of parabolic category $\cO$. By \cite{ES}, the (isomorphism classes) of indecomposable projective-injective modules are naturally labelled by $\mB^{\rm even}_k$ respectively $\mB^{\rm odd}_k$ and have all the same Loewy-length namely $2k+1$. (Cup diagrams on $k$ vertices label all indecomposable projective modules. The number of cups encodes, as in \cite[Proposition 58]{FSS}, the GK-dimension of the corresponding simple quotient, and induces a filtration categorifying the decomposition \eqref{decomprefined} by taking successive subquotient categories.) Now Conjecture~\ref{bigconj} implies with \cite[Theorem 10.1]{Stproj}

\begin{corollary}
Let $\cO:=\cO^\mp(\mg)_0\oplus\cO^{\mp'}(\mg)_0$ and consider $P := \oplus P(\op{Cup}(\cI))$ where $I$ runs over all $s$-tables. Then ${\rm Hom}_\cO (P,?): \cO \longrightarrow {\rm mod}-{\rm End}_{\cO}(P)$ is fully faithful on projective objects. Assuming Conjecture \ref{bigconj} this defines via the identification from  Lemma \ref{skewpyramids} a functor ${\rm Hom}_\cO (P,?): \cO \longrightarrow \mathcal{F}in_0(e)$, again fully faithful on projective objects. In particular $\mathcal{F}in_0(e)$ is a quotient category of $\cO$.
\end{corollary}

\begin{remark}
The indecomposable projective objects in $\mathcal{F}in_0(e)$ are also in bijection with orbital varieties of $G$, since $k$ is even. Identifying the two equivalent summands of $\cO$ corresponds then to passing to orbital varieties for the adjoint group. We normalized our setup using cup diagrams with an even and odd number of dots such that they describe precisely the fibres of the Spaltenstein-Steinberg map, see \cite[Theorem 3]{McGovern}.
\end{remark}

\bibliographystyle{alpha}
\bibliography{References_Springer}

\end{document}